\documentclass{article}%
\usepackage{amsmath}
\usepackage{amsfonts}
\usepackage{amssymb}
\usepackage{graphicx}%
\numberwithin{equation}{section}
\newtheorem{theorem}{Theorem}[section]

\newtheorem{claim}[theorem]{Claim}

\newtheorem{corollary}[theorem]{Corollary}

\newtheorem{definition}[theorem]{Definition}

\newtheorem{lemma}[theorem]{Lemma}

\newtheorem{proposition}[theorem]{Proposition}
\newtheorem{remark}[theorem]{Remark}

\newenvironment{proof}[1][Proof]{\noindent\textbf{#1.} }{\ \rule{0.5em}{0.5em}}
\begin{document}

\title{Basic properties of nonsmooth H\"{o}rmander's vector fields and Poincar\'{e}'s
inequality\thanks{\textbf{2000 AMS\ Classification}: Primary 53C17; Secondary
46E35, 26D10. \textbf{Keywords}: nonsmooth H\"{o}rmander's vector fields,
Poincar\'{e} inequality}}
\author{Marco Bramanti, Luca Brandolini, Marco Pedroni}
\maketitle

\begin{abstract}
We consider a family of vector fields
\[
X_{i}=\sum_{j=1}^{p}b_{ij}\left(  x\right)  \partial_{x_{j}}%
\]
($i=1,2,...,n;$ $n<p$) defined in some bounded domain $\Omega\subset
\mathbb{R}^{p}$ and assume that the $X_{i}$'s satisfy H\"{o}rmander's rank
condition of some step $r$ in $\Omega,$ and $b_{ij}\in C^{r-1}\left(
\overline{\Omega}\right)  .$ We extend to this nonsmooth context some results
which are well-known for smooth H\"{o}rmander's vector fields, namely: some
basic properties of the distance induced by the vector fields, the doubling
condition, Chow's connectivity theorem, and, under the stronger assumption
$b_{ij}\in C^{r-1,1}\left(  \Omega\right)  ,$ Poincar\'{e}'s inequality. By
known results, these facts also imply a Sobolev embedding. All these tools
allow to draw some consequences about second order differential operators
modeled on these nonsmooth H\"{o}rmander's vector fields:%
\[
\sum_{i,j=1}^{n}X_{i}^{\ast}\left(  a_{ij}\left(  x\right)  X_{j}\right)
\]
where $\left\{  a_{ij}\right\}  $ is a uniformly elliptic matrix of
$L^{\infty}\left(  \Omega\right)  $ functions.

\end{abstract}
\tableofcontents

\section{Introduction}

\subsection{The problem}

Let us consider a family of real valued vector fields
\[
X_{i}=\sum_{j=1}^{p}b_{ij}\left(  x\right)  \partial_{x_{j}}%
\]
($i=0,1,2,...,n;$ $n<p$) defined in some domain $\Omega\subset\mathbb{R}^{p}$.
For the moment, we do not specify the regularity of the $b_{ij}$'s, but just
assume that these coefficients have all the derivatives involved in the
formulae which we will write. Let us define the \textit{commutator} of two
vector fields:
\[
\lbrack X,Y]=XY-YX.
\]
We also call \textit{commutator of length} $r$ an iterated commutator of the
kind:%
\[
\left[  X_{i_{1}},\left[  X_{i_{2}},...\left[  X_{i_{r-1}},X_{i_{r}}\right]
...\right]  \right]  .
\]

One says that the system of vector fields $X_{0},\ldots,X_{n}$ satisfies
\textit{H\"{o}rmander's condition of step} $r$ in $\Omega$ if the vector space
spanned by the vector fields $X_{i}$'s and their commutators of length up to
$r$ is the whole $\mathbb{R}^{p}$ at each point of $\Omega$. A famous theorem
by H\"{o}rmander, 1967, \cite{H}, states that if the $X_{i}$'s are real
valued, have $C^{\infty}$ coefficients, and satisfy H\"{o}rmander's condition
of some step $r$ in $\Omega,$ then the linear second order differential
operator:%
\begin{equation}
L=\sum_{i=1}^{n}X_{i}^{2}+X_{0}. \label{H}%
\end{equation}
is hypoelliptic in $\Omega.$ This means, by definition, that whenever the
equation $Lu=f$ is satisfied in $\Omega$ in distributional sense, then for any
open subset $A\subset\Omega,$
\[
f\in C^{\infty}\left(  A\right)  \Longrightarrow u\in C^{\infty}\left(
A\right)  .
\]

Another consequence of H\"{o}rmander's condition, which is known since the
1930's, is the \textit{connectivity property}: any two points of $\Omega$ can
be joined by a sequence of arcs of integral lines of the vector fields
(\textquotedblleft Chow's theorem\textquotedblright, 1939 \cite{Ch}; see also
Rashevski, 1938 \cite{R}). This fact suggests that one can define a distance
induced by the vector fields, as the infimum of the lengths of the
\textquotedblleft admissible lines\textquotedblright\ (tangent at every point
to some linear combination of the $X_{i}$'s) connecting two points.

Starting from H\"{o}rmander's theorem, many other important properties have
been proved, in the last 40 years, both regarding systems of H\"{o}rmander's
vector fields and the metric they induce, and regarding second order
differential operators structured on H\"{o}rmander's vector fields, like
(\ref{H}). In the first group of results, we recall:

\begin{itemize}
\item the doubling property of the Lebesgue measure with respect to the metric
balls (Nagel-Stein-Wainger \cite{NSW});

\item Poincar\'{e}'s inequality with respect to the vector fields (Jerison
\cite{J}).
\end{itemize}

In the second group of results, we recall:

\begin{itemize}
\item the \textquotedblleft subelliptic estimates\textquotedblright\ of
$H^{\varepsilon,2}$ norm of $u$ in terms of $L^{2}$ norms of $Lu$ and $u$
(Kohn \cite{K});

\item the \textquotedblleft$W^{2,p}$ estimates\textquotedblright, involving
second order derivatives with respect to the vector fields $X_{i},$ in terms
of $L^{p}$ norms of $Lu$ and $u$ (Folland \cite{F1}, Rothschild-Stein
\cite{RS});

\item estimates on the fundamental solution of $L$ or $\partial_{t}-L$ (again
\cite{NSW}, Sanchez-Calle \cite{SC}, Jerison-Sanchez-Calle \cite{JS},
Fefferman-Sanchez-Calle \cite{FS}).
\end{itemize}

Now, it is fairly natural to ask whether part of the previous theory still
holds for a family of vector fields having only a partial regularity. Here are
just a few facts which suggest this question:

(i) to check H\"{o}rmander's condition of step $r$ one has to compute
derivatives of order up to $r-1$ of the coefficients of vector fields;

(ii) the definition of distance induced by a system of vector fields makes
sense as soon as the vector fields are, say, locally Lipschitz continuous (in
this general case, however, the distance of two points could be infinite, and
proving connectivity, studying the volume of metric balls, proving the
doubling condition and so on are open problems);

(iii) apart from H\"{o}rmander's theorem about hypoellipticity, which is
meaningful in the context of operators with $C^{\infty}$ coefficients, several
important results about second order differential operators built on
H\"{o}rmander's vector fields are stated in a form which makes sense also for
vector fields with a limited regularity (e.g., Poincar\'{e} inequality, a
priori estimates on $X_{i}X_{j}u$ in $L^{p}$ or H\"{o}lder spaces,
\textit{etc.}).

\subsection{Previous results}

Several authors have studied the subject of nonsmooth H\"{o}rmander's vector
fields, approaching the problem under different points of view. We give a
brief account of the main lines of research, without any attempt to quote
neither all the papers nor all the authors who have given contributions in
these directions.

1. \textit{Nonsmooth diagonal vector fields}%
\[
X_{i}=a_{i}\left(  x\right)  \partial_{x_{i}}.
\]
Here the typical assumptions are the following:

\begin{itemize}
\item the number of vector fields equals the dimension of the space;

\item the $i$-th vector field involves only the derivative in the $i$-th direction;

\item the coefficients $a_{i}$ can vanish, so the operator $\sum X_{i}^{2}$
can be degenerate;

\item the coefficients can be nonsmooth (typically, they are Lipschitz
continuous, and satisfy some other structural assumptions).
\end{itemize}

These operators have been first studied in several papers of the 1980's by
Franchi and Lanconelli, see \cite{FL1}, \cite{FL2}, \cite{FL3}, \cite{FL4},
\cite{FL5}, and also the more recent paper \cite{Fr}. A recent work by
Sawyer-Wheeden \cite{SW} deals extensively with these operators. Clearly, the
particular structure of these vector fields allows to use ad-hoc techniques
which cannot be employed in the general (non-diagonal) case.

2. \textit{\textquotedblleft Axiomatic theories\textquotedblright\ of general
Lipschitz vector fields}, and the metrics induced by them. This means that,
for instance, one assumes axiomatically the validity of a connectivity
theorem, a doubling property for the metric balls, a Poincar\'{e}'s inequality
for the \textquotedblleft gradient\textquotedblright\ defined by the system of
vector fields, and proves as a consequence other interesting properties of the
metric or of second order PDE's structured on the vector fields. A good deal
of papers have been written in this spirit; we just quote some of the Authors
and some of the papers on this subject, which are a good starting point for
further bibliographic references: Capogna, Danielli, Franchi, Gallot,
Garofalo, Gutierrez, Lanconelli, Morbidelli, Nhieu, Serapioni, Serra Cassano,
Wheeden; see \cite{CDG}, \cite{DGN}, \cite{FGaW}, \cite{FGuW}, \cite{FSSS},
\cite{GN1}, \cite{GN2}, \cite{LM}; see also the already quoted paper \cite{SW}
and the one by Hajlasz-Koskela \cite{HK}.

3. \textit{Nonsmooth vector fields of step two. }The two papers by
Montanari-Morbidelli \cite{MM1}, \cite{MM2} consider vector fields with
Lipschitz continuous coefficients, satisfying H\"{o}rmander's condition of
step two, plus some other structural condition. The goal of these papers is to
prove Poincar\'{e}'s and Sobolev' type inequalities for these vector fields.
Rampazzo-Sussman \cite{RaSu} adopt a different point of view; here the
assumptions are very weak, considering Lipschitz vector fields satisfying (at
step 2) a \textquotedblleft set valued Lie bracket condition\textquotedblright%
\ previously introduced by the same Authors in the context of control theory;
the Authors then establish some basic properties of this weak
\textquotedblleft commutator\textquotedblright.

4. \textit{\textquotedblleft Nonlinear vector fields\textquotedblright.} In
the context of Levi-type equations, several Authors have considered vector
fields with $C^{1,\alpha}$ coefficients, having a particular structure, and
satisfying H\"{o}rmander's condition of step 2; the final goal is to get a
regularity theory for certain classes of nonlinear equations, which can be
written as sum of squares of \textquotedblleft nonlinear vector
fields\textquotedblright, i.e. vector fields whose coefficients depend on the
first order derivatives of the solution. Assuming that the solution is
$C^{2,\alpha},$ these vector fields become $C^{1,\alpha},$ and a good
regularity theory for the corresponding linear equation then implies, by a
bootstrap argument, the smoothness of the solution. Some results of this kind
also involve higher steps. We refer to the papers by Citti \cite{C},
Citti-Montanari \cite{CM1}, \cite{CM}, \cite{CM2}, Montanari \cite{M1},
\cite{M2}, Citti-Lanconelli-Montanari \cite{CLM}, Montanari-Lanconelli
\cite{ML}, Montanari-Lascialfari \cite{MLas}, and references therein.

We also quote some papers by Vodopyanov-Karmanova (see \cite{KV}, \cite{V} and
references therein), where the Authors study the geometry of nonsmooth vector
fields, and in particular establish a connectivity theorem assuming that the
highest order commutators have $C^{1,\alpha}$ coefficients.

\subsection{Aim of the present research and main results}

Summarizing the discussion of the last paragraph, most of the previous results
about nonsmooth vector fields \textit{either }hold only for the step 2 case,
\textit{or} for vector fields with a particular structure, \textit{or }assume
axiomatically some important properties of the metric induced by the vector
fields themselves.

Our aim is to develop a theory for any system of vector fields satisfying
\textquotedblleft H\"{o}rmander's condition\textquotedblright,\ at any step,
requiring that the coefficients of the vector fields possess the minimal
number of derivatives necessary to check H\"{o}rmander's condition.

More precisely, our assumptions consist in asking that, for some integer
$r\geq2,$ the vector fields $X_{1},...,X_{n}$ possess $C^{r-1,1}\left(
\Omega\right)  $ coefficients, and satisfy H\"{o}rmander's condition at step
$r$. Under these assumptions, we prove some basic properties of the distance
induced by the $X_{i}$'s (see Propositions \ref{Prop fefferman phong} and
\ref{Prop Fefferman Phong d1}, Theorem \ref{Thm equivalent distances d d1}),
the doubling condition (Theorem \ref{Thm doubling nonsmooth} and Theorem
\ref{Thm equivalent distances d d1}), Chow's connectivity theorem (Theorem
\ref{Thm Chow}), and Poincar\'{e}'s inequality (Theorem \ref{Thm Poincare}).
Actually, most of our results (with the relevant exception of Poincar\'{e}'s
inequality) hold under the weaker assumption that $X_{i}\in C^{r-1}\left(
\Omega\right)  .$ We will make precise our assumptions later.

These results constitute a first set of tools regarding \textquotedblleft
nonsmooth H\"{o}rmander's vector fields\textquotedblright\ which is enough to
draw some interesting consequences. For instance, by known results, these
facts also imply a Sobolev embedding (Theorem \ref{Thm Sobolev met Poincare}).
All these tools then allow to prove some properties of solutions to second
order differential equations of the kind:%
\[
\sum_{i,j=1}^{n}X_{i}^{\ast}\left(  a_{ij}\left(  x\right)  X_{j}u\right)  =0
\]
where $\left\{  a_{ij}\right\}  $ is a uniformly elliptic matrix of
$L^{\infty}\left(  \Omega\right)  $ functions, and $X_{i}$ are nonsmooth
H\"{o}rmander's vector fields (see Theorem \ref{Thm Moser}). Many other
problems in this direction remain open, which we hope to address in a future.

Another feature of our work which we would like to point out here, is that we
take into account explicitly the possibility of \textit{weighted }vector
fields. To explain this point, we recall that H\"{o}rmander's theorem refers
to an operator of the kind%
\[
\sum_{i=1}^{n}X_{i}^{2}+X_{0}%
\]
and that, both in dealing with the metric induced by the $X_{i}$'s and in
dealing on a priori estimates for second order operators, the field $X_{0}$
has \textquotedblleft weight two\textquotedblright, compared with $X_{1}%
,X_{2},...,X_{n}$ which have \textquotedblleft weight one\textquotedblright:
in some sense, $X_{0}$ plays the role of a second order derivative, in a
similar way as the time derivative enters the heat equation. In
\S \ref{section subelliptic metric}, making precise our assumptions and
notation, we will explain how we take into account this fact.

\subsection{Logical structure of the paper}

The paper is sequenced into three parts. In the first part, consists in
\S \S \ 2-3, some results about nonsmooth vector fields are deduced from
analogous results which are known to hold for smooth vector fields. Namely, in
\S \ref{section subelliptic metric}, after introducing notation and making
precise our assumptions, we prove a first basic inequality relating the
subelliptic metric $d$ induced by nonsmooth vector fields and the Euclidean
one. Then, in \S \ref{section approximating balls} we introduce, in a standard
way, a family of smooth vector fields which approximate the nonsmooth ones in
the neighborhood of a point (by Taylor's expansion of their coefficients), and
prove that the metric balls of the distances induced by smooth and nonsmooth
vector fields are comparable. In view of the doubling condition which holds in
the smooth case, this implies the doubling condition also for the metric $d$.
This approximation technique for nonsmooth vector fields has been already used
by several authors, see for instance \cite{CLM}, \cite{CM2}.

In the second part, consisting in \S \S \ref{section exp and quasiexp}%
-\ref{section connectivity}, we study extensively exponential and
\textquotedblleft quasiexponential\textquotedblright\ maps built with our
vector fields. In contrast with the style of the first part, here we do not
use any approximation argument, but have to work directly with nonsmooth
vector fields. The key result in \S \ref{section exp and quasiexp} is Theorem
\ref{Thm map C_el}, which says that the quasiexponential maps built composing
in a suitable way the exponentials of our basic vector fields are approximated
by the exponentials of commutators. As a consequence of this result, in
\S \ref{section connectivity} we can prove Theorem \ref{Thm diffeomorfism},
which states that the set of points which are reachable moving along integral
curves of the vector fields from a fixed point, is diffeomorphic to a
neighborhood of the origin. The two theorems we have just quoted are perhaps
the technical core of the paper, and the possibility of proving them under our
mild smoothness assumptions relies on a careful study of regularity matters
related to exponential and quasiexponential maps. These two theorems have
several interesting consequences. The first is a version of Chow's
connectivity theorem (Theorem \ref{Thm Chow}). A second one is the proof of
the local equivalence of the \textquotedblleft control
distance\textquotedblright\ $d_{1}$ attached to our system of vector fields
(and defined without reference to the commutators) with $d$, as in the smooth
case (Theorem \ref{Thm equivalent distances d d1}). In the \textquotedblleft
axiomatic theories\textquotedblright\ of vector fields with Lipschitz
continuous coefficients, $d_{1}$ is the natural distance that can be defined,
while $d$ (which involves commutators) is generally meaningless. Therefore,
the equivalence of $d$ and $d_{1}$ is a crucial point, because it allows to
link the abstract results of axiomatic theories with our more concrete setting
(this fact will be actually useful in \S 6). A third consequence is the
possibility of controlling the increment of a function by means of its
gradient with respect to the vector fields $X_{i}$ (Theorem \ref{Thm Lagrange}).

The third part of the paper consisting in \S \S \ref{Section lifting}%
-\ref{section consequences}. Here we prove Poincar\'{e}'s inequality and draw
some consequences from the whole theory developed so far. In this part we set
$X_{0}\equiv0$ (as is natural in the context of Poincar\'{e}-type
inequalities) and strengthen our assumptions on the $X_{i}$'s, asking them to
belong to $C^{r-1,1},$ instead of $C^{r-1}$ (recall that $r$ is the maximum
length of commutators required to check H\"{o}rmander's condition). Part 3 is
in some sense a mix of the techniques employed in Part 1 and Part 2: namely,
we make use of the approximation by smooth vector fields and apply some known
results which hold in the smooth case (as in Part 1) but also have to make
explicit computation with nonsmooth vector fields (as in Part 2). Our strategy
to prove Poincar\'{e}'s inequality is to exploit the general approach
developed by Lanconelli-Morbidelli in \cite{LM}, as well as Jerison's method
of proving Poincar\'{e}'s inequality first for the lifted vector fields and
then in the general case. We will say more about this in
\S \S \ref{Section lifting}-\ref{section poincare}; here we just want to
stress that all the results proved in this paper before Poincar\'{e}'s
inequality are needed, in order to apply the results in \cite{LM} and derive
this result in our context.

Finally, in \S \ref{section consequences}, we show some of the facts which
immediately follow from our results, thanks to the existing \textquotedblleft
axiomatic theories\textquotedblright: a Sobolev embedding, $p$-Poincar\'{e}'s
inequality, and Moser's iteration for variational second order operators
structured on nonsmooth vector fields

The paper ends with an Appendix where we collect some miscellaneous known
results about ordinary differential equations, which are used throughout the
paper, together with the justification of a Claim made in
\S \ref{section approximating balls}.

\textbf{Acknowledgements.} We wish to thank Ermanno Lanconelli and Giovanna
Citti for some useful conversation on the subject of this research.

While we were completing this paper, Annamaria Montanari and Daniele
Morbidelli told us that they were working on similar problems, and have proved
some results similar to ours in \cite{MM3}. We thank these authors for sharing
with us this information, and Daniele Morbidelli for having made important
remarks on a preprint of this paper.

\section{The subelliptic metric\label{section subelliptic metric}}

\textbf{Notation. }Let $X_{0},X_{1},...,X_{n}$ be a system of real vector
fields, defined in a domain of $\mathbb{R}^{p}.$ Let us assign to each $X_{i}$
a \textit{weight} $p_{i}$, saying that%
\[
p_{0}=2\text{ and }p_{i}=1\text{ for }i=1,2,...n.
\]

The following standard notation, will be used throughout the paper. For any
multiindex%
\[
I=\left(  i_{1},i_{2},...,i_{k}\right)
\]
we define the \textit{weight} of $I$ as%
\[
\left\vert I\right\vert =\sum_{j=1}^{k}p_{i_{j}}.
\]
Sometimes, we will also use the (usual) \textit{length} of $I,$%
\[
\ell\left(  I\right)  =k.
\]
For any multiindex $I=\left(  i_{1},i_{2},...,i_{k}\right)  $ we set:%
\[
X_{I}=X_{i_{1}}X_{i_{2}}...X_{i_{k}}%
\]
and%
\[
X_{\left[  I\right]  }=\left[  X_{i_{1}},\left[  X_{i_{2}},...\left[
X_{i_{k-1}},X_{i_{k}}\right]  ...\right]  \right]  .
\]
If $I=\left(  i_{1}\right)  ,$ then%
\[
X_{\left[  I\right]  }=X_{i_{1}}=X_{I}.
\]

As usual, $X_{\left[  I\right]  }$ can be seen either as a differential
operator or as a vector field. We will write%
\[
X_{\left[  I\right]  }f
\]
to denote the differential operator $X_{\left[  I\right]  }$ acting on a
function $f$, and
\[
\left(  X_{\left[  I\right]  }\right)  _{x}%
\]
to denote the vector field $X_{\left[  I\right]  }$ evaluated at the point $x$.

\bigskip

\textbf{Assumptions (A). }We assume that for some integer $r\geq2$ and some
bounded domain (i.e., connected open subset) $\Omega\subset\mathbb{R}^{p}$ the
following hold:

\begin{itemize}
\item[(A1)] The coefficients of the vector fields $X_{1},X_{2},...,X_{n}$
belong to $C^{r-1}\left(  \overline{\Omega}\right)  ,$ while the coefficients
of $X_{0}$ belong to $C^{r-2\,}\left(  \overline{\Omega}\right)  .$ Here and
in the following, $C^{k}$ stands for the classical space of functions with
continuous derivatives up to order $k$.

\item[(A2)] The vectors $\left\{  \left(  X_{\left[  I\right]  }\right)
_{x}\right\}  _{\left\vert I\right\vert \leq r}$ span $\mathbb{R}^{p}$ at
every point $x\in\Omega$.
\end{itemize}

Assumptions (A) will be in force throughout this section and the following.
These assumptions are consistent in view of the following

\begin{remark}
Under the assumption (A1) above, for any $1\leq k\leq r,$ the differential
operators%
\[
\left\{  X_{I}\right\}  _{\left\vert I\right\vert \leq k}%
\]
are well defined, and have $C^{r-k}$ coefficients. The same is true for the
vector fields $\left\{  X_{\left[  I\right]  }\right\}  _{\left\vert
I\right\vert \leq k}.$
\end{remark}

\bigskip

\textbf{Dependence of the constants.} We will often write that some constant
depends on the vector fields $X_{i}$'s and some fixed domain $\Omega^{\prime
}\Subset\Omega$. (Actually, the dependence on the $X_{i}$'s will be usually
left understood). Explicitly, this will mean that the constant depends on:

(i) $\Omega^{\prime}$;

(ii) the norms $C^{r-1}\left(  \overline{\Omega}\right)  $ of the coefficients
of $X_{i}$ $\left(  i=1,2,...,n\right)  $ and the norms $C^{r-2}\left(
\overline{\Omega}\right)  $ of the coefficients of $X_{0}$;

(iii) the moduli of continuity on $\overline{\Omega}$ of the highest order
derivatives of the coefficients of the $X_{i}$'s $\left(
i=0,1,2,...,n\right)  .$

(iv) a positive constant $c_{0}$ such that the following bound holds:%
\[
\inf_{x\in\Omega^{\prime}}\max_{\left\vert I_{1}\right\vert ,\left\vert
I_{2}\right\vert ,...,\left\vert I_{p}\right\vert \leq r}\left\vert
\det\left(  \left(  X_{\left[  I_{1}\right]  }\right)  _{x},\left(  X_{\left[
I_{2}\right]  }\right)  _{x},...,\left(  X_{\left[  I_{p}\right]  }\right)
_{x}\right)  \right\vert \geq c_{0}%
\]
(where \textquotedblleft$\det$\textquotedblright\ denotes the determinant of
the $p\times p$ matrix having the vectors $\left(  X_{\left[  I_{i}\right]
}\right)  _{x}$ as rows).

Note that (iv) is a quantitative way of assuring the validity of
H\"{o}rmander's condition, uniformly in $\Omega^{\prime}$.

\bigskip

The subelliptic metric introduced by Nagel-Stein-Wainger \cite{NSW}, in this
situation is defined as follows:

\begin{definition}
\label{Definition CC distance}For any $\delta>0,$ let $C\left(  \delta\right)
$ be the class of absolutely continuous mappings $\varphi:\left[  0,1\right]
\longrightarrow\Omega$ which satisfy%
\[
\varphi^{\prime}\left(  t\right)  =\sum_{\left\vert I\right\vert \leq r}%
a_{I}\left(  t\right)  \left(  X_{\left[  I\right]  }\right)  _{\varphi\left(
t\right)  }\text{ a.e.}%
\]
with $a_{I}:\left[  0,1\right]  \rightarrow\mathbb{R}$ measurable functions,
\[
\left\vert a_{I}\left(  t\right)  \right\vert \leq\delta^{\left\vert
I\right\vert }.
\]
Then define%
\[
d\left(  x,y\right)  =\inf\left\{  \delta>0:\exists\varphi\in C\left(
\delta\right)  \text{ with }\varphi\left(  0\right)  =x,\varphi\left(
1\right)  =y\right\}  .
\]

\end{definition}

The following property can be proved exactly like in the smooth case (see for
instance Proposition 1.1 in \cite{NSW}). We present a proof for the sake of completeness.

\begin{proposition}
[Relation with the Euclidean distance]\label{Prop fefferman phong}Assume
(A1)-(A2). Then the function $d:\Omega\times\Omega\rightarrow\mathbb{R}$ is a
(finite) distance. Moreover, there exist a positive constant $c_{1}$ depending
on $\Omega$ and the $X_{i}$'s and, for every $\Omega^{\prime}\Subset\Omega,$ a
positive constant $c_{2}$ depending on $\Omega^{\prime}$ and the $X_{i}$'s,
such that%
\begin{equation}
c_{1}\left\vert x-y\right\vert \leq d\left(  x,y\right)  \leq c_{2}\left\vert
x-y\right\vert ^{1/r}\text{ for any }x,y\in\Omega^{\prime}.
\label{fefferman-phong}%
\end{equation}
Hence, in particular, the distance $d$ induces Euclidean topology.
\end{proposition}

\begin{proof}
It is clear by definition that $d$ is a distance. Namely, this follows from
the fact that the union of two consecutive admissible curves can be
reparametrized to give an admissible curve. To prove (\ref{fefferman-phong}),
let $\varphi\in C\left(  \rho\right)  $, for some $\rho$, be any curve joining
$x$ to $y$, contained in $\Omega$, then:%
\begin{align*}
\varphi^{\prime}\left(  t\right)   &  =\sum_{\left\vert I\right\vert \leq
r}a_{I}\left(  t\right)  \left(  X_{\left[  I\right]  }\right)  _{\varphi
\left(  t\right)  }\\
\varphi\left(  0\right)   &  =x,\varphi\left(  1\right)  =y,\left\vert
a_{I}\left(  t\right)  \right\vert \leq\rho^{\left\vert I\right\vert }.
\end{align*}
Hence%
\begin{align*}
\left\vert y-x\right\vert  &  =\left\vert \int_{0}^{1}\varphi^{\prime}\left(
t\right)  dt\right\vert \leq\int_{0}^{1}\sum_{\left\vert I\right\vert \leq
r}\left\vert a_{I}\left(  t\right)  \left(  X_{\left[  I\right]  }\right)
_{\varphi\left(  t\right)  }\right\vert dt\leq\\
&  \leq\sup_{\left\vert I\right\vert \leq r,z\in\Omega}\left\vert \left(
X_{\left[  I\right]  }\right)  _{z}\right\vert \cdot\sum_{\left\vert
I\right\vert \leq r}\rho^{\left\vert I\right\vert }\leq c\rho.
\end{align*}
By the definition of $d$, taking the infimum over $\rho$ we get the first
inequality in (\ref{fefferman-phong}).

To prove the second inequality, fix $x_{0}\in\Omega^{\prime},$ and select a
subset $\eta$ of multiindices $I,$ $\left\vert I\right\vert \leq r,$ such that
$\left\{  X_{\left[  I\right]  }\right\}  _{I\in\eta}$ is a basis of
$\mathbb{R}^{p}$ at $x_{0},$ and therefore in a small neighborhood $U\left(
x_{0}\right)  \Subset\Omega;$ by continuity of the vector fields $\left\{
X_{\left[  I\right]  }\right\}  _{I\in\eta},$ we can take $U\left(
x_{0}\right)  $ small enough so that the $p\times p$ matrix%
\[
\left\{  \alpha_{IJ}\left(  x\right)  \right\}  _{I,J\in\eta}\text{, with
}\alpha_{IJ}\left(  x\right)  =\left(  X_{\left[  I\right]  }\right)
_{x}\cdot\left(  X_{\left[  J\right]  }\right)  _{x}%
\]
be uniformly positive in $U\left(  x_{0}\right)  $:
\begin{equation}
\sum_{I,J\in\eta}\alpha_{IJ}\left(  x\right)  \xi_{I}\xi_{J}\geq
c_{0}\left\vert \xi\right\vert ^{2}\text{ for any }x\in U\left(  x_{0}\right)
\text{, }\xi\in\mathbb{R}^{p}. \label{nonvanishing}%
\end{equation}
Now, for any $x,y\in U\left(  x_{0}\right)  ,$ let $\gamma$ be any $C^{1}$
curve contained in $U\left(  x_{0}\right)  ,$ such that $\gamma\left(
0\right)  =x,\gamma\left(  1\right)  =y$, and such that length$\left(
\gamma\right)  \leq c\left\vert x-y\right\vert ;$ moreover, we take $\gamma$
of constant speed, hence
\begin{equation}
\left\vert \gamma^{\prime}\left(  t\right)  \right\vert =\text{length}\left(
\gamma\right)  . \label{constantlength}%
\end{equation}
Since the vector fields $\left\{  X_{\left[  I\right]  }\right\}  _{I\in\eta}$
are a basis of $\mathbb{R}^{p}$ at any point of $U\left(  x_{0}\right)  ,$ we
can write%
\[
\gamma^{\prime}\left(  t\right)  =\sum_{I\in\eta}a_{I}\left(  t\right)
\left(  X_{\left[  I\right]  }\right)  _{\gamma\left(  t\right)  }%
\]
for suitable functions $a_{I},$ where, by (\ref{nonvanishing})
\[
\left\vert \gamma^{\prime}\left(  t\right)  \right\vert ^{2}\geq c_{0}%
\sum_{J\in\eta}\left\vert a_{J}\left(  t\right)  \right\vert ^{2}.
\]
Hence, by (\ref{constantlength}),
\begin{equation}
\left\vert a_{I}\left(  t\right)  \right\vert \leq c\left\vert \gamma^{\prime
}\left(  t\right)  \right\vert \leq c\left\vert x-y\right\vert \leq
c\left\vert x-y\right\vert ^{\left\vert I\right\vert /r}. \label{NSW allegro}%
\end{equation}
Therefore $\gamma\in C\left(  c\left\vert x-y\right\vert ^{1/r}\right)  $ and
the second inequality in (\ref{fefferman-phong}) follows, for any $x,y\in
U\left(  x_{0}\right)  .$ A compactness argument gives the general case.
\end{proof}

\section{Approximating vector fields and the doubling
condition\label{section approximating balls}}

A key tool in the study of the nonsmooth vector fields $X_{i}$ is to
approximate them, locally, with smooth vector fields, as we shall explain in
this section.

Let us start with the following general remark, which follows from the
standard Taylor formula.

For any $f\in C^{k}\left(  \Omega\right)  \ $and $\Omega^{\prime}\Subset
\Omega,$ let us define the following moduli of continuity:%
\begin{align*}
\omega_{\alpha}\left(  \delta\right)   &  =\sup\left\{  \left\vert D^{\alpha
}f\left(  x\right)  -D^{\alpha}f\left(  y\right)  \right\vert :x,y\in
\Omega^{\prime},\left\vert x-y\right\vert \leq\delta\right\}  \text{ for any
}\left\vert \alpha\right\vert =k;\\
\omega_{k}\left(  \delta\right)   &  =\max_{\left\vert \alpha\right\vert
=k}\omega_{\alpha}\left(  \delta\right)  .
\end{align*}
Then the following holds:%
\[
f\left(  x\right)  =\sum_{\left\vert \alpha\right\vert \leq k}\frac{D^{\alpha
}f\left(  x_{0}\right)  }{\alpha!}\left(  x-x_{0}\right)  ^{\alpha}+O\left(
\left\vert x-x_{0}\right\vert ^{k}\omega_{k}\left(  \left\vert x-x_{0}%
\right\vert \right)  \right)  \text{ for any }x,x_{0}\in\Omega^{\prime}.
\]
The error term $O\left(  \left\vert x-x_{0}\right\vert ^{k}\omega_{k}\left(
\left\vert x-x_{0}\right\vert \right)  \right)  $ can be rewritten as
$o\left(  \left\vert x-x_{0}\right\vert ^{k}\right)  $, where this symbol
means that%
\[
\frac{o\left(  \left\vert x-x_{0}\right\vert ^{k}\right)  }{\left\vert
x-x_{0}\right\vert ^{k}}\rightarrow0\text{ for }x\rightarrow x_{0},
\]
\textit{uniformly for }$x_{0}$\textit{ ranging in }$\Omega^{\prime}.$ We
stress that, although elementary, this remark is crucial in allowing us to
prove the doubling condition assuming the coefficients of $X_{i}$ just in
$C^{r-1}$ (and not, for instance, in $C^{r-1,1},$ as we shall do later).

Now, fix a point $x_{0}\in\Omega;$ for any $i=0,1,2,...,n$, let us consider
the vector field
\[
X_{i}=\sum_{j=1}^{p}b_{ij}\left(  x\right)  \partial_{x_{j}};
\]
let $p_{ij}^{r}\left(  x\right)  $ be the Taylor polynomial of $b_{ij}\left(
x\right)  $ of center $x_{0}$ and order $r-p_{i}$; note that, under assumption
(A1), and by the above remark,%
\begin{equation}
b_{ij}\left(  x\right)  =p_{ij}^{r}\left(  x\right)  +o\left(  \left\vert
x-x_{0}\right\vert ^{r-p_{i}}\right)  \label{Holder approximation}%
\end{equation}
with the above meaning of the symbol $o\left(  \cdot\right)  $.

Set%
\[
S_{i}^{x_{0}}=\sum_{j=1}^{p}p_{ij}^{r}\left(  x\right)  \partial_{x_{j}}.
\]
We will often write $S_{i}$ in place of $S_{i}^{x_{0}},$ leaving the
dependence on the point $x_{0}$ implicitly understood.

From (\ref{Holder approximation}) immediately follows:

\begin{proposition}
\label{Proposition S_i}Assume (A1) (see \S \ref{section subelliptic metric}).
Then the $S_{i}^{x_{0}}$'s $\left(  i=0,1,2,...,n\right)  $ are smooth vector
fields defined in the whole space, satisfying:%
\[
\left(  S_{I}\right)  _{x_{0}}=\left(  X_{I}\right)  _{x_{0}}\text{ and
}\left(  S_{\left[  I\right]  }\right)  _{x_{0}}=\left(  X_{\left[  I\right]
}\right)  _{x_{0}}\text{ for any }I\text{ with }\left\vert I\right\vert \leq
r\text{.}%
\]
Moreover,%
\begin{equation}
X_{\left[  I\right]  }-S_{\left[  I\right]  }=\sum_{j=1}^{p}c_{I}^{j}\left(
x\right)  \partial_{x_{j}}\text{ with }c_{I}^{j}\left(  x\right)  =o\left(
\left\vert x-x_{0}\right\vert ^{r-\left\vert I\right\vert }\right)
.\label{X_I-S_I}%
\end{equation}

\end{proposition}

We also need to check that the $S_{i}^{x_{0}}$'s satisfy H\"{o}rmander's
condition in a neighborhood of $x_{0},$ with some uniform control on the
diameter of this neighborhood:

\begin{lemma}
\label{Lemma intorno S}For every domain $\Omega^{\prime}\Subset\Omega$ there
exists a constant $\delta>0$ depending on $\Omega^{\prime}$ and the $X_{i}$'s,
such that for any $x_{0}\in\Omega^{\prime}$ the smooth vector fields
$S_{1}^{x_{0}},S_{2}^{x_{0}},...,S_{n}^{x_{0}}$ satisfy H\"{o}rmander's
condition in%
\[
U_{\delta}\left(  x_{0}\right)  =\left\{  x\in\Omega:\left\vert x-x_{0}%
\right\vert <\delta\right\}  .
\]

\end{lemma}

\begin{proof}
Let $f\left(  x,x_{0}\right)  =\max_{\eta}$ $\left\vert \det\left\{  \left(
S_{\left[  I\right]  }^{x_{0}}\right)  _{x}\right\}  _{I\in\eta}\right\vert $
where the maximum is taken over all the possible choices of family $\eta$ of
$p$ multiindices $I$ with $\left\vert I\right\vert \leq r$. Writing the
explicit form of the $S_{i}^{x_{0}}$'s:%
\begin{align*}
S_{i}^{x_{0}}  &  =\sum_{j=1}^{n}\left(  \sum_{\left\vert \alpha\right\vert
\leq r-1}\frac{D^{\alpha}b_{ij}\left(  x_{0}\right)  }{\alpha!}\left(
x-x_{0}\right)  ^{\alpha}\right)  \partial_{x_{j}},\text{ where}\\
X_{i}  &  =\sum_{j=1}^{n}b_{ij}\left(  x\right)  \partial_{x_{j}},
\end{align*}
we see that the function $f$ is continuous in $\Omega\times\Omega$. Also
observe now that, since $\left(  S_{\left[  I\right]  }^{x_{0}}\right)
_{x_{0}}=\left(  X_{\left[  I\right]  }\right)  _{x_{0}}$ we have%
\[
f\left(  x_{0},x_{0}\right)  =\max_{\eta}\left\vert \det\left\{  \left(
X_{\left[  I\right]  }\right)  _{x_{0}}\right\}  _{I\in\eta}\right\vert \geq
c_{0}>0\text{ }\forall x_{0}\in\Omega^{\prime}.
\]

The uniform continuity of $f$ (in a suitable domain that contains
$\Omega^{\prime}\times\Omega^{\prime}$) allows to find $\delta>0$ such that
$f\left(  x,x_{0}\right)  \geqslant\frac{1}{2}c_{0}$ is $\left\vert
x-x_{0}\right\vert \leqslant\delta$. This proves that in $U_{\delta}\left(
x_{0}\right)  $ the $S_{i}^{x_{0}}$ satisfy H\"{o}rmander's condition.
\end{proof}

The family $\left\{  S_{i}^{x_{0}}\right\}  _{i=1}^{n}$ will be a key tool for
us. Namely, throughout the paper we will apply to the $S_{i}^{x_{0}}$'s four
important results proved by Nagel-Stein-Wainger \cite{NSW} for smooth
H\"{o}rmander's vector fields, namely: the estimate on the volume of metric
balls; the doubling condition (both contained in \cite[Theorem 1]{NSW}); the
equivalence between two different distances induced by the vector fields
(\cite[Theorem 4]{NSW}), and a more technical result which we will recall
later as Theorem \ref{Thm Morbidelli NSW}. Since, on the other hand, for every
different point $x_{0}\in\Omega^{\prime}$ we are considering a
\textit{different} \textit{system} of smooth vector fields, we are obliged to
check that the constants appearing in Nagel-Stein-Wainger's estimates depend
on the smooth vector fields in a way that allows to keep them under uniform
control, for $x_{0}$ ranging in $\Omega^{\prime}\Subset\Omega.$ This is
possible in view of the following:

\begin{claim}
\label{Claim Jerison}Let $S_{1},S_{2},...,S_{n}$ be a system of smooth
H\"{o}rmander's vector fields of step $r$ in some neighbourhood $\Omega$ of a
bounded domain $\Omega^{\prime}\subset\mathbb{R}^{p}.$ Then all the constants
appearing in the estimates proved in \cite{NSW} depend on the $S_{i}$'s only
through the following quantities:

\begin{enumerate}
\item an upper bound on the $C^{k}\left(  \overline{\Omega^{\prime}}\right)  $
norms of the coefficients of the $S_{i}$'s, for some \textquotedblleft
large\textquotedblright\ $k$ only depending on the numbers $p,n,r$;

\item a positive lower bound on%
\[
\inf_{x\in\Omega^{\prime}}\max_{\left\vert I_{1}\right\vert ,\left\vert
I_{2}\right\vert ,...,\left\vert I_{p}\right\vert \leq r}\left\vert
\det\left(  \left(  S_{\left[  I_{1}\right]  }\right)  _{x},\left(  S_{\left[
I_{2}\right]  }\right)  _{x},...,\left(  S_{\left[  I_{p}\right]  }\right)
_{x}\right)  \right\vert .
\]

\end{enumerate}
\end{claim}

A justification of this Claim will be sketched in the Appendix. Here we just
recall that a similar claim (specifically referring to the doubling condition
proved in \cite{NSW}) was first made by Jerison in \cite{J}.

Now, let us fix a domain $\Omega^{\prime}\Subset\Omega;$ for any $x_{0}%
\in\Omega^{\prime}$ we can see by the explicit form of the $S_{i}^{x_{0}}$ that

\begin{enumerate}
\item the $C^{k}\left(  \overline{\Omega^{\prime}}\right)  $ norms of the
coefficients of the $S_{i}^{x_{0}}$'s are bounded, for all $k,$ by a constant
only depending on the $C^{r-p_{i}}\left(  \overline{\Omega^{\prime}}\right)  $
norms of the coefficients of the vector fields $X_{i},$ the numbers $r,p$ and
the diameter of $\Omega^{\prime}$.
\end{enumerate}

Moreover, from the proof of Lemma \ref{Lemma intorno S} we read that:

\begin{enumerate}
\item[2.] there exists a constant $c_{0}>0$ such that for any $x_{0}\in
\Omega^{\prime},$ if $U_{\delta}\left(  x_{0}\right)  $ is the neighborhood
appearing in Lemma \ref{Lemma intorno S}, where the $S_{i}^{x_{0}}$ satisfy
H\"{o}rmander's condition, then
\[
\inf_{x\in U_{\delta}\left(  x_{0}\right)  }\max_{\left\vert I_{j}\right\vert
\leq r}\left\vert \det\left(  \left(  S_{\left[  I_{1}\right]  }^{x_{0}%
}\right)  _{x},\left(  S_{\left[  I_{2}\right]  }^{x_{0}}\right)
_{x},...,\left(  S_{\left[  I_{p}\right]  }^{x_{0}}\right)  _{x}\right)
\right\vert \geq c_{0}.
\]
The constant $c_{0}$ depends on vector fields $X_{i}$'s only through the
$C^{r-p_{i}}\left(  \overline{\Omega^{\prime}}\right)  $ norms of the
coefficients, the moduli of continuity on $\Omega^{\prime}$ of the highest
order derivatives of the coefficients, and the positive quantity%
\[
\inf_{x\in\Omega^{\prime}}\max_{\left\vert I_{j}\right\vert \leq r}\left\vert
\det\left(  \left(  X_{\left[  I_{1}\right]  }\right)  _{x},\left(  X_{\left[
I_{2}\right]  }\right)  _{x},...,\left(  X_{\left[  I_{p}\right]  }\right)
_{x}\right)  \right\vert .
\]

\end{enumerate}

The above discussion allows us to assure that every time we will apply to the
system of approximating vector fields $S_{i}^{x_{0}}$ some results proved in
\cite{NSW} for smooth H\"{o}rmander's vector fields, the constants appearing
in these estimates will be bounded, uniformly for $x_{0}$ ranging in
$\Omega^{\prime},$ in terms of quantities related to our original nonsmooth
vector fields $X_{i}$.

\bigskip

In order to prove the doubling condition for the balls defined by the distance
induced by nonsmooth vector fields, the simplest way is to compare this
distance to the one induced by the smooth approximating vector fields $S_{i}$.
We will show that these two distances are locally equivalent, in a suitable
pointwise sense, which will be enough to deduce the doubling condition:

\begin{theorem}
[Approximating balls]\label{Thm dX equiv dS}Assume (A1)-(A2). For any fixed
$x_{0}\in\Omega^{\prime}\Subset\Omega,$ let $S_{i}^{x_{0}}$ be the smooth
vector fields defined as above. Let us denote by $d_{X}$ and $d_{S}$ the
distances induced by the $X_{i}$'s and the $S_{i}$'s, respectively, and by
$B_{X}$ and $B_{S}$ the corresponding metric balls. There exist positive
constants $c_{1},c_{2},r_{0}$ depending on $\Omega,\Omega^{\prime}$ and the
$X_{i}$'s$,$ but not on $x_{0}$, such that%
\[
B_{S^{x_{0}}}\left(  x_{0},c_{1}\rho\right)  \subset B_{X}\left(  x_{0}%
,\rho\right)  \subset B_{S^{x_{0}}}\left(  x_{0},c_{2}\rho\right)
\]
for any $\rho<r_{0}.$
\end{theorem}

\begin{proof}
Let $x\in B_{S^{x_{0}}}\left(  x_{0},\rho\right)  .$ This means there exists
$\phi\left(  t\right)  $ such that%
\[
\left\{
\begin{array}
[c]{l}%
\phi^{\prime}\left(  t\right)  =\sum_{\left\vert I\right\vert \leq r}%
a_{I}\left(  t\right)  \left(  S_{\left[  I\right]  }^{x_{0}}\right)
_{\phi\left(  t\right)  }\\
\phi\left(  0\right)  =x_{0},\phi\left(  1\right)  =x
\end{array}
\right.
\]
with $\left\vert a_{I}\left(  t\right)  \right\vert \leq\rho^{\left\vert
I\right\vert }.$ Let $\gamma\left(  t\right)  $ be \textit{a solution} to the
system%
\[
\left\{
\begin{array}
[c]{l}%
\gamma^{\prime}\left(  t\right)  =\sum_{\left\vert I\right\vert \leq r}%
a_{I}\left(  t\right)  \left(  X_{\left[  I\right]  }\right)  _{\gamma\left(
t\right)  }\\
\gamma\left(  0\right)  =x_{0},
\end{array}
\right.
\]
and set $x^{\prime}=\gamma\left(  1\right)  .$ (Note that in general we don't
have \textit{uniqueness}, because the $X_{\left[  I\right]  }$'s are just
continuous if $\left\vert I\right\vert =r$ and the functions $a_{I}\left(
\cdot\right)  $ can be merely measurable; however, existence is granted by
Carath\'{e}odory's theorem, see the Appendix).

By definition, $x^{\prime}\in B_{X}\left(  x_{0},\rho\right)  .$ We have%
\begin{align*}
&  \left\vert \gamma\left(  t\right)  -\phi\left(  t\right)  \right\vert
=\left\vert \int_{0}^{t}\left(  \gamma^{\prime}\left(  s\right)  -\phi
^{\prime}\left(  s\right)  \right)  ds\right\vert \leq\\
&  \leq\sum_{\left\vert I\right\vert \leq r}\int_{0}^{t}\left\vert
a_{I}\left(  s\right)  \right\vert \left\vert \left(  X_{\left[  I\right]
}\right)  _{\gamma\left(  s\right)  }-\left(  S_{\left[  I\right]  }^{x_{0}%
}\right)  _{\phi\left(  s\right)  }\right\vert ds\leq\\
&  \leq\sum_{\left\vert I\right\vert \leq r}\int_{0}^{t}\left\vert
a_{I}\left(  s\right)  \right\vert \left\{  \left\vert \left(  X_{\left[
I\right]  }\right)  _{\gamma\left(  s\right)  }-\left(  X_{\left[  I\right]
}\right)  _{\phi\left(  s\right)  }\right\vert +\left\vert \left(  X_{\left[
I\right]  }\right)  _{\phi\left(  s\right)  }-\left(  S_{\left[  I\right]
}^{x_{0}}\right)  _{\phi\left(  s\right)  }\right\vert \right\}  ds\\
&  \equiv A+B.
\end{align*}
By (\ref{X_I-S_I}) we have, for $\rho\leq r_{0}$, $r_{0}$ small enough%
\[
B\leq c\sum_{\left\vert I\right\vert \leq r}\int_{0}^{t}\rho^{\left\vert
I\right\vert }\left\vert \phi\left(  s\right)  -x_{0}\right\vert
^{r-\left\vert I\right\vert }ds.
\]
Since, by (\ref{fefferman-phong}), $\left\vert \phi\left(  s\right)
-x_{0}\right\vert \leq cd_{S^{x_{0}}}\left(  \phi\left(  s\right)
,x_{0}\right)  \leq c\rho,$we obtain%
\[
B\leq\sum_{\left\vert I\right\vert \leq r}c\rho^{\left\vert I\right\vert }%
\rho^{r-\left\vert I\right\vert }=c\rho^{r}.
\]
Moreover,%
\begin{align*}
A  &  \leq\sum_{\left\vert I\right\vert \leq r}\int_{0}^{t}\rho^{\left\vert
I\right\vert }\left\vert \left(  X_{\left[  I\right]  }\right)  _{\gamma
\left(  s\right)  }-\left(  X_{\left[  I\right]  }\right)  _{\phi\left(
s\right)  }\right\vert ds\leq\\
&  \leq c\sum_{\left\vert I\right\vert <r}\int_{0}^{t}\rho^{\left\vert
I\right\vert }\left\vert \gamma\left(  s\right)  -\phi\left(  s\right)
\right\vert ds+\sum_{\left\vert I\right\vert =r}\rho^{r}\cdot2\sup
_{z\in\overline{\Omega}}\left\vert \left(  X_{\left[  I\right]  }\right)
_{z}\right\vert \\
&  \leq c\rho\int_{0}^{t}\left\vert \gamma\left(  s\right)  -\phi\left(
s\right)  \right\vert ds+c\rho^{r}%
\end{align*}
where we used the fact that the $X_{\left[  I\right]  }\in C^{1}\left(
\overline{\Omega}\right)  $ for $\left\vert I\right\vert <r$ and $X_{\left[
I\right]  }\in C^{0}\left(  \overline{\Omega}\right)  $ for $\left\vert
I\right\vert =r$. Therefore we have, for any $t\in\left(  0,1\right)  $%
\[
\left\vert \gamma\left(  t\right)  -\phi\left(  t\right)  \right\vert \leq
c\rho\int_{0}^{t}\left\vert \gamma\left(  s\right)  -\phi\left(  s\right)
\right\vert ds+c\rho^{r}.
\]
By Gronwall's Lemma (see the Appendix) this implies%
\[
\left\vert \gamma\left(  t\right)  -\phi\left(  t\right)  \right\vert \leq
c\rho^{r}%
\]
for any $t\in\left(  0,1\right)  ,$ and so
\[
\left\vert x-x^{\prime}\right\vert \leq c\rho^{r}%
\]
which, again by (\ref{fefferman-phong}), implies%
\[
d_{X}\left(  x,x^{\prime}\right)  \leq c\left\vert x-x^{\prime}\right\vert
^{1/r}\leq c\rho.
\]
Since we already know that $x^{\prime}\in B_{X}\left(  x_{0},\rho\right)  $,
we infer $x\in B_{X}\left(  x_{0},c_{1}\rho\right)  .$ In other words,%
\[
B_{S^{x_{0}}}\left(  x_{0},\rho\right)  \subset B_{X}\left(  x_{0},c_{1}%
\rho\right)  .
\]
We can now repeat the same argument exchanging the roles of $d_{X}%
,d_{S^{x_{0}}},$ and get%
\[
B_{X}\left(  x_{0},\rho\right)  \subset B_{S^{x_{0}}}\left(  x_{0},c_{2}%
\rho\right)  .
\]
Actually, in this case, some arguments simplify, due to the smoothness of the
vector fields $S_{i}^{x_{0}}.$
\end{proof}

By the doubling condition which holds for the balls induced by smooth vector
fields, proved by Nagel-Stein-Wainger (see \cite[Thm.1]{NSW}), the above
Theorem immediately implies the following

\begin{theorem}
[Doubling condition]\label{Thm doubling nonsmooth}Assume (A1)-(A2). For any
domain $\Omega^{\prime}\Subset\Omega,$ there exist positive constants
$c,r_{0},$ depending on $\Omega,\Omega^{\prime}$ and the $X_{i}$'s$,$ such
that%
\[
\left\vert B_{X}\left(  x_{0},2\rho\right)  \right\vert \leq c\left\vert
B_{X}\left(  x_{0},\rho\right)  \right\vert
\]
for any $x_{0}\in\Omega^{\prime},$ $\rho<r_{0}.$
\end{theorem}

Nagel-Stein-Wainger in \cite{NSW} deduce the doubling condition from a sharp
result about the volume of metric balls, which we recall here. Also this
result follows, in the case of nonsmooth vector fields, by the above theorem
about approximating balls. Even though, in our approach, the volume estimate
of nonsmooth balls is not necessary to prove the nonsmooth doubling condition,
it can be of independent interest.

\begin{theorem}
[Volume of metric balls]Let $\eta$ be any family of $p$ multiindices
$I_{1},I_{2},...,I_{p}$ with $\left\vert I_{j}\right\vert \leq r$; let
$\left\vert \eta\right\vert =\sum_{j=1}^{p}\left\vert I_{j}\right\vert $. Let
$\lambda_{\eta}\left(  x\right)  $ be the determinant of the $p\times p$
matrix of rows $\left\{  \left(  X_{\left[  I_{j}\right]  }\right)
_{x}\right\}  _{I_{j}\in\eta}$. For any $\Omega^{\prime}\Subset\Omega$ there
exist positive constants $c_{1},c_{2},r_{0}$ depending on $\Omega
,\Omega^{\prime}$ and the $X_{i}$'s$,$ such that%
\[
c_{1}\sum_{\eta}\left\vert \lambda_{\eta}\left(  x\right)  \right\vert
\rho^{\left\vert \eta\right\vert }\leq\left\vert B_{X}\left(  x,\rho\right)
\right\vert \leq c_{2}\sum_{\eta}\left\vert \lambda_{\eta}\left(  x\right)
\right\vert \rho^{\left\vert \eta\right\vert }%
\]
for any $\rho<r_{0},$ $x\in\Omega^{\prime}$, where the sum is taken over any
family $\eta$ with the above properties.
\end{theorem}

\begin{proof}
By Nagel-Stein-Wainger's theorem (see \cite[Thm.1]{NSW}),
\[
c_{1}\sum_{\eta}\left\vert \lambda_{\eta}\left(  x\right)  \right\vert
\rho^{\left\vert \eta\right\vert }\leq\left\vert B_{S}\left(  x,\rho\right)
\right\vert \leq c_{2}\sum_{\eta}\left\vert \lambda_{\eta}\left(  x\right)
\right\vert \rho^{\left\vert \eta\right\vert }%
\]
where $B_{S}$ is the ball induced by the smooth vector fields $S_{i}^{x}$.
Note that $\left(  X_{\left[  I\right]  }\right)  _{x}=\left(  S_{\left[
I\right]  }^{x}\right)  _{x}$ for $\left\vert I\right\vert \leq r$, therefore
the quantity $\lambda_{\eta}\left(  x\right)  $ computed for the system
$X_{i}$ is the same of that computed for the system $S_{i}^{x}$. Moreover, the
constants $c_{1},c_{2},r_{0}$ do not depend on the point $x$, but only on
$\Omega,\Omega^{\prime}$ and the $X_{i}$'s$.$ Then the result follows by
Theorem \ref{Thm dX equiv dS}.
\end{proof}

\section{Exponential and quasiexponential maps\label{section exp and quasiexp}%
}

In this section we slightly strengthen our assumptions as follows:

\textbf{Assumptions (B). }We keep assumptions (A) but, in the case $r=2,$ we
also require $X_{0}$ to have Lipschitz continuous (instead of merely
continuous) coefficients.

Accordingly, the constants in our estimates will depend on the $X_{i}$ through
the quantities stated in \S \ref{section subelliptic metric} (see
\textquotedblleft Dependence of the constants\textquotedblright) and the
Lipschitz norms of the coefficients of $X_{0}.$

\bigskip

Let us recall the standard definition of \textit{exponential of a vector
field}. We set:%
\[
\exp\left(  tX\right)  \left(  x_{0}\right)  =\varphi\left(  t\right)
\]
where $\varphi$ is the solution to the Cauchy problem%
\begin{equation}
\left\{
\begin{array}
[c]{l}%
\varphi^{\prime}\left(  \tau\right)  =X_{\varphi\left(  \tau\right)  }\\
\varphi\left(  0\right)  =x_{0}%
\end{array}
\right.  \label{Cauchy2}%
\end{equation}
The point $\exp\left(  tX\right)  \left(  x_{0}\right)  $ is uniquely defined
for $t\in\mathbb{R}$ small enough, as soon as $X$ has Lipschitz continuous
coefficients, by the classical Cauchy's theorem about existence and uniqueness
for solutions to Cauchy problems. For a fixed $\Omega^{\prime}\Subset\Omega,$
a $t$-neighborhood of zero where $\exp\left(  tX\right)  \left(  x_{0}\right)
$ is defined can be found uniformly for $x_{0}$ ranging in $\Omega^{\prime}$
(see the Appendix).

Equivalently, we can write%
\[
\exp\left(  tX\right)  \left(  x_{0}\right)  =\phi\left(  1\right)
\]
where $\phi$ is the solution to the Cauchy problem%
\[
\left\{
\begin{array}
[c]{l}%
\phi^{\prime}\left(  \tau\right)  =tX_{\phi\left(  \tau\right)  }\\
\phi\left(  0\right)  =x_{0}.
\end{array}
\right.
\]
By definition of subelliptic distance, this implies in particular that%
\begin{equation}
d\left(  \exp\left(  t^{p_{i}}X_{i}\right)  \left(  x_{0}\right)
,x_{0}\right)  \leq ct \label{d(exp,x)}%
\end{equation}
(recall that $p_{i}$ is the weight of $X_{i},$ defined in
\S \ref{section subelliptic metric}).

Let us also define the following \textit{quasiexponential maps }(for
$i_{1},i_{2},...\in\left\{  0,1,...,n\right\}  $):%
\begin{align*}
C_{1}\left(  t,X_{i_{1}}\right)   &  =\exp\left(  t^{p_{i_{1}}}X_{i_{1}%
}\right)  ;\\
C_{2}\left(  t,X_{i_{1}}X_{i_{2}}\right)   &  =\exp\left(  -t^{p_{i_{2}}%
}X_{i_{2}}\right)  \exp\left(  -t^{p_{i_{1}}}X_{i_{1}}\right)  \exp\left(
t^{p_{i_{2}}}X_{i_{2}}\right)  \exp\left(  t^{p_{i_{1}}}X_{i_{1}}\right)  ;\\
&  ...\\
C_{l}\left(  t,X_{i_{1}}X_{i_{2}}...X_{i_{l}}\right)   &  =\\
=C_{l-1}  &  \left(  t,X_{i_{2}}...X_{i_{l}}\right)  ^{-1}\exp\left(
-t^{p_{i_{1}}}X_{i_{1}}\right)  C_{l-1}\left(  t,X_{i_{2}}...X_{i_{l}}\right)
\exp\left(  t^{p_{i_{1}}}X_{i_{1}}\right)
\end{align*}

\begin{remark}
\label{Remark quasiexponential}Note that, in this definition, the exponential
is taken only on the vector fields $X_{i}$ ($i=1,2,...,n$), which have at
least $C^{1}$ coefficients, and $X_{0},$ which has at least Lipschitz
continuous coefficients, hence $C_{l}$ is well defined, for $t$ small enough.
Also, note that $C_{\ell\left(  I\right)  }\left(  t,X_{I}\right)  $ is the
product of a fixed number (depending on $\ell\left(  I\right)  $) of factors
of the kind $\exp\left(  \pm t^{p_{i}}X_{i}\right)  $ with $i=0,1,2,...,n$
(and $p_{i}=2,1,1,...,1$, respectively). In particular, this implies that each
map
\[
x\longmapsto C_{\ell\left(  I\right)  }\left(  t,X_{I}\right)  \left(
x\right)
\]
is invertible, for $t$ small enough.

Moreover, if%
\[
x=C_{l}\left(  t,X_{i_{1}}X_{i_{2}}...X_{i_{l}}\right)  \left(  x_{0}\right)
,
\]
this means that the points $x_{0},x$ can be joined by a curve composed by a
finite number of integral curves of the $X_{i}$'s, and that $d\left(
x,x_{0}\right)  \leq ct$ (see (\ref{d(exp,x)})). We are going to show that
every point $x$ in a small neighborhood of $x_{0}$ can be obtained in this
way: this will imply Chow's connectivity theorem.
\end{remark}

The key result about the maps $C_{l}$ defined above is the following:

\begin{theorem}
[Approximation of quasiexponential maps with commutators]\label{Thm map C_el}%
For any fixed $x_{0}\in\Omega,$ let $\eta$ be a set of $p$ multiindices $I$
with $\left\vert I\right\vert \leq r$ such that $\left\{  \left(  X_{\left[
I\right]  }\right)  _{x_{0}}\right\}  _{I\in\eta}$ is a basis of
$\mathbb{R}^{p}$. Then there exists a neighborhood $U$ of $x_{0}$ such that
for any $x\in U$ and any $I\in\eta$%
\begin{align*}
C_{\ell\left(  I\right)  }\left(  t,X_{I}\right)  \left(  x\right)   &
=x+t^{\left\vert I\right\vert }\left(  X_{\left[  I\right]  }\right)
_{x}+o\left(  t^{\left\vert I\right\vert }\right)  \text{ as }t\rightarrow0\\
C_{\ell\left(  I\right)  }\left(  t,X_{I}\right)  ^{-1}\left(  x\right)   &
=x-t^{\left\vert I\right\vert }\left(  X_{\left[  I\right]  }\right)
_{x}+o\left(  t^{\left\vert I\right\vert }\right)  \text{ as }t\rightarrow0
\end{align*}
where the remainder $o\left(  t^{\left\vert I\right\vert }\right)  $ is a map
$x\mapsto f\left(  t\right)  \left(  x\right)  $ such that%
\[
\sup_{x\in\overline{U}}\frac{\left\vert f\left(  t\right)  \left(  x\right)
\right\vert }{t^{\left\vert I\right\vert }}\rightarrow0\text{ as }%
t\rightarrow0.
\]

\end{theorem}

The above theorem says that moving in a suitable way along a chain of integral
lines of the $X_{i}$'s can give, as a net result, a displacement approximately
in the direction of any commutator of the vector fields. Since the commutators
span, this will imply that we can reach any point in this way.

The proof of the above theorem is organized in several steps. First of all, we
have to get some sharp information about the degree of regularity of
exponential maps. We start with the following classical result:

\begin{theorem}
\label{Thm Petrovski}Let $F\left(  t,x\right)  =\exp\left(  tX\right)  \left(
x\right)  ,$ i.e.%
\begin{equation}
\left\{
\begin{array}
[c]{l}%
\frac{\partial F}{\partial t}=X_{F\left(  t,x\right)  }\\
F\left(  0,x\right)  =x
\end{array}
\right.  \label{Cauchy}%
\end{equation}
where $X$ is $C^{k}$ in a neighborhood of $x_{0}.$ Then, the function $\left(
t,x\right)  \longmapsto F\left(  t,x\right)  $ is $C^{k}$ in a neighborhood of
$\left(  0,x_{0}\right)  .$
\end{theorem}

\begin{proof}
In \cite[\S 21 chap.3 and \S 29 in chap.4]{P}, it is proved that the
derivatives
\[
\frac{\partial^{\alpha}F}{\partial x^{\alpha}}\left(  t,x\right)
\]
are continuous in a neighborhood of $\left(  0,x_{0}\right)  ,$ for
$\left\vert \alpha\right\vert \leq k$. To complete the proof, we have to check
the continuity of mixed derivatives%
\[
\frac{\partial^{\alpha+\beta}F}{\partial x^{\alpha}\partial t^{\beta}}\left(
t,x\right)
\]
for $\left\vert \alpha\right\vert +\left\vert \beta\right\vert \leq
k,\left\vert \beta\right\vert \geq1$. If $\alpha=0,$ the required regularity
is read from the equation. The general case requires an inductive reasoning.
To fix ideas, let us consider the case $k=2.$ Then we have:%
\[
\frac{\partial^{2}F\left(  t,x\right)  }{\partial x_{i}\partial t}%
=\frac{\partial}{\partial x_{i}}\left(  X\left(  F\left(  t,x\right)  \right)
\right)  =J_{X}\left(  F\left(  t,x\right)  \right)  \cdot\frac{\partial
F\left(  t,x\right)  }{\partial x_{i}}%
\]
where $J_{X}$ denotes the Jacobian matrix of the map $x\mapsto X_{x}$. Since
we already know that $\frac{\partial F\left(  t,x\right)  }{\partial x_{i}}$
and $J_{X}$ are continuous, continuity of $\frac{\partial^{2}F\left(
t,x\right)  }{\partial x_{i}\partial t}$ follows. The general case can be
treated analogously.
\end{proof}

\begin{corollary}
\label{Corollary 1}If $X$ is $C^{k}$ in a neighborhood of $x_{0}$ and $F$ is
as in (\ref{Cauchy}), then%
\[
\frac{\partial^{m+\alpha}F}{\partial t^{m}\partial x^{\alpha}}\in C\left(
U\left(  0,x_{0}\right)  \right)
\]
for some neighborhood $U\left(  0,x_{0}\right)  ,$ if $m\geq1$ and
$m+\left\vert \alpha\right\vert \leq k+1$.
\end{corollary}

\begin{proof}
If $m+\left\vert \alpha\right\vert \leq k$ this is contained in the previous
theorem, so let $m+\left\vert \alpha\right\vert =k+1,$ $m\geq1.$ Since $F$ is
$C^{k}$ by the previous theorem, $X$ is $C^{k}$ by assumption, and%
\[
\frac{\partial F\left(  t,x\right)  }{\partial t}=X_{F\left(  t,x\right)  },
\]
then $\frac{\partial F}{\partial t}\in C^{k}\left(  U\left(  0,x_{0}\right)
\right)  .$ Hence%
\[
\frac{\partial^{m+\alpha}F}{\partial t^{m}\partial x^{\alpha}}=\frac
{\partial^{m-1+\alpha}}{\partial t^{m-1}\partial x^{\alpha}}\frac{\partial
F}{\partial t}\in C\left(  U\left(  0,x_{0}\right)  \right)  .
\]

\end{proof}

\begin{corollary}
\label{Corollary differentiabke k+1}If $X$ is $C^{k}$ in a neighborhood of
$x_{0}$ and $F$ is as in (\ref{Cauchy}), then $F$ is $k+1$ times
differentiable at $\left(  0,x\right)  ,$ for any $x$ in that neighborhood of
$x_{0}$.
\end{corollary}

\begin{proof}
By Theorem \ref{Thm Petrovski}, we are left to prove that%
\[
\frac{\partial^{m+\alpha}F}{\partial t^{m}\partial x^{\alpha}}\text{ is
differentiable at }\left(  0,x\right)  \text{ for }m+\left\vert \alpha
\right\vert =k.
\]
If $m\geq1,$ this fact is contained in Corollary \ref{Corollary 1}, hence we
have to prove that%
\[
\frac{\partial^{\alpha}F}{\partial x^{\alpha}}\text{ is differentiable at
}\left(  0,x\right)  \text{ for }\left\vert \alpha\right\vert =k.
\]
Let us write%
\begin{align*}
&  \frac{\partial^{\alpha}F}{\partial x^{\alpha}}\left(  t,x+h\right)
-\frac{\partial^{\alpha}F}{\partial x^{\alpha}}\left(  0,x\right)  =\\
&  =\left[  \frac{\partial^{\alpha}F}{\partial x^{\alpha}}\left(
t,x+h\right)  -\frac{\partial^{\alpha}F}{\partial x^{\alpha}}\left(
0,x+h\right)  \right]  +\left[  \frac{\partial^{\alpha}F}{\partial x^{\alpha}%
}\left(  0,x+h\right)  -\frac{\partial^{\alpha}F}{\partial x^{\alpha}}\left(
0,x\right)  \right] \\
&  \equiv A+B.
\end{align*}
By Corollary \ref{Corollary 1}, $\frac{\partial}{\partial t}\left(
\frac{\partial^{\alpha}F}{\partial x^{\alpha}}\right)  $ is continuous, hence
for some $\tau\in\left(  0,t\right)  $ we have%
\[
A=t\frac{\partial^{\alpha+1}F}{\partial t\partial x^{\alpha}}\left(
\tau,x+h\right)  =t\left[  \frac{\partial^{\alpha+1}F}{\partial t\partial
x^{\alpha}}\left(  0,x\right)  +o\left(  1\right)  \right]  \text{ for
}\left(  t,h\right)  \rightarrow0.
\]
On the other hand, $F\left(  0,x\right)  =x$ for every $x,$ hence%
\[
B=\frac{\partial^{\alpha}}{\partial x^{\alpha}}\left(  x+h-x\right)  =0
\]
so%
\[
\frac{\partial^{\alpha}F}{\partial x^{\alpha}}\left(  t,x+h\right)
-\frac{\partial^{\alpha}F}{\partial x^{\alpha}}\left(  0,x\right)
=t\frac{\partial^{\alpha+1}F}{\partial t\partial x^{\alpha}}\left(
0,x\right)  +o\left(  \sqrt{t^{2}+h^{2}}\right)  ,
\]
and $\frac{\partial^{\alpha}F}{\partial x^{\alpha}}$ is differentiable at
$\left(  0,x\right)  .$
\end{proof}

In order to apply the previous results to quasiexponential maps, it is
convenient to express the following step in an abstract way:

\begin{proposition}
\label{Proposition F(t)}For some positive integer $l,$ let us consider a
family of functions%
\[
F\left(  t_{1},t_{2},...,t_{l}\right)  \left(  \cdot\right)
\]
defined in a neighborhood of $x_{0}$, with values in $\mathbb{R}^{p},$
depending on $l$ scalar parameters $t_{1},t_{2},...,t_{l}$, ranging in a
neighborhood of $0,$ in such a way that:%
\[
F\left(  t_{1},t_{2},...,t_{l}\right)  \left(  x\right)  =x
\]
as soon as at least one of the $t_{j}$ is equal to $0,$ and%
\[
\left(  t_{1},t_{2},...,t_{l},x\right)  \mapsto F\left(  t_{1},t_{2}%
,...,t_{l}\right)  \left(  x\right)
\]
is $C^{l-1}$ in a neighborhood of $\left(  0,0,...,0,x\right)  $ and
differentiable $l$ times at $\left(  0,0,...,0,x\right)  .$ Then the following
expansion holds:%
\[
F\left(  t_{1},t_{2},...,t_{l}\right)  \left(  x\right)  =x+t_{1}t_{2}%
...t_{l}\frac{\partial^{l}F}{\partial t_{1}\partial t_{2}...\partial t_{l}%
}\left(  0,0,...,0\right)  \left(  x\right)  +o\left(  t_{1}t_{2}%
...t_{l}\right)
\]
as $\left(  t_{1},t_{2},...,t_{l}\right)  \rightarrow0.$
\end{proposition}

\begin{proof}
Since $F\left(  t_{1},t_{2},...,t_{l}\right)  \left(  x\right)  =x$ if
$t_{i}=0$ for some $i,$ then%
\[
\frac{\partial F}{\partial t_{j}}\left(  t_{1},t_{2},...,t_{l}\right)  \left(
x\right)  =0\text{ if }t_{i}=0\text{ for some }i\neq j.
\]
Then we can write, since $F$ is $C^{l-1},$
\begin{align}
F\left(  t_{1},t_{2},...,t_{l}\right)  \left(  x\right)   &  =x+\int
_{0}^{t_{1}}\frac{\partial F}{\partial t_{1}}\left(  u_{1},t_{2}%
,...,t_{l}\right)  du_{1}\nonumber\\
&  =x+\int_{0}^{t_{1}}\left[  \frac{\partial F}{\partial t_{1}}\left(
u_{1},t_{2},...,t_{l}\right)  -\frac{\partial F}{\partial t_{1}}\left(
u_{1},0,t_{3},...,t_{l}\right)  \right]  du_{1}\nonumber\\
&  =x+\int_{0}^{t_{1}}\int_{0}^{t_{2}}\frac{\partial^{2}F}{\partial
t_{1}\partial t_{2}}\left(  u_{1},u_{2},t_{3}...,t_{l}\right)  du_{2}%
du_{1}\nonumber\\
&  =...\nonumber\\
&  =x+\int_{0}^{t_{1}}...\int_{0}^{t_{l-1}}\frac{\partial^{l-1}F}{\partial
t_{1}\partial t_{2}...\partial t_{l-1}}\left(  u_{1},u_{2},...,u_{l-1}%
,t_{l}\right)  du_{l-1}...du_{1}. \label{F(t)}%
\end{align}
By assumption, $\frac{\partial^{l-1}F}{\partial t_{1}\partial t_{2}...\partial
t_{l-1}}$ is differentiable at $\left(  0,0,...,0\right)  \left(  x\right)  ,$
and
\[
\frac{\partial}{\partial t_{j}}\frac{\partial^{l-1}F}{\partial t_{1}\partial
t_{2}...\partial t_{l-1}}\left(  0,0,...,0\right)  \left(  x\right)  =0\text{
for any }j\neq l,
\]
hence the last expression in (\ref{F(t)}) equals%
\[
x+\int_{0}^{t_{1}}...\int_{0}^{t_{l-1}}\left[  t_{l}\frac{\partial^{l-1}%
F}{\partial t_{1}\partial t_{2}...\partial t_{l-1}}\left(  0,0,...,0\right)
+o\left(  \sqrt{u_{1}^{2}+...u_{l-1}^{2}+t_{l}^{2}}\right)  \right]
du_{l-1}...du_{1}.
\]
However, performing if necessary the integration with respect to the variables
$u_{i}$ in a different order, we can always assume that $t_{l}\geq\max\left(
t_{1},t_{2},...,t_{l-1}\right)  ,$ so that $o\left(  \sqrt{u_{1}%
^{2}+...u_{l-1}^{2}+t_{l}^{2}}\right)  =o\left(  t_{l}\right)  $ and we get%
\begin{align*}
&  F\left(  t_{1},t_{2},...,t_{l}\right)  \left(  x\right) \\
&  =x+\int_{0}^{t_{1}}...\int_{0}^{t_{l-1}}\left[  t_{l}\frac{\partial^{l-1}%
F}{\partial t_{1}\partial t_{2}...\partial t_{l-1}}\left(  0,0,...,0\right)
+o\left(  t_{l}\right)  \right]  du_{l-1}...du_{1}\\
&  =x+t_{1}t_{2}...t_{l}\frac{\partial^{l}F}{\partial t_{1}\partial
t_{2}...\partial t_{l}}\left(  0,0,...,0\right)  \left(  x\right)  +o\left(
t_{1}t_{2}...t_{l}\right)  .
\end{align*}

\end{proof}

We now come back to our vector fields. For fixed $\ell\leq r$ and $X_{i_{1}%
},X_{i_{2}},...,X_{i_{\ell}},$ with $\left\{  i_{1},i_{2},...,i_{\ell
}\right\}  \subset\left\{  0,1,2,...,n\right\}  ,$ let us define recursively
the following maps:
\begin{align*}
\mathcal{C}_{1}\left(  t_{1}\right)  \left(  x\right)   &  =\exp\left(
t_{1}X_{i_{1}}\right)  \left(  x\right) \\
\mathcal{C}_{2}\left(  t_{1},t_{2}\right)  \left(  x\right)   &  =\exp\left(
-t_{2}X_{i_{2}}\right)  \exp\left(  -t_{1}X_{i_{1}}\right)  \exp\left(
t_{2}X_{i_{2}}\right)  \exp\left(  t_{1}X_{i_{1}}\right)  \left(  x\right) \\
&  \vdots\\
\mathcal{C}_{l}\left(  t_{1},\ldots t_{l}\right)  \left(  x\right)   &
=\mathcal{C}_{l-1}\left(  t_{2},\ldots,t_{l}\right)  ^{-1}\exp\left(
-t_{1}X_{1}\right)  \mathcal{C}_{\ell-1}\left(  t_{2},\ldots,t_{l}\right)
\exp\left(  t_{1}X_{1}\right)  \left(  x\right)
\end{align*}

Note that $\mathcal{C}_{\ell}\left(  t^{p_{i_{1}}},t^{p_{i_{2}}}%
,\ldots,t^{p_{i_{l}}}\right)  $ coincides with the map $C_{l}\left(
t,X_{i_{1}}X_{i_{2}}...X_{i_{l}}\right)  $ previously defined. The previous
Proposition implies the following:

\begin{theorem}
\label{Thm Luca 1}For any multiindex $I$ with $\left\vert I\right\vert \leq
r,$ $l=\ell\left(  I\right)  $ we have%
\begin{equation}
\mathcal{C}_{l}\left(  t_{1},\ldots t_{l}\right)  \left(  x\right)
=x+t_{1}t_{2}\ldots t_{l}\frac{\partial^{l}\mathcal{C}_{l}\left(
0,0,...,0\right)  }{\partial t_{1}\partial t_{2}\cdots\partial t_{l}}\left(
x\right)  +o\left(  t_{1}t_{2}\ldots t_{l}\right)  \label{Luca 0}%
\end{equation}
as $\left(  t_{1},t_{2},\ldots,t_{l}\right)  \rightarrow0$. In particular:%
\begin{equation}
C_{l}\left(  t,X_{i_{1}}X_{i_{2}}...X_{i_{l}}\right)  \left(  x\right)
=x+t^{\left\vert I\right\vert }\frac{\partial^{l}\mathcal{C}_{l}\left(
0,0,...,0\right)  }{\partial t_{1}\partial t_{2}\cdots\partial t_{l}}\left(
x\right)  +o\left(  t^{\left\vert I\right\vert }\right)  \label{Luca 1}%
\end{equation}
as $t\rightarrow0$, where the symbol $o\left(  \cdot\right)  $ has the meaning
explained in Theorem \ref{Thm map C_el}.
\end{theorem}

\begin{proof}
We start noting that $\mathcal{C}_{l}\left(  t_{1},\ldots t_{l}\right)  $
reduces to the identity if at least one component of $\left(  t_{1},\ldots
t_{l}\right)  $ vanishes. This can be proved inductively, as follows. For
$l=1,$ this is just the identity $\exp\left(  0\right)  \left(  x\right)  =x$.
Assume this holds up to $l-1$. Then, if $t_{1}=0$ we have%
\begin{align*}
\mathcal{C}_{l}\left(  0,t_{2},\ldots t_{l}\right)  \left(  x\right)   &
=\mathcal{C}_{l-1}\left(  t_{2},\ldots,t_{l}\right)  ^{-1}\exp\left(
0\right)  \mathcal{C}_{l-1}\left(  t_{2},\ldots,t_{l}\right)  \exp\left(
0\right)  \left(  x\right) \\
&  =\mathcal{C}_{l-1}\left(  t_{2},\ldots,t_{l}\right)  ^{-1}\mathcal{C}%
_{l-1}\left(  t_{2},\ldots,t_{l}\right)  \left(  x\right)  =x
\end{align*}
On the other side, if $\left(  t_{2},\ldots,t_{l}\right)  $ has some component
that vanishes then%
\begin{align*}
&  \mathcal{C}_{l}\left(  t_{1},\ldots t_{l}\right)  \left(  x\right)  =\\
&  =\mathcal{C}_{l-1}\left(  t_{2},\ldots,t_{l}\right)  ^{-1}\exp\left(
-t_{1}X_{1}\right)  \mathcal{C}_{l-1}\left(  t_{2},\ldots,t_{l}\right)
\exp\left(  t_{1}X_{1}\right)  \left(  x\right) \\
&  =\mathbf{Id}\exp\left(  -t_{1}X_{1}\right)  \mathbf{Id}\exp\left(
t_{1}X_{1}\right)  \left(  x\right)  =\exp\left(  -t_{1}X_{1}\right)
\exp\left(  t_{1}X_{1}\right)  \left(  x\right)  =x.
\end{align*}

Now, assume first that the multiindex $I$ does not contain any $0;$ then
$\ell\left(  I\right)  =\left\vert I\right\vert \leq r;$ each vector field
$X_{i_{j}}$ is $C^{r-1},$ so that the function $\mathcal{C}_{l}\left(
t_{1},\ldots t_{l}\right)  \left(  x\right)  $ is $C^{r-1}$ in a neighborhood
of $\left(  x,0\right)  $ and, by Corollary \ref{Corollary differentiabke k+1}
$r$ times differentiable at $t=0$. Hence we can apply Proposition
\ref{Proposition F(t)} and conclude (\ref{Luca 0}).

If, instead, the multiindex $I$ contains some $0,$ since $X_{0}$ is just
$C^{r-2},$ by Corollary \ref{Corollary differentiabke k+1} we will conclude
that the function $\mathcal{C}_{l}\left(  t_{1},\ldots t_{l}\right)  \left(
x\right)  $ is only $r-1$ times differentiable at $t=0$. On the other hand, in
this case $\ell\left(  I\right)  \leq\left\vert I\right\vert -1\leq r-1$
(because $X_{0}$ has weight $2$), so the function $\mathcal{C}_{l}\left(
t_{1},\ldots t_{l}\right)  \left(  x\right)  $ is still $l$ times
differentiable at $t=0,$ and Proposition \ref{Proposition F(t)} still implies
(\ref{Luca 0}).

Finally, (\ref{Luca 0}) implies (\ref{Luca 1}) letting $t_{i}=t^{p_{i}}$ for
$i=1,2,...,l.$
\end{proof}

The identity (\ref{Luca 1}) in Theorem \ref{Thm Luca 1} will imply Theorem
\ref{Thm map C_el} as soon as we will prove the following:

\begin{theorem}
\label{Thm Luca 2}For any multiindex $I$ with $\left\vert I\right\vert \leq
r,$ $l=\ell\left(  I\right)  $ we have%
\begin{equation}
\frac{\partial^{l}\mathcal{C}_{l}\left(  0,0,...,0\right)  }{\partial
t_{1}\cdots\partial t_{l}}\left(  x\right)  =\left(  X_{\left[  I\right]
}\right)  _{x}. \label{commutatore}%
\end{equation}

\end{theorem}

In order to prove Theorem \ref{Thm Luca 2}, still another abstract Lemma is useful:

\begin{lemma}
\label{Lemma A B}Let $O$ be an open subset of $\mathbb{R}^{p}$ and
$A,B:\left(  -\varepsilon,\varepsilon\right)  \times O\rightarrow
\mathbb{R}^{p}$ be two $C^{1}$ funtions (for some $\varepsilon>0$); assume
that for every $x\in O$ we have $A\left(  0,x\right)  =x$ and $B\left(
0,x\right)  =x$. Since $\frac{\partial A}{\partial x}\left(  0,x\right)
=\mathbf{Id}$ it follows that for every $t$ sufficiently small $A\left(
t,\cdot\right)  $ is invertible. We denote with $A^{-1}\left(  t,x\right)  $
the inverse of this function. Similarly let $B^{-1}\left(  t,x\right)  $ the
inverse of $B\left(  t,\cdot\right)  .$ Let%
\[
F\left(  t,s,x\right)  =A^{-1}\left(  t,B^{-1}\left(  s,A\left(  t,B\left(
s,x\right)  \right)  \right)  \right)  .
\]
Assume that for every $x\in O$, $A$ and $B$ are two times differentiable at
$\left(  0,x\right)  $. Then%
\[
\frac{\partial^{2}F}{\partial t\partial s}\left(  0,0,x\right)  =\frac
{\partial^{2}A}{\partial x\partial t}\left(  0,x\right)  \frac{\partial
B}{\partial s}\left(  0,x\right)  -\frac{\partial^{2}B}{\partial s\partial
x}\left(  0,x\right)  \frac{\partial A}{\partial t}\left(  0,x\right)  .
\]
\qquad
\end{lemma}

We have used a compact matrix notation, where for instance $\frac{\partial
^{2}A}{\partial x\partial t}\left(  0,x\right)  $ stands for the Jacobian of
the map%
\[
x\mapsto\frac{\partial A}{\partial t}\left(  0,x\right)  .
\]

\begin{proof}
Since $\frac{\partial A}{\partial x}\left(  0,x\right)  =\mathbf{Id}$ we have
$\frac{\partial^{2}A}{\partial x_{i}\partial x_{j}}\left(  0,x\right)  =0$ and
the same holds with $A$ replaced by $A^{-1}$, $B$ and $B^{-1}$. Then%
\begin{align*}
&  \frac{\partial F}{\partial t}\left(  0,s,x\right)  =\frac{\partial A^{-1}%
}{\partial t}\left(  0,B^{-1}\left(  s,A\left(  0,B\left(  s,x\right)
\right)  \right)  \right) \\
&  +\frac{\partial A^{-1}}{\partial x}\left(  0,B^{-1}\left(  s,A\left(
t,B\left(  s,x\right)  \right)  \right)  \right)  \frac{\partial B^{-1}%
}{\partial x}\left(  s,A\left(  0,B\left(  s,x\right)  \right)  \right)
\frac{\partial A}{\partial t}\left(  0,B\left(  s,x\right)  \right)  .
\end{align*}
Since $B^{-1}\left(  s,A\left(  0,B\left(  s,x\right)  \right)  \right)
=B^{-1}\left(  s,B\left(  s,x\right)  \right)  =x$ and $\frac{\partial A^{-1}%
}{\partial x}\left(  0,x\right)  =\mathbf{Id}$ the above equation reduces to%
\[
\frac{\partial F}{\partial t}\left(  0,s,x\right)  =\frac{\partial A^{-1}%
}{\partial t}\left(  0,x\right)  +\frac{\partial B^{-1}}{\partial x}\left(
s,B\left(  s,x\right)  \right)  \frac{\partial A}{\partial t}\left(
0,B\left(  s,x\right)  \right)  .
\]

Let us compute%
\begin{align*}
\frac{\partial F\left(  0,0,x\right)  }{\partial t\partial s}  &  =\left[
\frac{\partial B^{-1}}{\partial s\partial x}\left(  0,B\left(  0,x\right)
\right)  +\frac{\partial^{2}B^{-1}}{\partial x^{2}}\left(  0,B\left(
0,x\right)  \right)  \frac{\partial B}{\partial s}\left(  0,x\right)  \right]
\frac{\partial A}{\partial t}\left(  0,B\left(  0,x\right)  \right) \\
&  +\frac{\partial B^{-1}}{\partial x}\left(  0,B\left(  0,x\right)  \right)
\frac{\partial^{2}A}{\partial x\partial t}\left(  0,B\left(  0,x\right)
\right)  \frac{\partial B}{\partial s}\left(  0,x\right)
\end{align*}

Since $\frac{\partial B^{-1}}{\partial x}\left(  0,B\left(  0,x\right)
\right)  =\mathbf{Id}$ and $\frac{\partial^{2}B^{-1}}{\partial x^{2}}\left(
0,B\left(  0,x\right)  \right)  =0$ we have%
\[
\frac{\partial F\left(  0,0,x\right)  }{\partial t\partial s}=\frac
{\partial^{2}B^{-1}}{\partial s\partial x}\left(  0,x\right)  \frac{\partial
A}{\partial t}\left(  0,x\right)  +\frac{\partial^{2}A}{\partial x\partial
t}\left(  0,x\right)  \frac{\partial B}{\partial s}\left(  0,x\right)  .
\]

Finally since $B^{-1}\left(  s,B\left(  s,x\right)  \right)  =x$ a simple
computation shows that $\frac{\partial B^{-1}}{\partial s}\left(  0,x\right)
=-\frac{\partial B}{\partial s}\left(  0,x\right)  $ and therefore%
\[
\frac{\partial F\left(  0,0,x\right)  }{\partial t\partial s}=-\frac
{\partial^{2}B}{\partial s\partial x}\left(  0,x\right)  \frac{\partial
A}{\partial t}\left(  0,x\right)  +\frac{\partial^{2}A}{\partial x\partial
t}\left(  0,x\right)  \frac{\partial B}{\partial s}\left(  0,x\right)  .
\]

\end{proof}

\begin{proof}
[Proof of Theorem \ref{Thm Luca 2}]We prove the theorem by induction on $l$.
For $l=1$ the Theorem is trivial:%
\[
\frac{\partial\mathcal{C}_{1}\left(  0\right)  \left(  x\right)  }{\partial
t}=\frac{\partial}{\partial t}\exp\left(  tX_{i}\right)  _{/t=0}\left(
x\right)  =\left(  X_{i}\right)  _{x}.
\]

Assume the theorem holds for $l-1$ and let us prove it for $l\geq2$. Let%
\[
A\left(  t_{2},\ldots t_{l},x\right)  =C_{\ell-1}\left(  t_{2},\ldots
t_{l}\right)  \left(  x\right)
\]
and%
\[
B\left(  t_{1},x\right)  =\exp\left(  t_{1}X_{i_{1}}\right)  \left(  x\right)
.
\]
In order to apply Lemma \ref{Lemma A B}, we have to check that $A,B$ are
$C^{1}$ and twice differentiable at $t=0.$ Let us distinguish the following cases:

a) The multiindex $I$ does not contain any $0.$ Then all the $X_{i_{j}}$ are
$C^{r-1},$ that is at least $C^{1},$ hence $A,B$ are $C^{1}$ and, by Corollary
\ref{Corollary differentiabke k+1}, twice differentiable at $t=0.$

b) The multiindex $I$ contains at least a $0.$ Then, since we are commuting at
least two vector fields, at least one of which is $X_{0}$, we have that
$r\geq3.$ Therefore all the $X_{i}$'s are at least $C^{1},$ hence $A,B$ are
$C^{1}$ and, by Corollary \ref{Corollary differentiabke k+1}, twice
differentiable at $t=0.$

We can then apply Lemma \ref{Lemma A B} to $A,B,$ with respect to the
variables $t_{1},t_{2}$ (regarding $t_{3},..,t_{l}$ as parameters),
obtaining:
\begin{align*}
&  \frac{\partial^{2}C_{l}}{\partial t_{1}\partial t_{2}}\left(
0,0,t_{3},\ldots,t_{l}\right)  \left(  x\right) \\
&  =-\frac{\partial^{2}B}{\partial t_{1}\partial x}\left(  0,x\right)
\frac{\partial A}{\partial t_{2}}\left(  0,x\right)  +\frac{\partial^{2}%
A}{\partial x\partial t_{2}}\left(  0,x\right)  \frac{\partial B}{\partial
t_{1}}\left(  0,x\right) \\
&  =-\frac{\partial X_{i_{1}}}{\partial x}\left(  x\right)  \frac{\partial
C_{l-1}\left(  0,t_{3},\ldots,t_{\ell}\right)  }{\partial t_{2}}\left(
x\right)  +\frac{\partial}{\partial x}\frac{\partial C_{l-1}\left(
0,t_{3},\ldots,t_{\ell}\right)  }{\partial t_{2}}\left(  x\right)  X_{i_{1}%
}\left(  x\right)
\end{align*}

We can now compute the remaining $\ell-2$ derivatives in $0$ (observe that by
Theorem \ref{Thm Luca 1} we already know that we can compute $r$ derivatives
of $C_{l-1}$ at $t=0$). This yields%
\begin{align*}
&  \frac{\partial^{l}C_{\ell}}{\partial t_{1}\cdots\partial t_{l}}\left(
0,\ldots,0\right)  =\\
=  &  -\frac{\partial X_{i_{1}}}{\partial x}\left(  x\right)  \frac{\partial
C_{l-1}\left(  0,\ldots,0\right)  }{\partial t_{2}\cdots\partial t_{l}}\left(
x\right)  +\frac{\partial}{\partial x}\frac{\partial C_{l-1}\left(
0,\ldots,0\right)  }{\partial t_{2}\cdots\partial t_{l}}\left(  x\right)
X_{i_{1}}\left(  x\right)  .
\end{align*}
Since by inductive assumption we have $\frac{\partial C_{l-1}\left(
0,\ldots,0\right)  }{\partial t_{2}\cdots\partial t_{l}}\left(  x\right)
=\left[  X_{i_{2}},\ldots\left[  X_{i_{\ell-1}},X_{i_{\ell}}\right]  \right]
_{x}$ the Theorem follows.
\end{proof}

As already noted, from Theorem \ref{Thm Luca 1} and Theorem \ref{Thm Luca 2},
Theorem \ref{Thm map C_el} follows.

\section{Connectivity and equivalent distances\label{section connectivity}}

In this section we have to further strengthen our assumption on $X_{0}$:

\textbf{Assumptions (C). }We keep assumptions (A) but, in the case $r=2,$ we
also require $X_{0}$ to have $C^{1}$ coefficients (instead of merely
continuous, as in \S \ref{section subelliptic metric} or Lipschitz continuous,
as in \S \ref{section exp and quasiexp}) .

Let us define the maps:%

\[
E_{I}\left(  t\right)  =\left\{
\begin{array}
[c]{l}%
C_{\ell\left(  I\right)  }\left(  t^{1/\left\vert I\right\vert },X_{I}\right)
\text{ if }t\geq0\\
C_{\ell\left(  I\right)  }\left(  \left\vert t\right\vert ^{1/\left\vert
I\right\vert },X_{I}\right)  ^{-1}\text{ if }t<0
\end{array}
\right.  .
\]
for any $I\in\eta$ (where $\eta$ is like in Theorem \ref{Thm map C_el}). By
Theorem \ref{Thm map C_el}, the following expansion holds:%
\begin{equation}
E_{I}\left(  t\right)  \left(  x\right)  =x+t\left(  X_{\left[  I\right]
}\right)  _{x}+o\left(  t\right)  \text{ as }t\rightarrow0. \label{sviluppo E}%
\end{equation}

We are now in position to state the main result of this section:

\begin{theorem}
\label{Thm diffeomorfism}Let $\Omega^{\prime}\Subset\Omega,$ $x_{0}\in
\Omega^{\prime}$ and let $\left\{  X_{\left[  I_{j}\right]  }\right\}
_{I_{j}\in\eta}$ be any family of $p$ commutators (with $\left\vert
I_{j}\right\vert \leq r$) which span $\mathbb{R}^{p}$ at $x_{0},$ satisfying%
\begin{equation}
\left\vert \det\left\{  \left(  X_{\left[  I_{j}\right]  }\right)  _{x_{0}%
}\right\}  _{I_{j}\in\eta}\right\vert \geq\left(  1-\varepsilon\right)
\max_{\zeta}\left\vert \det\left\{  \left(  X_{\left[  I_{j}\right]  }\right)
_{x_{0}}\right\}  _{I_{j}\in\zeta}\right\vert \label{eta opportuno}%
\end{equation}
for some $\varepsilon\in\left(  0,1\right)  .$ Then there exist constants
$\delta_{1},\delta_{2}>0$, depending on $\Omega^{\prime},\varepsilon$ and the
$X_{i}$'s, such that the map%
\[
\left(  h_{1},h_{2},...,h_{p}\right)  \mapsto E_{I_{1}}\left(  h_{1}\right)
E_{I_{2}}\left(  h_{2}\right)  ...E_{I_{p}}\left(  h_{p}\right)  \left(
x_{0}\right)
\]
is a $C^{1}$ diffeomorphism of a neighborhood of the origin $\left\{
h:\left\vert h\right\vert <\delta_{1}\right\}  $ onto a neighborhood $U\left(
x_{0}\right)  $ of $x_{0}$ containing $\left\{  x:\left\vert x-x_{0}%
\right\vert <\delta_{2}\right\}  .$ It is also a $C^{1}$ map in the joint
variables $h_{1},h_{2},...,h_{p},x$ for $x\in\Omega^{\prime}$ and $\left\vert
h\right\vert <\delta_{1}.$
\end{theorem}

To stress the dependence of this diffeomorphism on the system of vector fields
$\left\{  X_{i}\right\}  $, the choice of the basis $\eta,$ and the point $x,$
we will write%
\[
E_{\eta}^{X}\left(  x,h\right)  =E_{I_{1}}\left(  h_{1}\right)  E_{I_{2}%
}\left(  h_{2}\right)  \cdots E_{I_{p}}\left(  h_{p}\right)  \left(  x\right)
.
\]

\begin{proof}
First of all, let us check that the map%
\begin{equation}
\left(  x,h\right)  \longmapsto E_{I_{1}}\left(  h_{1}\right)  E_{I_{2}%
}\left(  h_{2}\right)  \cdots E_{I_{p}}\left(  h_{p}\right)  \left(  x\right)
\label{map}%
\end{equation}
is of class $C^{1}$ for $x\in\overline{\Omega^{\prime}}$ and $\left\vert
h\right\vert \leq\delta,$ for some $\delta>0$. This will follow, by
composition, if we prove that for any multiindex $I$, the map%
\[
\left(  t,x\right)  \mapsto E_{I}\left(  t\right)  \left(  x\right)
\]
is $C^{1}$.

Assume first that the multiindex $I$ does not contain any $0,$ so that each
vector field which enters the definition of $E_{I}\left(  t\right)  \left(
x\right)  $ is of class $C^{r-1}.$ Then, by Corollary \ref{Corollary 1}, and
Corollary \ref{Corollary differentiabke k+1} we know that each function%
\[
\left(  t,x\right)  \longmapsto C_{I}\left(  t\right)  \left(  x\right)
\]
is $C^{r-1}$ in the joint variables, for $t$ in a neighborhood of the origin
and $x$ in a fixed neighborhood of some $x_{0}$, and differentiable $r$ times
at $\left(  0,x\right)  .$ By composition, the map $\left(  t,x\right)
\mapsto E_{I}\left(  t\right)  \left(  x\right)  $ is $C^{1}$ in $x$, and has
continuous $t$ derivative for $t\neq0.$ Moreover, the expansion
(\ref{sviluppo E}) shows that there exists%
\begin{equation}
\frac{\partial E_{I}\left(  0\right)  \left(  x\right)  }{\partial t}=\left(
X_{\left[  I\right]  }\right)  _{x}. \label{DE/Dt}%
\end{equation}
It remains to prove that
\begin{equation}
\frac{\partial E_{I}\left(  t\right)  \left(  x\right)  }{\partial
t}\rightarrow\left(  X_{\left[  I\right]  }\right)  _{x}\text{ for
}t\rightarrow0. \label{cont 0}%
\end{equation}
Since $C_{I}\left(  t\right)  \left(  x\right)  $ is differentiable $r$ times
at $t=0$ (and $r\geq\left\vert I\right\vert $), the expansion of $C_{I}\left(
t\right)  \left(  x\right)  $ given by Theorem \ref{Thm map C_el} also says
that%
\[
\frac{\partial C_{I}\left(  t\right)  \left(  x\right)  }{\partial
t}=\left\vert I\right\vert t^{\left\vert I\right\vert -1}\left(  X_{\left[
I\right]  }\right)  _{x}+o\left(  t^{\left\vert I\right\vert -1}\right)
\]
Then we can compute:%
\begin{align*}
&  \lim_{t\rightarrow0}\frac{\partial E_{I}\left(  t\right)  \left(  x\right)
}{\partial t}=\lim_{t\rightarrow0}\frac{1}{\left\vert I\right\vert
t^{1-1/\left\vert I\right\vert }}\frac{\partial C_{I}\left(  t^{1/\left\vert
I\right\vert }\right)  \left(  x\right)  }{\partial t}=\\
&  =\lim_{t\rightarrow0}\frac{1}{\left\vert I\right\vert t^{\left\vert
I\right\vert -1}}\frac{\partial C_{I}\left(  t\right)  \left(  x\right)
}{\partial t}=\lim_{t\rightarrow0}\frac{1}{\left\vert I\right\vert
t^{\left\vert I\right\vert -1}}\left(  \left\vert I\right\vert t^{\left\vert
I\right\vert -1}\left(  X_{\left[  I\right]  }\right)  _{x}+o\left(
t^{\left\vert I\right\vert -1}\right)  \right)  =\left(  X_{\left[  I\right]
}\right)  _{x}%
\end{align*}
and this allows to conclude that the map (\ref{map}) is $C^{1}$.

Let us now consider the case when the multiindex $I$ also contains some $0,$
so that the vector field $X_{0},$ of class $C^{r-2},$ enters the definition of
$E_{I}\left(  t\right)  \left(  x\right)  .$ This case requires a more careful
inspection. First of all, by our Assumptions (C) all the $X_{i}$'s $\left(
i=0,1,...,n\right)  $ are at least $C^{1},$ and the function $E_{I}\left(
t\right)  \left(  x\right)  $ is $C^{1}$ in the joint variables for $t\neq0.$
Again, (\ref{DE/Dt}) is in force, and we are reduced to proving (\ref{cont 0}%
). Assume $\left\vert I\right\vert =r$ (the case $\left\vert I\right\vert <r$
is easier). Since $X_{0}$ has weight two, we have $l=\ell\left(  I\right)  <r$.

Let us consider the function $\mathcal{C}_{l}\left(  t_{1},\ldots
,t_{l}\right)  \left(  x\right)  $ introduced in the previous section; we
have:%
\begin{equation}
g\left(  t\right)  \equiv E_{I}\left(  t\right)  \left(  x\right)
=\mathcal{C}_{l}\left(  t^{p_{k_{1}}/r},\ldots,t^{p_{k_{l}}/r}\right)  \left(
x\right)  \label{g(t)}%
\end{equation}
where each $p_{k_{i}}$ is the weight of a vector field ($p_{k_{i}}=1$ or $2$),
and $p_{k_{1}}+p_{k_{2}}+...+p_{k_{l}}=r.$ Our goal consists in proving that
$g^{\prime}\left(  t\right)  \rightarrow g^{\prime}\left(  0\right)  $ as
$t\rightarrow0.$

\begin{claim}
The function%
\[
\left(  t_{1},\ldots,t_{l},x\right)  \mapsto\mathcal{C}_{l}\left(
t_{1},\ldots t_{l}\right)  \left(  x\right)
\]
is $C^{l-1},$ and is $l$ times differentiable at any point $\left(
t_{1},\ldots,t_{l},x\right)  $ such that $t_{j}=0$ if $p_{k_{j}}=2.$
\end{claim}

\begin{proof}
[Proof of the Claim]We know that%
\[
\mathcal{C}_{l}\left(  t_{1},\ldots t_{l}\right)  \left(  x\right)
=\prod\limits_{j=1}^{N}\exp\left(  \pm t_{k_{j}}X_{k_{j}}\right)  \left(
x\right)  .
\]

If $p_{k_{j}}=1,$ then $X_{k_{j}}\in C^{r-1}\subset C^{l}$ (since $l<r$), and
the function%
\[
\left(  t,x\right)  \mapsto\exp\left(  \pm tX_{k_{j}}\right)  \left(
x\right)
\]
is $l$ times differentiable at any point $\left(  t,x\right)  ;$

if $p_{k_{j}}=2,$ then $X_{k_{j}}\in C^{r-2}\subset C^{l-1}$ (since $l<r$),
and the function%
\[
\left(  t,x\right)  \mapsto\exp\left(  \pm tX_{k_{j}}\right)  \left(
x\right)
\]
is $l$ times differentiable at any point $\left(  0,x\right)  ,$ by Corollary
14. By composition, the Claim follows.
\end{proof}

\begin{claim}
The function%
\[
t_{h}\mapsto\frac{\partial^{l}\mathcal{C}_{l}}{\partial t_{1}...\partial
t_{l}}\left(  0,...,t_{h},...,0\right)  \left(  x\right)
\]
is continuous if $p_{i_{h}}=1.$
\end{claim}

\begin{proof}
[Proof of the Claim]By the previous claim, this derivative actually exists; we
have to prove its continuity. Indeed, one can easily check by induction that
the derivative
\[
\frac{\partial^{l}\mathcal{C}_{l}}{\partial t_{1}...\partial t_{l}}\left(
0,...,t_{h},...,0\right)  \left(  x\right)
\]
is a polynomial in variables of the form%
\begin{equation}
\frac{\partial^{\left\vert \alpha\right\vert }}{\partial x^{\alpha}}\left(
\exp\left(  \pm t_{i}X_{i}\right)  \right)  \left(  \prod_{j}\exp\left(  \pm
t_{k_{j}}X_{k_{j}}\right)  \left(  x\right)  \right)  \label{primotipo}%
\end{equation}
with $\left\vert \alpha\right\vert \leqslant\ell$ and%
\begin{equation}
\frac{\partial^{\left\vert \alpha\right\vert +1}}{\partial t_{i}\partial
x^{\alpha}}\left(  \exp\left(  \pm t_{i}X_{i}\right)  \right)  \left(
\prod_{j}\exp\left(  \pm t_{k_{j}}X_{k_{j}}\right)  \left(  x\right)  \right)
=\pm\frac{\partial^{\left\vert \alpha\right\vert }}{\partial x^{\alpha}}%
X_{i}\left(  \prod_{j}\exp\left(  \pm t_{k_{j}}X_{k_{j}}\right)  \left(
x\right)  \right)  \label{secondotipo}%
\end{equation}
with $\left\vert \alpha\right\vert \leqslant\ell-1$. These derivatives should
be evaluated for $t_{i}=0$ when $i\neq h$.

Let us consider the derivatives of the form (\ref{primotipo}) and assume first
$i\neq h$ so that $t_{i}=0$. In this case the map $\exp\left(  \pm t_{i}%
X_{i}\right)  $ reduces to the identity and the derivative is obsiously
continuous. When $i=h$ the continuity follows from the fact that $X_{h}\in
C^{\ell}$.

The continuity of the derivatives of the form (\ref{secondotipo}) follows from
the fact that in this case $\left\vert \alpha\right\vert \leqslant\ell-1$ and
$X_{i}\in C^{\ell-1}$.
\end{proof}

By Proposition \ref{Proposition F(t)}, we know that%
\[
\mathcal{C}_{l}\left(  t_{1},t_{2},...,t_{l}\right)  \left(  x\right)
=x+t_{1}t_{2}...t_{l}\frac{\partial^{l}\mathcal{C}_{l}}{\partial t_{1}\partial
t_{2}...\partial t_{l}}\left(  0,0,...,0\right)  \left(  x\right)  +o\left(
t_{1}t_{2}...t_{l}\right)  .
\]
Assume for a moment that we can prove an analogous expansion for first order
derivatives of $\mathcal{C}_{l},$ namely%
\begin{equation}
\frac{\partial\mathcal{C}_{l}\left(  t_{1},\ldots t_{l}\right)  \left(
x\right)  }{\partial t_{1}}=t_{2}t_{3}...t_{l}\frac{\partial^{l}%
\mathcal{C}_{l}}{\partial t_{1}\partial t_{2}...\partial t_{l-1}\partial
t_{l}}\left(  0,0,...,0\right)  \left(  x\right)  +o\left(  t_{2}t_{3}%
...t_{l}\right)  . \label{expansion DC_l}%
\end{equation}
Then we could easily conclude the proof as follows.

Let us compute%
\[
g^{\prime}\left(  t\right)  =\sum_{j=1}^{l}\frac{p_{k_{j}}}{r}t^{p_{k_{j}%
}/r-1}\frac{\partial\mathcal{C}_{l}\left(  t^{p_{k_{1}}/r},\ldots,t^{p_{k_{l}%
}/r}\right)  \left(  x\right)  }{\partial t_{j}}=
\]
applying (\ref{expansion DC_l}) to every $j$-derivative of $\mathcal{C}_{l}$%
\begin{align*}
&  =\sum_{j=1}^{l}\frac{p_{k_{j}}}{r}t^{p_{k_{j}}/r-1}\left[  \frac
{\partial\mathcal{C}_{l}\left(  0,0,...,0\right)  \left(  x\right)  }{\partial
t_{j}}t^{\sum_{i\neq j}p_{k_{j}}/r}+o\left(  t^{\sum_{i\neq j}p_{k_{j}}%
/r}\right)  \right] \\
&  =\sum_{j=1}^{l}\frac{p_{k_{j}}}{r}t^{p_{k_{j}}/r-1}\left[  \frac
{\partial\mathcal{C}_{l}\left(  0,0,...,0\right)  \left(  x\right)  }{\partial
t_{j}}t^{1-p_{k_{j}}/r}+o\left(  t^{1-p_{k_{j}}/r}\right)  \right] \\
&  =\frac{\partial\mathcal{C}_{l}\left(  0,0,...,0\right)  \left(  x\right)
}{\partial t_{j}}\sum_{j=1}^{l}\frac{p_{k_{j}}}{r}\left[  1+o\left(  1\right)
\right]  =\frac{\partial\mathcal{C}_{l}\left(  0,0,...,0\right)  \left(
x\right)  }{\partial t_{j}}+o\left(  1\right) \\
&  \rightarrow\frac{\partial\mathcal{C}_{l}\left(  0,0,...,0\right)  \left(
x\right)  }{\partial t_{j}}=g^{\prime}\left(  0\right)  \text{ as
}t\rightarrow0.
\end{align*}
So we are left to prove (\ref{expansion DC_l}). What we can actually prove is
a slightly less general assertion, which is enough to perform the above computation:

\begin{claim}
The expansion (\ref{expansion DC_l}) holds if%
\[
t_{i}=t^{p_{k_{i}}/r}\text{ for }i=1,2,...,l,\text{ and }t\rightarrow0.
\]

\end{claim}

\begin{proof}
[Proof of the Claim]As we have seen in the proof of Proposition
\ref{Proposition F(t)}, since $\mathcal{C}_{l}$ is $C^{l-1}$ we can write%
\[
\mathcal{C}_{l}\left(  t_{1},\ldots t_{l}\right)  \left(  x\right)
=x+\int_{0}^{t_{1}}...\int_{0}^{t_{l-1}}\frac{\partial^{l-1}\mathcal{C}_{l}%
}{\partial t_{1}\partial t_{2}...\partial t_{l-1}}\left(  u_{1},u_{2}%
,...,u_{l-1},t_{l}\right)  \left(  x\right)  du_{l-1}...du_{1}%
\]
and hence, differentiating with respect to $t_{1},$%
\begin{align}
&  \frac{\partial\mathcal{C}_{l}\left(  t_{1},\ldots t_{l}\right)  \left(
x\right)  }{\partial t_{1}}=\int_{0}^{t_{2}}...\int_{0}^{t_{l-1}}%
\frac{\partial^{l-1}\mathcal{C}_{l}}{\partial t_{1}\partial t_{2}...\partial
t_{l-1}}\left(  t_{1},u_{2},...,u_{l-1},t_{l}\right)  \left(  x\right)
du_{l-1}...du_{1}\label{dC/dt_l}\\
&  =\int_{0}^{t_{2}}...\int_{0}^{t_{l-1}}\left[  t_{l}\frac{\partial
^{l}\mathcal{C}_{l}}{\partial t_{1}\partial t_{2}...\partial t_{l-1}\partial
t_{l}}\left(  0,0,...,0\right)  \left(  x\right)  +\right.  \nonumber\\
&  ~~~~~~~~~~~~~~~\left.  +o\left(  \sqrt{t_{1}^{2}+u_{2}^{2}+...+u_{l-1}%
^{2}+t_{l}^{2}}\right)  \right]  du_{l-1}...du_{1}\nonumber\\
&  =t_{2}t_{3}...t_{l}\frac{\partial^{l}\mathcal{C}_{l}}{\partial
t_{1}\partial t_{2}...\partial t_{l-1}\partial t_{l}}\left(  0,0,...,0\right)
\left(  x\right)  +t_{2}t_{3}...t_{l-1}\cdot o\left(  \left\vert t\right\vert
\right)  .\nonumber
\end{align}
Now, if
\begin{equation}
\max_{j=1,...,l}\left\vert t_{j}\right\vert =\left\vert t_{i}\right\vert
\text{ with }i\neq1,\label{max t_i}%
\end{equation}
without loss of generality we can suppose that this maximum is assumed for
$i=l.$ In this case we can write%
\[
\frac{\partial\mathcal{C}_{l}\left(  t_{1},\ldots t_{l}\right)  \left(
x\right)  }{\partial t_{1}}=t_{2}t_{3}...t_{l}\frac{\partial^{l}%
\mathcal{C}_{l}}{\partial t_{1}\partial t_{2}...\partial t_{l-1}\partial
t_{l}}\left(  0,0,...,0\right)  \left(  x\right)  +o\left(  t_{2}t_{3}%
...t_{l}\right)  .
\]
Note that, for $t_{i}=t^{p_{k_{i}}/r}$ and $t\rightarrow0,$ condition
(\ref{max t_i}) just means $p_{k_{1}}=2.$

Assume, instead, that
\[
\max_{j=1,...,l}\left\vert t_{j}\right\vert =\left\vert t_{1}\right\vert
,\text{ that is }p_{k_{1}}=1.
\]
In this case, we start again with (\ref{dC/dt_l}) but now we exploit the fact
that $\frac{\partial^{l-1}\mathcal{C}_{l}}{\partial t_{1}\partial
t_{2}...\partial t_{l-1}}$ is differentiable at $\left(  t_{1},0,...,0\right)
\left(  x\right)  $ (see the previous Claim). Hence%
\begin{align*}
&  \frac{\partial\mathcal{C}_{l}\left(  t_{1},\ldots t_{l}\right)  \left(
x\right)  }{\partial t_{1}}=\int_{0}^{t_{2}}...\int_{0}^{t_{l-1}}%
\frac{\partial^{l-1}\mathcal{C}_{l}}{\partial t_{1}\partial t_{2}...\partial
t_{l-1}}\left(  t_{1},u_{2},...,u_{l-1},t_{l}\right)  \left(  x\right)
du_{l-1}...du_{1}\\
&  =\int_{0}^{t_{2}}...\left[  \int_{0}^{t_{l-1}}t_{l}\frac{\partial
^{l}\mathcal{C}_{l}}{\partial t_{1}\partial t_{2}...\partial t_{l-1}\partial
t_{l}}\left(  t_{1},0,...,0\right)  \left(  x\right)  \right. \\
&  \;\ \;\;\ \;\;\ \;\;\ \;\left.  +o\left(  \sqrt{u_{2}^{2}+...+u_{l-1}%
^{2}+t_{l}^{2}}\right)  \right]  du_{l-1}...du_{1}\\
&  =t_{2}t_{3}...t_{l}\frac{\partial^{l}\mathcal{C}_{l}}{\partial
t_{1}\partial t_{2}...\partial t_{l-1}\partial t_{l}}\left(  t_{1}%
,0,...,0\right)  \left(  x\right)  +o\left(  t_{2}...t_{l}\right)  .
\end{align*}
Finally, by the first Claim we have proved, $t_{1}\mapsto\frac{\partial
^{l}\mathcal{C}_{l}}{\partial t_{1}...\partial t_{l}}\left(  t_{1}%
,0,...,0\right)  \left(  x\right)  $ is continuous since $p_{i_{1}}=1$. Hence%
\begin{align*}
\frac{\partial\mathcal{C}_{l}\left(  t_{1},\ldots t_{l}\right)  \left(
x\right)  }{\partial t_{1}}  &  =t_{2}t_{3}...t_{l}\left[  \frac{\partial
^{l}\mathcal{C}_{l}}{\partial t_{1}...\partial t_{l}}\left(  0,0,...,0\right)
\left(  x\right)  +o\left(  1\right)  \right] \\
&  =t_{2}t_{3}...t_{l}\frac{\partial^{l}\mathcal{C}_{l}}{\partial
t_{1}...\partial t_{l}}\left(  0,0,...,0\right)  \left(  x\right)  +o\left(
t_{2}t_{3}...t_{l}\right)
\end{align*}
which completes the proof of the Claim.
\end{proof}

We have therefore proved that the map $\left(  t,x\right)  \mapsto
E_{I}\left(  t\right)  \left(  x\right)  $ is $C^{1}$. To prove that it is a
diffeomorphism from a neighborhood of the origin onto a neighborhood of
$x_{0}$, then, it will be enough to show that the Jacobian determinant of
$E_{\eta}^{X}\left(  x_{0},\cdot\right)  $ at the origin is nonzero. Let us
compute%
\begin{align*}
&  \frac{\partial}{\partial h_{i}}E_{I_{1}}\left(  h_{1}\right)  E_{I_{2}%
}\left(  h_{2}\right)  \cdots E_{I_{p}}\left(  h_{p}\right)  \left(
x_{0}\right)  _{/h=0}=\\
&  =E_{I_{1}}\left(  0\right)  E_{I_{2}}\left(  0\right)  \cdots\frac{\partial
E_{I_{i}}}{\partial h_{i}}\left(  0\right)  \cdots E_{I_{p}}\left(  0\right)
\left(  x_{0}\right)  =\left(  \frac{\partial E_{I_{i}}\left(  0\right)
}{\partial h_{i}}\right)  \left(  x_{0}\right)  =\left(  X_{\left[
I_{i}\right]  }\right)  _{x_{0}}%
\end{align*}

Hence the Jacobian of (\ref{map}) at zero is the matrix having as rows the
vectors $\left(  X_{\left[  I_{i}\right]  }\right)  _{x_{0}}$; since the
$\left(  X_{\left[  I_{i}\right]  }\right)  _{x_{0}}$ are a basis for
$\mathbb{R}^{p}$, the Jacobian is nonsingular.

Moreover, the same Jacobian is uniformly continuous for $x\in\overline
{\Omega^{\prime}},\left\vert h\right\vert \leq\delta$; therefore from the
standard proof of the inverse mapping theorem (see e.g. \cite[p.221]{Ru}) one
can see that our map is a diffeomorphism of a neighborhood of the origin
$\left\{  h:\left\vert h\right\vert <\delta_{1}\right\}  $ onto a neighborhood
$U\left(  x_{0}\right)  $ of $x_{0}$ containing $\left\{  x:\left\vert
x-x_{0}\right\vert <\delta_{2}\right\}  ,$ with $\delta_{1},\delta_{2}$
depending on the number $\varepsilon$ and the $X_{i}$'s.
\end{proof}

The above theorem has important consequences, the first of which is the following:

\begin{theorem}
[Chow's theorem for nonsmooth vector fields]\label{Thm Chow}
\end{theorem}

\begin{enumerate}
\item (Local statement of connectivity). For any $x_{0}\in\Omega$ there exist
two neighborhoods of $x_{0}$, $U\subset V\subset\Omega,$ such that any two
points of $U$ can be connected by a curve contained in $V,$ which is composed
by a finite number of arcs, integral curves of the vector fields $X_{i}$ for
$i=0,1,2,...,n.$

\item (Global statement of connectivity). If $\Omega$ is connected, for any
couple of points $x,y\in\Omega$ there exists a curve joining $x$ to $y$ and
contained in $\Omega,$ which is composed by a finite number of arcs, integral
curves of the vector fields $X_{i}$ for $i=0,1,2,...,n.$
\end{enumerate}

\begin{proof}
1. For any fixed $x_{0}\in\Omega,$ let $U\left(  x_{0}\right)  $ be a
neighborhood of $x_{0}$ where, by Theorem \ref{Thm diffeomorfism}, the
diffeomorphism $E_{\eta}^{X}\left(  x_{0},\cdot\right)  $ is well defined for
a suitable choice of $\eta$. More precisely, we can choose a neighborhood of
the kind%
\[
U_{\delta}\left(  x_{0}\right)  =\left\{  E_{\eta}^{X}\left(  x_{0},h\right)
:\left\vert h\right\vert <\delta\right\}
\]
(with $\delta$ small enough so that $U_{\delta}\left(  x_{0}\right)
\subset\Omega$). By definition of the map $E_{\eta}^{X}\left(  x_{0}%
,\cdot\right)  $, this means that every point of $U\left(  x_{0}\right)  $ can
be joined to $x_{0}$ with a curve which is composed by a finite number of
arcs, integral curves of the vector fields $X_{i}$ for $i=0,1,2,...,n$ with
coefficients of the order of $\delta^{1/r}$. Then we can also say that any two
points of $U\left(  x_{0}\right)  $ can be joined by a curve in a similar way.
Moreover, for each point $\gamma\left(  t\right)  $ of such a curve we have
\[
\left\vert \gamma\left(  t\right)  -x_{0}\right\vert \leq cd\left(
\gamma\left(  t\right)  ,x_{0}\right)  <c\delta^{1/r}.
\]
Let us choose $h$ small enough so that
\[
V_{\delta}\left(  x_{0}\right)  =\left\{  x:\left\vert x-x_{0}\right\vert
<c\delta^{1/r}\right\}  \subset\Omega;
\]
then we have the statement 1, choosing $U=U_{\delta}\left(  x_{0}\right)
,V=U\cup V_{\delta}\left(  x_{0}\right)  .$

2. Now we can cover any compact connected subset $\Omega^{\prime}$ of $\Omega$
with a finite number of neighborhoods $U\left(  x_{i}\right)  \subset V\left(
x_{i}\right)  \subset\Omega$ in such a way that any two points of
$\Omega^{\prime}$ can be joined by a curve as above, contained in the union of
the $V\left(  x_{i}\right)  $'s, and therefore in $\Omega$.
\end{proof}

Theorem \ref{Thm Chow} shows that it is possible to join any two points of
$\Omega$ using only integral lines of the vector fields $X_{i}.$ This
justifies the following:

\begin{definition}
For any $\delta>0,$ let $C_{1}\left(  \delta\right)  $ be the class of
absolutely continuous mappings $\varphi:\left[  0,1\right]  \longrightarrow
\Omega$ which satisfy%
\[
\varphi^{\prime}\left(  t\right)  =\sum_{i=0}^{n}a_{i}\left(  t\right)
\left(  X_{i}\right)  _{\varphi\left(  t\right)  }\text{ a.e.}%
\]
with
\[
\left\vert a_{0}\left(  t\right)  \right\vert \leq\delta^{2},\left\vert
a_{i}\left(  t\right)  \right\vert \leq\delta\text{ for }i=1,2,...,n.
\]
We define%
\[
d_{1}\left(  x,y\right)  =\inf\left\{  \delta>0:\exists\varphi\in C_{1}\left(
\delta\right)  \text{ with }\varphi\left(  0\right)  =x,\varphi\left(
1\right)  =y\right\}  .
\]

\end{definition}

\begin{remark}
\label{remark d-d1}In the smooth case, and for $X_{0}\equiv0$, the distance
$d_{1}$ has been introduced in \cite{NSW}. The Authors also prove the
equivalence of $d$ and $d_{1}$ (note that our $d_{1}$ is the distance called
$\rho_{4}$ in \cite{NSW}). One can easily check that this equivalence, in the
smooth case, still holds in presence of a vector field $X_{0}$ of weight 2.
\end{remark}

By the last theorem, the quantity $d_{1}\left(  x,y\right)  $ is finite for
every $x,y\in\Omega$. It is easy to see that $d_{1}$ is a distance (it is
still true that the union of two consecutive admissible curves can be
reparametrized to give an admissible curve) and, just by definition, one
always has
\[
d\left(  x,y\right)  \leq d_{1}\left(  x,y\right)  \text{.}%
\]
We also have the following:

\begin{proposition}
\label{Prop Fefferman Phong d1}For any $\Omega^{\prime}\Subset\Omega$ there
exist positive constants $c_{1},c_{2}$ such that%
\[
c_{1}\left\vert x-y\right\vert \leq d_{1}\left(  x,y\right)  \leq
c_{2}\left\vert x-y\right\vert ^{1/r}\text{ for any }x,y\in\Omega^{\prime}.
\]

\end{proposition}

\begin{proof}
The first inequality is obvious because we already know that $d$ satisfies it,
and $d\leq d_{1}$. So let us prove the second one.

Fix $x_{0}\in\Omega^{\prime}$, and let us consider the map $E_{\eta}%
^{X}\left(  x_{0},h\right)  $ defined in Theorem \ref{Thm diffeomorfism}, for
a suitable choice of $\eta$. Since
\[
h\longmapsto E_{\eta}^{X}\left(  x_{0},h\right)
\]
is a diffeomorphism, there exist positive constants $k_{1},k_{2}$ such that,
for $x=E_{\eta}^{X}\left(  x_{0},h\right)  $ in a suitable neighborhood of
$x_{0},$ we have:%
\[
k_{1}\left\vert x-x_{0}\right\vert \leq\max_{i=1,...,p}\left\vert
h_{i}\right\vert \leq k_{2}\left\vert x-x_{0}\right\vert .
\]
On the other hand, saying that $x=E_{\eta}^{X}\left(  x_{0},h\right)  ,$ by
definition means that there exists a curve $\gamma$ joining $x_{0}$ to $x$,
which is composed by a finite number $N$ (this number being under control) of
arcs of integral curves of vector fields of the kinds%
\[
\pm h_{j}^{p_{i}/\left\vert I_{j}\right\vert }X_{i}%
\]
for $i=0,1,2,...,n,$ $j=1,2,...,p$, where $\left\{  \left(  X_{\left[
I_{j}\right]  }\right)  _{x_{0}}\right\}  $ is a basis of $\mathbb{R}^{p}$
(see Remark \ref{Remark quasiexponential}). This means that $\gamma$ satisfies%
\[
\left\{
\begin{array}
[c]{l}%
\gamma^{\prime}\left(  \tau\right)  =\sum_{i=0}^{n}a_{i}\left(  \tau\right)
\left(  X_{i}\right)  _{\gamma\left(  \tau\right)  }\\
\gamma\left(  0\right)  =x_{0},\gamma\left(  1\right)  =x
\end{array}
\right.
\]
with%
\[
\left\vert a_{i}\left(  \tau\right)  \right\vert \leq c\left(  \max
_{j=1,...,p}\left\vert h_{j}\right\vert \right)  ^{p_{i}/r}\leq c\left\vert
x-x_{0}\right\vert ^{p_{i}/r}.
\]
This implies that $\gamma\in C_{1}\left(  c\left\vert x-x_{0}\right\vert
^{1/r}\right)  ,$ that is
\begin{equation}
d_{1}\left(  x,x_{0}\right)  \leq c\left\vert x-x_{0}\right\vert ^{1/r}.
\label{d1 FP}%
\end{equation}
So far, we have proved that every point $x_{0}$ has a neighborhood $U$ such
that for any $x\in U$ one has (\ref{d1 FP}), where $c$ is locally uniformly
bounded with respect to $x_{0}$. Then one can also say that every point
$x_{0}$ has a neighborhood $V$ such that for any $x,y\in U$ one has%
\[
d_{1}\left(  x,y\right)  \leq c\left\vert x-y\right\vert ^{1/r}.
\]
A covering argument then implies the desired statement.
\end{proof}

We now want to prove the local equivalence of $d$ and $d_{1}.$ To this aim, we
fix a point $x_{0}\in\Omega^{\prime}\Subset\Omega$ and make use once again of
the smooth approximating vector fields $S_{i}^{x_{0}}$. Let us now denote by%
\[
d_{S},d_{S,1},d_{X},d_{X,1}%
\]
the distances $d$ and $d_{1}$ induced by the systems $\left\{  S_{i}^{x_{0}%
}\right\}  $ and $\left\{  X_{i}\right\}  $, respectively. The above
proposition allows us to repeat also for the distances $d_{S,1},d_{X,1}$ the
proof of Theorem \ref{Thm dX equiv dS}, and get the following:

\begin{theorem}
\label{Thm equivalent balls 1}For any $\Omega^{\prime}\Subset\Omega,$ there
exist positive constants $c_{1},c_{2},r_{0}$ such that%
\[
B_{S}^{1}\left(  x_{0},c_{1}\rho\right)  \subset B_{X}^{1}\left(  x_{0}%
,\rho\right)  \subset B_{S}^{1}\left(  x_{0},c_{2}\rho\right)
\]
for any $x_{0}\in\Omega^{\prime},\rho<r_{0},$ where $B_{S}^{1},B_{X}^{1}$
\ denote the metric balls with respect to $d_{S,1},d_{X,1}$, respectively.
\end{theorem}

Since, by the smooth theory, we already know that $d_{S,1}$ is locally
equivalent to $d_{S}$ (see Remark \ref{remark d-d1}) the last theorem
immediately implies the following result, which strengthens in a quantitative
way the connectivity result contained in Chow's theorem:

\begin{theorem}
\label{Thm equivalent distances d d1}The distances $d_{X,1}$ and $d_{X}$ are
locally equivalent in $\Omega^{\prime}.$ More precisely there exist positive
constants $\rho_{0}$ and $C$ such that for every $w\in\Omega^{\prime}$ and
$y,z\in B_{X}\left(  w,\rho_{0}\right)  $ we have%
\[
d_{X,1}\left(  y,z\right)  \leq Cd_{X}\left(  y,z\right)  .
\]
(The reverse inequality $d\left(  y,z\right)  \leq d_{1}\left(  y,z\right)  $
obviously holds by definition of $d,d_{1}$). As a consequence, the doubling
condition of Theorem \ref{Thm doubling nonsmooth} still holds with respect to
$d_{1}$.
\end{theorem}

Another quantitative consequence of the connectivity property is the
possibility of a pointwise control of the increment of a function $f$ by means
of its \textquotedblleft gradient\textquotedblright\ $Xf=\left(
X_{i}f\right)  _{i=0}^{n}$ (that is, $\left\vert Xf\right\vert $ is an upper
gradient in the terminology of \cite{HK}).

\begin{theorem}
\label{Thm Lagrange}Let $f\in C^{1}\left(  B_{1}\left(  x_{0},\rho\right)
\right)  $, let $x\in B_{1}\left(  x_{0},\rho\right)  $ and let $\gamma\in
C_{1}\left(  \rho\right)  $ be a curve that joins $x_{0}$ with $x$. Then%
\[
\left\vert f\left(  x\right)  -f\left(  x_{0}\right)  \right\vert
\leqslant\sqrt{n}\rho\int_{0}^{1}\left\vert Xf\left(  \gamma\left(  t\right)
\right)  \right\vert dt.
\]
As a consequence we also have%
\begin{align*}
\left\vert f\left(  x\right)  -f\left(  x_{0}\right)  \right\vert  &
\leq\sqrt{n}d_{1}\left(  x,x_{0}\right)  \cdot\sup_{B_{1}\left(  x_{0}%
,\rho\right)  }\left\vert Xf\right\vert \text{ \ \ \ }\forall x\in
B_{1}\left(  x_{0},\rho\right)  ,\\
\left\vert f\left(  x\right)  -f\left(  y\right)  \right\vert  &  \leq\sqrt
{n}d_{1}\left(  x,y\right)  \cdot\sup_{B_{1}\left(  x_{0},\rho\right)
}\left\vert Xf\right\vert \text{ \ \ \ }\forall x,y\in B_{1}\left(
x_{0},\frac{\rho}{3}\right)  .
\end{align*}

\end{theorem}

\begin{proof}
Let $x\in B_{1}\left(  x_{0},\rho\right)  ,$ then there exists a curve
$\gamma\left(  t\right)  $ such that%
\begin{align*}
\gamma\left(  0\right)   &  =x_{0};\gamma\left(  1\right)  =x\\
\gamma^{\prime}\left(  t\right)   &  =\sum_{i=1}^{n}a_{i}\left(  t\right)
\left(  X_{i}\right)  _{\gamma\left(  \tau\right)  }%
\end{align*}
with $\left\vert a_{i}\left(  t\right)  \right\vert \leq\rho,$ then%
\begin{align*}
f\left(  x\right)  -f\left(  x_{0}\right)   &  =\int_{0}^{1}\frac{d}%
{dt}\left(  f\left(  \gamma\left(  t\right)  \right)  \right)  dt=\\
&  =\int_{0}^{1}\sum_{i=1}^{n}a_{i}\left(  t\right)  \left(  X_{i}\right)
_{\gamma\left(  t\right)  }\cdot\nabla f\left(  \gamma\left(  t\right)
\right)  dt\\
&  =\int_{0}^{1}\sum_{i=1}^{n}a_{i}\left(  t\right)  \left(  X_{i}f\right)
\left(  \gamma\left(  t\right)  \right)  dt
\end{align*}%
\begin{align*}
\left\vert f\left(  x\right)  -f\left(  x_{0}\right)  \right\vert  &  \leq
\int_{0}^{1}\sqrt{\sum_{i=1}^{n}a_{i}\left(  t\right)  ^{2}}\cdot\sqrt
{\sum_{i=1}^{n}\left(  X_{i}f\right)  \left(  \gamma\left(  t\right)  \right)
}dt\\
&  \leq\sqrt{n}\rho\int_{0}^{1}\left\vert Xf\left(  \gamma\left(  t\right)
\right)  \right\vert dt\\
&  \leq\sqrt{n}\rho\sup_{B_{1}\left(  x_{0},\rho\right)  }\left\vert
Xf\right\vert .
\end{align*}
For any $x,y\in B_{1}\left(  x_{0},\rho/3\right)  ,$ now, we have $y\in
B_{1}\left(  x,2\rho/3\right)  $ and%
\[
\left\vert f\left(  x\right)  -f\left(  y\right)  \right\vert \leq\sqrt
{n}d\left(  x,y\right)  \sup_{B_{1}\left(  x,2\rho/3\right)  }\left\vert
Xf\right\vert \leq\sqrt{n}d\left(  x,y\right)  \sup_{B_{1}\left(  x_{0}%
,\rho\right)  }\left\vert Xf\right\vert .
\]

\end{proof}

\section{Lifting of nonsmooth vector fields\label{Section lifting}}

In the proof of Poincar\'{e}'s inequality we will use an idea of Jerison (see
\cite{J}) which consists in deriving such inequality first for free vector
fields and then in the general case. This approach requires to develop in the
context on nonsmooth vector fields both Rothschild-Stein's lifting technique
\cite{RS} and an estimate for the volume of lifted balls which was originally
proved by Sanchez-Calle \cite{SC} and Nagel-Stein-Wainger \cite{NSW}. These
tools can also be of independent interest.

In what follows we keep the previous assumptions on the vector fields
$X_{1},X_{2},\ldots,X_{n}$ and we take $X_{0}\equiv0$. Recall that a system of
$n$ vector fields is said to be \emph{free up to step} $r$ if the vector
fields and their commutators of length at most $r$ do not satisfy any linear
dependence relations except those which follow from anticommutativity and
Jacobi identity.

\begin{theorem}
[Lifting theorem]\label{Thm lifting}For every $x_{0}\in\Omega$, there exist a
neighborhood $U\left(  x_{0}\right)  ,$ an integer $m$ and vector fields of
the form%
\begin{equation}
\widetilde{X}_{k}=X_{k}+\sum_{j=1}^{m}u_{kj}\left(  x,t\right)  \frac
{\partial}{\partial t_{j}}\text{ }\left(  k=1,2,...,n\right)  \label{Liftati}%
\end{equation}
defined for $\left(  x,t\right)  \in U\left(  x_{0}\right)  \times I$ where
$I$ is a neighborhood of $0\in$ $\mathbb{R}^{m}$, which are free up to step
$r$ and such that $\left\{  \widetilde{X}_{\left[  I\right]  }\left(
x,t\right)  \right\}  _{\left\vert I\right\vert \leq r}$ span $\mathbb{R}%
^{p+m}$ for every $\left(  x,t\right)  \in U\left(  x_{0}\right)  \times I$.
Moreover the $u_{kj}\left(  x,t\right)  $ can be taken as polynomials of
degree at most $r-1$.
\end{theorem}

After the original paper \cite{RS}, alternative proofs of this result (for
smooth vector fields) have been given by several authors. For our purposes the
most useful is the one given by H\"{o}rmander-Melin \cite{HM}. Indeed, a
careful inspection of their proof shows that it actually requires only the
$C^{r-1}$ regularity of the coefficients and therefore it applies also to our
nonsmooth context.

In the sequel we will use both the lifting procedure and our approximation
procedure by Taylor expansion of degree $r-1$ (see \S 3). Since the
coefficients $u_{kj}\left(  x,t\right)  $ are polynomials of degree $\leq r-1$
one can easily see that these two procedures commute. We will denote by
$\widetilde{S}_{i}^{x}$ the \textquotedblleft lifted approximating
field\textquotedblright.

The next theorem contains a comparison between the volume of balls with
respect to the original vector fields $X_{i}$ and their lifted $\widetilde
{X}_{i};$ we will denote these balls with the symbol $B,\widetilde{B},$
respectively. Thanks to the results proved in \S 5, here the distance induced
by each system of vector fields can be either $d$ (definition 2.2) or $d_{1}$
(definition 5.6). To prove Poincar\'{e}'s inequality we will apply the result
for $d_{1}$.

\begin{theorem}
\label{Thm Sanchez-Calle}Let $x_{0}$ and $U\left(  x_{0}\right)  ,I$ be as in
the above theorem. There exist positive constants $c_{1},c_{2},r_{0},$ and
$\delta\in\left(  0,1\right)  $ such that for any $\left(  x,h\right)  \in
U\left(  x_{0}\right)  \times I$, any $y\in B\left(  x,\delta\rho\right)  ,$
$0<\rho<r_{0},$ we have, denoting by $\left\vert \cdot\right\vert $ the volume
of a ball in the appropriate dimension,
\begin{equation}
c_{1}\frac{\left\vert \widetilde{B}\left(  \left(  x,h\right)  ,\rho\right)
\right\vert }{\left\vert B\left(  x,\rho\right)  \right\vert }\leqslant
\int_{\mathbb{R}^{m}}\chi_{\widetilde{B}\left(  \left(  x,h\right)
,\rho\right)  }\left(  y,s\right)  ds\leqslant c_{2}\frac{\left\vert
\widetilde{B}\left(  \left(  x,h\right)  ,\rho\right)  \right\vert
}{\left\vert B\left(  x,\rho\right)  \right\vert }. \label{Sanchez-Calle}%
\end{equation}
Actually the second inequality holds for every $y\in\mathbb{R}^{p}$. Also, the
projection of $\widetilde{B}\left(  \left(  x,h\right)  ,\rho\right)  $ on
$\mathbb{R}^{p}$ is exactly $B\left(  x,\rho\right)  .$
\end{theorem}

\begin{proof}
As already noted, for smooth vector fields (\ref{Sanchez-Calle}) has been
proved in \cite{SC} and \cite{NSW}. See also \cite{J} where the result is
stated exactly in this form. Let us denote by $B_{S^{x}},\widetilde{B}_{S^{x}%
}$ the balls defined with respect to the vector fields $S_{i}^{x}$ and
$\widetilde{S}_{i}^{x}$ respectively. Since, by Theorem
\ref{Thm equivalent balls 1},%
\[
B_{S^{x}}\left(  x,k_{1}\rho\right)  \subset B\left(  x,\rho\right)  \subset
B_{S^{x}}\left(  x,k_{2}\rho\right)
\]
and%
\[
\widetilde{B}_{S^{x}}\left(  \left(  x,h\right)  ,k_{1}\rho\right)
\subset\widetilde{B}\left(  \left(  x,h\right)  ,\rho\right)  \subset
\widetilde{B}_{S^{x}}\left(  \left(  x,h\right)  ,k_{2}\rho\right)  ,
\]
the result follows from (\ref{Sanchez-Calle}) applied to the smooth vector
fields $S_{i}^{x}$ and $\widetilde{S}_{i}^{x}$ and the doubling property of
Theorem \ref{Thm equivalent distances d d1}. Also, since the lifted vector
field $\widetilde{X}_{i}$ projects onto $X_{i},$ by definition of distance the
projection of $\widetilde{B}\left(  \left(  x,h\right)  ,\rho\right)  $ on
$\mathbb{R}^{p}$ is exactly $B\left(  x,\rho\right)  .$
\end{proof}

\section{Poincar\'{e}'s inequality\label{section poincare}}

For smooth H\"{o}rmander's vector fields, Poincar\'{e}'s inequality has been
proved by Jerison in \cite{J}. Lanconelli-Morbidelli in \cite{LM} have
developed a general approach to Poincar\'{e}'s inequality for (possibly
nonsmooth) vector fields: they first prove an abstract result, which deduces
Poincar\'{e}'s inequality from a property which they call \textquotedblleft
representability of balls by means of controllable almost exponential
maps\textquotedblright, and then show how to apply this general result in
several different situations. One of these situations is the classical case of
smooth H\"{o}rmander's vector fields.

We will prove Poincar\'{e}'s inequality in our context applying the
aformentioned abstract result. To check the assumption of this theorem, we
will exploit all our previous theory, plus some results and arguments used in
\cite{LM} to handle the smooth case. Also, for technical reasons which will be
explained later, we need to apply this abstract result to free vector fields,
and then derive Poincar\'{e}'s inequality in the general case from that proved
in the free case, as already done by Jerison \cite{J} in the smooth setting.
By the way, we remark that it seems hard to apply directly Jerison's argument
to the nonsmooth vector fields (without relying on Lanconelli-Morbidelli's
theory), since this would require also Rothschild-Stein's approximation
technique, which in our nonsmooth setting is not presently available.

Henceforth we will further strengthen our assumption as follows:

\bigskip

\textbf{Assumptions (D). }We assume that for some integer $r\geq2$ and some
bounded domain $\Omega\subset\mathbb{R}^{p}$ the following hold:

\begin{itemize}
\item[(D1)] The coefficients of the vector fields $X_{1},X_{2},...,X_{n}$
belong to $C^{r-1,1}\left(  \Omega\right)  ,$ while $X_{0}\equiv0.$ Here and
in the following, $C^{k,1}$ stands for the classical space of functions with
Lipschitz continuous derivatives up to order $k$.

\item[(D2)] The vectors $\left\{  \left(  X_{\left[  I\right]  }\right)
_{x}\right\}  _{\left\vert I\right\vert \leq r}$ span $\mathbb{R}^{p}$ at
every point $x\in\Omega$.
\end{itemize}

Assumptions (D) will be in force throughout this section and the following.
See, however, our Remark \ref{Remark rough Poincare} at the end of this
section, for an inequality that holds under weaker assumptions.

\begin{remark}
The lacking of $X_{0}$ is a natural assumption dealing with Poincar\'{e}-type
inequalities. Note that, since $X_{0}\equiv0,$ length and weight of a
multiindex now coincide. Also note that under the assumption (D1) above, for
any $1\leq k\leq r,$ the differential operators%
\[
\left\{  X_{I}\right\}  _{\left\vert I\right\vert \leq k}%
\]
are well defined, and have $C^{r-k,1}$ coefficients. The same is true for the
vector fields $\left\{  X_{\left[  I\right]  }\right\}  _{\left\vert
I\right\vert \leq k}.$
\end{remark}

\textbf{Dependence of the constants.} Whenever we will write that some
constant depends on the vector fields $X_{i}$'s and some fixed domain
$\Omega^{\prime}\Subset\Omega$, this will mean that the constant depends on:

(i) diam$\left(  \Omega^{\prime}\right)  $;

(ii) the norms $C^{r-1,1}\left(  \Omega\right)  $ of the coefficients of
$X_{i}$ $\left(  i=1,2,...,n\right)  $;

(iii) a positive constant $c_{0}$ such that the following bound holds:%
\[
\inf_{x\in\Omega^{\prime}}\max_{\left\vert I_{1}\right\vert ,\left\vert
I_{2}\right\vert ,...,\left\vert I_{p}\right\vert \leq r}\left\vert
\det\left(  \left(  X_{\left[  I_{1}\right]  }\right)  _{x},\left(  X_{\left[
I_{2}\right]  }\right)  _{x},...,\left(  X_{\left[  I_{p}\right]  }\right)
_{x}\right)  \right\vert \geq c_{0}.
\]

In the following we will sometimes work with the lifted vector fields, defined
in \S 6, so that the constants appearing in our results will also depend on
the constants in (ii)-(iii) associated to the lifted vector fields. However,
as observed by Jerison \cite[Lemma (4.2) p.511]{J}, these in turn only depend
on the constants in (ii)-(iii) corresponding to the original vector fields
$X_{i}$'s.

Let us state our main result:

\begin{theorem}
[Poincar\'{e}'s inequality]\label{Thm Poincare}For any $\Omega^{\prime}%
\Subset\Omega$ there exist constants $c,r_{0}>0,\lambda\geq1$ such that for
any $d_{X,1}$-ball $B=B\left(  x,\rho\right)  ,$ with $\rho\leq r_{0}$,
$x\in\Omega^{\prime},$ any $u\in C^{1}\left(  \overline{\lambda B}\right)  ,$
with $\lambda B=B\left(  x,\lambda\rho\right)  ,$ we have
\begin{equation}
\int_{B\times B}\left\vert u\left(  y\right)  -u\left(  x\right)  \right\vert
dydx\leq c\rho\left\vert B\right\vert \int_{\lambda B}\left\vert Xu\left(
y\right)  \right\vert dy \label{Poincare 1}%
\end{equation}
where:%
\[
\left\vert Xu\left(  y\right)  \right\vert =\sqrt{\sum_{j=1}^{n}\left\vert
X_{j}u\left(  y\right)  \right\vert ^{2}}.
\]

\end{theorem}

Note that (\ref{Poincare 1}) is equivalent to the following (perhaps more
familiar) form of Poincar\'{e}'s inequality:%
\begin{equation}
\int_{B}\left\vert u\left(  y\right)  -u_{B}\right\vert dy\leq c\rho
\int_{\lambda B}\left\vert Xu\left(  y\right)  \right\vert dy
\label{Poincare 2}%
\end{equation}
where, as usual, $u_{B}$ denotes the average of $u$ over $B$.

\bigskip

We start showing that Poincar\'{e}'s inequality for free vector fields implies
the same result in the general case.

So, fix $x_{0}\in\Omega$ and a neighborhood $U\left(  x_{0}\right)
\subset\Omega$ where the lifting theorem \ref{Thm lifting} can be applied, and
let $\widetilde{X}_{i}$ be the lifted free vector fields in $U\left(
x_{0}\right)  \times I$. Then for any $U^{\prime}\left(  x_{0}\right)  \Subset
U\left(  x_{0}\right)  ,$ $x\in U^{\prime}\left(  x_{0}\right)  ,$ if
$\widetilde{B}=\widetilde{B}\left(  \left(  x,h\right)  ,\rho\right)  $ is a
$d_{\widetilde{X},1}$-ball with $\rho\leq r_{0}$ and $\lambda\widetilde
{B}=\widetilde{B}\left(  \left(  x,h\right)  ,\lambda\rho\right)  ,$ we have
\begin{equation}
\int_{\widetilde{B}\times\widetilde{B}}\left\vert u\left(  y,t\right)
-u\left(  z,s\right)  \right\vert dydzdsdt\leq c\rho\left\vert \widetilde
{B}\right\vert \int_{\lambda\widetilde{B}}\left\vert \widetilde{X}u\left(
y,t\right)  \right\vert dydt\label{lifted Poincare}%
\end{equation}
for any $u\in C^{1}\left(  \overline{\lambda\widetilde{B}}\right)  .$ If we
apply (\ref{lifted Poincare}) to a function $u\left(  y\right)  $ independent
of $t,$ $u\in C^{1}\left(  \overline{\lambda B}\right)  ,$ then by
(\ref{Liftati}) we get
\[
\int_{\widetilde{B}\times\widetilde{B}}\left\vert u\left(  y\right)  -u\left(
z\right)  \right\vert dydzdsdt\leq c\rho\left\vert \widetilde{B}\right\vert
\int_{\lambda\widetilde{B}}\left\vert Xu\left(  y\right)  \right\vert dydt
\]
that is (since $\widetilde{B}$ projects onto $B$)
\begin{align*}
&  \int_{B\times B}\left\vert u\left(  y\right)  -u\left(  z\right)
\right\vert dydz\int_{\mathbb{R}^{m}}\chi_{\widetilde{B}\left(  \left(
x,h\right)  ,\rho\right)  }\left(  y,t\right)  dt\int_{\mathbb{R}^{m}}%
\chi_{\widetilde{B}\left(  \left(  x,h\right)  ,\rho\right)  }\left(
z,s\right)  ds\leq\\
&  \leq c\rho\left\vert \widetilde{B}\right\vert \int_{\lambda B}\left\vert
Xu\left(  y\right)  \right\vert dy\int_{\mathbb{R}^{m}}\chi_{\widetilde
{B}\left(  \left(  x,h\right)  ,\rho\right)  }\left(  y,t\right)  dt
\end{align*}
which, by Theorem \ref{Thm Sanchez-Calle}, implies (\ref{Poincare 1}).

A standard compactess argument then gives theorem \ref{Thm Poincare}.

\bigskip

We now proceed to prove theorem \ref{Thm Poincare} under the additional
assumption that the vector fields $X_{i}$ are free up to step $r$.

Let us start fixing some notation. Throughout this section we will assume
fixed a bounded domain $\Omega^{\prime}\Subset\Omega.$ Let $\left\{
X_{\left[  I_{j}\right]  }\right\}  _{I_{j}\in\eta}$ be any particular family
of $p$ commutators of our vector fields $\left\{  X_{i}\right\}  _{i=1}^{n}$,
with $\left\vert I_{j}\right\vert \leq r$; let%
\begin{align*}
\left\vert \eta\right\vert  &  =\sum_{j=1}^{p}\left\vert I_{j}\right\vert ;\\
\left\Vert h\right\Vert _{\eta} &  =\max_{j=1,...,p}\left\vert h_{j}%
\right\vert ^{1/\left\vert I_{j}\right\vert }\text{ for any }h\in
\mathbb{R}^{p};\\
Q_{\eta}\left(  \rho\right)   &  =\left\{  h\in\mathbb{R}^{p}:\left\Vert
h\right\Vert _{\eta}\leq\rho\right\}  \text{ for any }\rho>0.
\end{align*}
Let us recall that (see Theorem \ref{Thm diffeomorfism}) $E_{\eta}^{X}\left(
x,h\right)  $ is $C^{1}$ in the joint variables $\left(  x,h\right)  $ for
$x\in\Omega^{\prime}$ and $\left\vert h\right\vert <\delta_{1}.$ Moreover, if
$\left\{  X_{\left[  I_{j}\right]  }\right\}  _{I_{j}\in\eta}$ is a family of
commutators which spans $\mathbb{R}^{p}$ at $x_{0}\in\Omega^{\prime}$ (and
therefore in the whole $\Omega^{\prime},$ since the $X_{i}$'s are free) and
satisfies (\ref{eta opportuno}), then $h\longmapsto$ $E_{\eta}^{X}\left(
x,h\right)  $ is a diffeomorphism of a neighborhood of the origin $\left\{
h:\left\vert h\right\vert <\delta_{1}\right\}  $ onto a neighborhood $U\left(
x_{0}\right)  $ of $x_{0}$ containing $\left\{  x:\left\vert x-x_{0}%
\right\vert <\delta_{2}\right\}  .$

We also denote by%
\[
D_{E_{\eta}^{X}\left(  x,h\right)  }%
\]
the modulus of the Jacobian determinant of the mapping $h\longmapsto E_{\eta
}^{X}\left(  x,h\right)  .$ The function $D_{E_{\eta}^{X}\left(  x,h\right)
}$ is continuous for $x\in\Omega^{\prime}$ and $\left\vert h\right\vert
<\delta_{1}$;\ moreover,%
\[
D_{E_{\eta}^{X}\left(  x,0\right)  }=\left\vert \det\left\{  \left(
X_{\left[  I_{j}\right]  }\right)  _{x}\right\}  _{j=1}^{p}\right\vert
>0\text{ for any }x\in\Omega^{\prime}.
\]

For any fixed $x_{0}\in\Omega^{\prime},$ let now $\left\{  S_{i}^{x_{0}%
}\right\}  $ be the system of smooth approximating vector fields, introduced
in \S \ref{section approximating balls}. We know that (see Lemma
\ref{Lemma intorno S}) the $S_{i}^{x_{0}}$'s satisfy H\"{o}rmander's condition
in%
\[
U_{\delta}\left(  x_{0}\right)  =\left\{  x\in\Omega:\left\vert x-x_{0}%
\right\vert <\delta\right\}  .
\]
For any $x_{0}\in\Omega^{\prime}$ we can therefore consider the system of
smooth H\"{o}rmander's vector fields $\left\{  S_{i}^{x_{0}}\right\}  $ in
$U_{\delta}\left(  x_{0}\right)  ,$ and perform for this system, as in
\S \ref{section connectivity}, the construction of the maps
\[
C_{\ell\left(  I\right)  }\left(  t,S_{I}^{x_{0}}\right)  \left(  x\right)
\]
and that of the corresponding diffeomorphism%
\[
E_{\eta}^{S^{x_{0}}}\left(  x,h\right)  .
\]

It is now time to recall what is the abstract result proved in \cite{LM}
regarding Poincar\'{e} inequality, how the Authors use it to deduce Jerison's
Poincar\'{e} inequality (in the smooth case), and how we can adapt their
arguments to our context. As we will see, a key feature of our approach is a
suitable mix of the two systems of vector fields $\left\{  X_{i}\right\}
,\left\{  S_{i}^{x_{0}}\right\}  $, and of the corresponding maps $E_{\eta
}^{X},E_{\eta}^{S^{x_{0}}}.$

Let us start recalling a definition from \cite{LM}.

\begin{definition}
\label{Definition quasiexponential}Let $O$ be a bounded open set in
$\mathbb{R}^{p},$ and $Q$ a neighborhood of the origin. We say that a function%
\[
E:O\times Q\rightarrow\mathbb{R}^{p}%
\]
is an \emph{almost exponential map} if

\begin{enumerate}
\item[(i)] $E\left(  x,0\right)  =x$ $\forall x\in O$

\item[(ii)] the map $h\longmapsto E\left(  x,h\right)  $ is $C^{1}$ and
$1$-$1$ on $Q$

\item[(iii)] the following condition holds:%
\[
\frac{1}{a}D_{E\left(  x,0\right)  }\leq D_{E\left(  x,h\right)  }\leq
aD_{E\left(  x,0\right)  }\forall\left(  x,h\right)  \in O\times Q,\text{ some
constant }a>1
\]
(where $D_{E\left(  x,h\right)  }$ stands for the modulus of the Jacobian
determinant of the mapping $h\longmapsto E\left(  x,h\right)  $).
\end{enumerate}
\end{definition}

The abstract theorem proved by Lanconelli-Morbidelli reads as follows:

\begin{theorem}
[Theorem 2.1 in \cite{LM}]\label{Thm 2.1. LM}Let $B=B_{1}^{X}\left(
x_{0},\rho\right)  $ be a fixed ball. Assume there exists an open set
$O\subset B,$ an almost exponential map $E:O\times Q\rightarrow\mathbb{R}^{p}$
and two positive constants $\alpha,\beta$ satisfying the following conditions:

\begin{enumerate}
\item[(i)] $\left\vert B\right\vert \leq\alpha\left\vert O\right\vert $ and
$B\subset E\left(  x,Q\right)  $ for every $x\in O;$

\item[(ii)] $E$ is $X$-controllable with a hitting time $T\leq\alpha\rho;$

\item[(iii)] $\left\vert \left(  \alpha+1\right)  B\right\vert \leq
\beta\left\vert O\right\vert .$
\end{enumerate}

Then there exists $c>0$ such that%
\[
\int_{B\times B}\left\vert u\left(  y\right)  -u\left(  x\right)  \right\vert
dydx\leq c\rho\left\vert B\right\vert \int_{\left(  1+\alpha\right)
B}\left\vert Xu\left(  y\right)  \right\vert dy
\]
for any $u\in C^{1}\left(  \overline{\left(  1+\alpha\right)  B}\right)  ,$
where $\left(  1+\alpha\right)  B=B\left(  x_{0},\left(  1+\alpha\right)
\rho\right)  $. The constant $c$ only depends on the numbers $\alpha,\beta,$
the constant $a$ appearing in Definition \ref{Definition quasiexponential} and
the constant $b$ appearing in Definition \ref{Definition X-controllable}.
\end{theorem}

We will recall and comment later the definition of \textquotedblleft%
$X$-controllable map\textquotedblright.

Our strategy consists in showing that the map%
\[
\left(  x,h\right)  \longmapsto E_{\eta}^{S^{x}}\left(  x,h\right)
\]
satisfies the assumption of the previous theorem, on suitable domains $O,Q,$
for a suitable choice of $\eta$. Note that this map, built upon the system
$\left\{  S_{i}^{x_{0}}\right\}  ,$ will be shown to satisfy the assumptions
of the theorem \textit{with respect to the system} $\left\{  X_{i}\right\}  $.
Also, note that in the definition of this map $E$ the point $x_{0}$ where the
system $\left\{  S_{i}^{x_{0}}\right\}  $ approximates $\left\{
X_{i}\right\}  $ is taken equal to $x$, that is \textquotedblleft
unfrozen\textquotedblright. Therefore our map $E$ will be only Lipschitz
continuous with respect to $x$. These facts will require some care.

First of all, we can apply to the system of smooth H\"{o}rmander's vector
fields $S_{i}^{x_{0}},$ in the domain $U_{\delta}\left(  x_{0}\right)  ,$
Theorem 4.1 in \cite{LM}:

\begin{theorem}
\label{Thm Morbidelli NSW}For any $x_{0}\in\Omega^{\prime}\Subset\Omega$ there
exist positive constants $r_{0},c_{1},c_{2},$ with $c_{2}<c_{1}<1,$ such that
for any family $\eta$ of $p$ commutators, $x\in U_{\delta}\left(
x_{0}\right)  ,\rho\leq r_{0}$ satisfying the inequality%
\[
D_{E_{\eta}^{S^{x_{0}}}\left(  x,0\right)  }\rho^{\left\vert \eta\right\vert
}\geq\frac{1}{2}\max_{\zeta}D_{E_{\zeta}^{S^{x_{0}}}\left(  x,0\right)  }%
\rho^{\left\vert \zeta\right\vert }%
\]
the following assertions hold:

(a) If $h\in Q_{\eta}\left(  c_{1}\rho\right)  $ then%
\[
\frac{1}{4}D_{E_{\eta}^{S^{x_{0}}}\left(  x,0\right)  }\leq D_{E_{\eta
}^{S^{x_{0}}}\left(  x,h\right)  }\leq4D_{E_{\eta}^{S^{x_{0}}}\left(
x,0\right)  }%
\]

(b) $B_{S^{x_{0}}}^{1}\left(  x,c_{2}\rho\right)  \subset E_{\eta}^{S^{x_{0}}%
}\left(  x,Q_{\eta}\left(  c_{1}\rho\right)  \right)  .$

(c) The function $E_{\eta}^{S^{x_{0}}}\left(  x,\cdot\right)  $ is one-to-one
on the set $Q_{\eta}\left(  c_{1}\rho\right)  .$

\noindent Here $B_{S^{x_{0}}}^{1}$ stands for the metric ball with respect to
the distance $d_{S^{x_{0}},1}$.
\end{theorem}

As the Authors write in \cite{LM}, the above theorem, in the case of smooth
H\"{o}rmander's vector fields, has a proof similar to that of Theorem 7 in
\cite{NSW}, which is written in detail in \cite{Mor}. Note that the constants
$r_{0},c_{1},c_{2}$ in the above Theorem only depend on the $X_{i}$'s and
$\Omega^{\prime}$, and not on $x_{0}$ (see the discussion in
\S \ \ref{section approximating balls})\textit{.} We can then set $x=x_{0}$ in
the above theorem, obtaining the following:

\begin{theorem}
\label{Thm 4.1 LM}For any $\Omega^{\prime}\Subset\Omega$ there exist positive
constants $r_{0},c_{1},c_{2},$ with $c_{2}<c_{1}<1,$ such that for any family
$\eta$ of $p$ commutators, $x\in\Omega^{\prime},\rho\leq r_{0}$ satisfying the
inequality%
\begin{equation}
D_{E_{\eta}^{S^{x}}\left(  x,0\right)  }\rho^{\left\vert \eta\right\vert }%
\geq\frac{1}{2}\max_{\zeta}D_{E_{\zeta}^{S^{x}}\left(  x,0\right)  }%
\rho^{\left\vert \zeta\right\vert } \label{optimal basis}%
\end{equation}
the following assertions hold:

(a') If $h\in Q_{\eta}\left(  c_{1}\rho\right)  $ then%
\begin{equation}
\frac{1}{4}D_{E_{\eta}^{S^{x}}\left(  x,0\right)  }\leq D_{E_{\eta}^{S^{x}%
}\left(  x,h\right)  }\leq4D_{E_{\eta}^{S^{x}}\left(  x,0\right)  }
\label{D E eta}%
\end{equation}

(b') $B_{X}^{1}\left(  x,c_{2}\rho\right)  \subset E_{\eta}^{S^{x}}\left(
x,Q_{\eta}\left(  c_{1}\rho\right)  \right)  .$

(c') The function $E_{\eta}^{S^{x}}\left(  x,\cdot\right)  $ is one-to-one on
the set $Q_{\eta}\left(  c_{1}\rho\right)  .$

\noindent Here $B_{X}^{1}$ stands for the metric ball with respect to the
distance $d_{X,1}$.
\end{theorem}

Note that (b') also exploits Theorem \ref{Thm equivalent balls 1} (with
possibly a smaller value of $c_{2}$).

In the following we will need to shrink the constant $r_{0}$ appearing in this
theorem; however this is not restrictive.

In order to find the sets $O,Q$ to which we will apply Theorem
\ref{Thm 2.1. LM} we now proceed like in \cite[p.336]{LM}: let $B$ be a
$d_{1}^{X}$-ball centered at some $x_{0}\in\Omega^{\prime}$ of radius
$\rho<c_{2}r_{0}/2.$ For any family $\eta$ of $p$ commutators of lenght $\leq
r$, we define%
\begin{equation}
\Omega_{\eta}=\left\{  x\in B:D_{E_{\eta}\left(  x,0\right)  }\left(
\frac{2\rho}{c_{2}}\right)  ^{\left\vert \eta\right\vert }>\frac{1}{2}%
\max_{\zeta}D_{E_{\zeta}\left(  x,0\right)  }\left(  \frac{2\rho}{c_{2}%
}\right)  ^{\left\vert \zeta\right\vert }\right\}  .\label{Omega eta}%
\end{equation}
Here $D_{E_{\eta}\left(  x,0\right)  }$ stands for both $D_{E_{\eta}^{S^{x}%
}\left(  x,0\right)  }$ and $D_{E_{\eta}^{X}\left(  x,0\right)  }$ (since the
two quantities coincide). At least one of the sets $\Omega_{\eta}$ satisfies%
\begin{equation}
\left\vert \Omega_{\eta}\right\vert \geq\frac{1}{N}\left\vert B\right\vert
\label{B - Omega}%
\end{equation}
where $N$ is the total number of $p$-tuples available. Let us choose one of
such $\eta$'s and denote by $Q$ the box%
\[
Q=\left\{  h\in\mathbb{R}^{p}:\left\Vert h\right\Vert _{\eta}<\frac{2c_{1}%
}{c_{2}}\rho\right\}  .
\]
From now on, the basis $\eta$ is chosen once and for all. By (a') and (c') of
Theorem \ref{Thm 4.1 LM}, the function%
\begin{align}
E:\Omega_{\eta}\times Q &  \rightarrow\mathbb{R}^{p}\label{quasiexp E}\\
\left(  x,h\right)   &  \mapsto E_{\eta}^{S^{x}}\left(  x,h\right)  \nonumber
\end{align}
is an \textit{almost exponential map}. The fact that $h\mapsto E_{\eta}%
^{S^{x}}\left(  x,h\right)  $ is $C^{1}$ follows from Theorem
\ref{Thm diffeomorfism} applied to the smooth vector fields $S_{i}^{x}$ (for
any frozen $x$). \ 

We will show that the almost exponential map $E$ we have just built satisfies
assumptions (i),(ii),(iii) in Theorem \ref{Thm 2.1. LM}. This will imply our
Poincar\'{e}'s inequality.

By (\ref{B - Omega}), $\left\vert B\right\vert \leq N\left\vert \Omega_{\eta
}\right\vert ,$ while by (b') of Theorem \ref{Thm 4.1 LM}, $B\subset E\left(
x,Q\right)  $ for every $x\in\Omega_{\eta}.$ Thus assumption (i) in Theorem
\ref{Thm 2.1. LM} is satisfied, while assumption (iii) follows from the
doubling condition for $d_{1}^{X}$ balls, which we have proved in Theorem
\ref{Thm equivalent distances d d1}, plus inequality $\left\vert B\right\vert
\leq N\left\vert \Omega_{\eta}\right\vert .$

It remains to prove that the map $E$ is \textquotedblleft$X$-controllable with
a hitting time $T\leq\alpha\rho$\textquotedblright,\ that is, condition (ii).
Let us first recall the definition of this concept, as appears in
\cite[p.330]{LM}:

\begin{definition}
\label{Definition X-controllable}We say that an almost exponential map
$E:O\times Q\rightarrow\mathbb{R}^{p}$ is $X$-controllable with a hitting time
$T$ if there exists a function $\gamma:O\times Q\times\left[  0,T\right]
\rightarrow\mathbb{R}^{p}$ such that

(C1) For any $\left(  x,h\right)  \in O\times Q,$ $t\mapsto\gamma\left(
x,h,t\right)  $ is an $X$-subunit path connecting $x$ and $E\left(
x,h\right)  ,$ that is%
\[
\left\{
\begin{array}
[c]{l}%
\frac{d}{dt}\gamma\left(  x,h,t\right)  =\sum a_{j}\left(  t\right)  \left(
X_{j}\right)  _{\gamma\left(  x,h,t\right)  }\text{ for suitable }a_{j}\text{
with }\sum\left\vert a_{j}\left(  t\right)  \right\vert ^{2}\leq1\\
\gamma\left(  x,h,0\right)  =x;\gamma\left(  x,h,T\left(  x,h\right)  \right)
=E\left(  x,h\right)
\end{array}
\right.
\]
for a suitable $T\left(  x,h\right)  \leq T.$

(C2) For any $\left(  h,t\right)  \in Q\times\left[  0,T\right]  ,$
$x\mapsto\gamma\left(  x,h,t\right)  $ is a one-to-one $C^{1}$ map having
jacobian determinant bounded away from zero, i.e.%
\[
b\equiv\inf_{O\times Q\times\left[  0,T\right]  }\left\vert \frac
{\partial\gamma}{\partial x}\right\vert >0.
\]

\end{definition}

We start noting that condition (C2) is used in \cite[p.330]{LM} only once, in
the following change of variable:%
\[
\int_{O}\left\vert Xu\left(  \gamma\left(  x,h,t\right)  \right)  \right\vert
dx\leq\frac{1}{b}\int_{B\left(  x_{0},\left(  \alpha+1\right)  \rho\right)
}\left\vert Xu\left(  z\right)  \right\vert dz.
\]
It is then apparent that (C2) can be replaced by the weaker assumption:

\bigskip

(C2') For any $\left(  h,t\right)  \in Q\times\left[  0,T\right]  ,$
$x\mapsto\gamma\left(  x,h,t\right)  $ is a one-to-one bilipschitz map having
jacobian determinant bounded away from zero, i.e.%
\[
b\equiv\inf_{O\times Q\times\left[  0,T\right]  }\left\vert \frac
{\partial\gamma}{\partial x}\right\vert >0.
\]

In view of Theorem \ref{Thm 2.1. LM}, our proof of Theorem \ref{Thm Poincare}
will be completed as soon as we will prove the following:

\begin{proposition}
\label{Prop X contr}The almost exponential map $E$ defined in
(\ref{quasiexp E}) is $X$-controllable with a hitting time $T\leq\alpha\rho,$
in the sense of the above definition with (C1), (C2'), where $\alpha$ and
$\rho$ are as above.
\end{proposition}

\begin{remark}
\label{Remark liberi}Before going on, we have to make an important
observation, in order to explain the role played by the fact that our vector
fields $X_{i}$ are free. In the following, we need to choose a basis $\left\{
X_{\left[  I\right]  }\right\}  _{I\in\eta}$ satisfying \emph{simultaneously
}the condition expressed in (\ref{Omega eta}), which is fundamental to apply
the theory of Lanconelli-Morbidelli, and the condition (\ref{eta opportuno})
in Theorem \ref{Thm diffeomorfism}, which allows us to have a quantitative
control on the diffeomorphism induced by the basis. This could impossible for
a general family of H\"{o}rmander's vector fields, but it is easy as soon as
they are free.

Namely, since the vector fields $X_{i}$ are free up to step $r$, they satisfy
the same commutation relations (up to step $r$) at any point of $\Omega.$
Therefore all the possible families of $p$ vector fields chosen among the
commutators of length $\leq r$ of the $X_{i}$'s can be grouped in two classes:%
\begin{align*}
\mathcal{E}  &  =\left\{  \eta:D_{E_{\eta}\left(  x,0\right)  }\neq0\text{ for
all }x\in\Omega\right\}  ;\\
\mathcal{E}^{c}  &  =\left\{  \eta:D_{E_{\eta}\left(  x,0\right)  }=0\text{
for all }x\in\Omega\right\}  .
\end{align*}
If we choose a basis $\eta$ satisfying the relation%
\[
D_{E_{\eta}\left(  x,0\right)  }\left(  \frac{2\rho}{c_{2}}\right)
^{\left\vert \eta\right\vert }>\frac{1}{2}\max_{\zeta}D_{E_{\zeta}\left(
x,0\right)  }\left(  \frac{2\rho}{c_{2}}\right)  ^{\left\vert \zeta\right\vert
},
\]
this means that $\eta\in\mathcal{E}$, hence $D_{E_{\eta}\left(  x,0\right)
}\neq0$ in the whole $\Omega.$ In order to apply Theorem
\ref{Thm diffeomorfism} with a control on the constants which are involved, we
need to know that, for some fixed $\varepsilon,$%
\begin{equation}
D_{E_{\eta}\left(  x,0\right)  }>\left(  1-\varepsilon\right)  \max_{\zeta
}D_{E_{\zeta}\left(  x,0\right)  } \label{epsi}%
\end{equation}
Now,%
\[
D_{E_{\eta}\left(  x,0\right)  }\geq\min_{\zeta\in\mathcal{E}}D_{E_{\zeta
}\left(  x,0\right)  }>\left(  1-\varepsilon\right)  \max_{\zeta}D_{E_{\zeta
}\left(  x,0\right)  }%
\]
because: since the vector fields are free, all the determinants relative to
different bases control each other by universal constants. This means that
(\ref{epsi}) holds with a universal constant $\varepsilon.$
\end{remark}

Let us now proceed toward the proof of Proposition \ref{Prop X contr}. First
of all, we have the following:

\begin{lemma}
\label{Brandolemma}There exists a constant $c>0$ depending on $\Omega^{\prime
}$ and the $X_{i}$'s such that
\begin{equation}
E_{\eta}^{S^{x}}\left(  x,Q_{\eta}\left(  \rho\right)  \right)  \subset
E_{\eta}^{X}\left(  x,Q_{\eta}\left(  c\rho\right)  \right)  \text{
\ \ }\forall x\in\Omega_{\eta}. \label{box inclusion}%
\end{equation}

\end{lemma}

\begin{proof}
Let $y\in E_{\eta}^{S^{x}}\left(  x,Q_{\eta}\left(  \rho\right)  \right)  ;$
this means that%
\[
y=E_{\eta}^{S^{x}}\left(  x,h\right)  \text{ for some }\left\Vert h\right\Vert
_{\eta}\leq\rho,
\]
and we want to prove that there exists $h^{\prime},$ with $\left\Vert
h^{\prime}\right\Vert _{\eta}\leq c\rho$, such that%
\[
E_{\eta}^{S^{x}}\left(  x,h\right)  =E_{\eta}^{X}\left(  x,h^{\prime}\right)
.
\]
Since, by Theorem \ref{Thm diffeomorfism}, for any $x\in\Omega^{\prime}$ the
mapping
\[
h^{\prime}\longmapsto E_{\eta}^{X}\left(  x,h^{\prime}\right)
\]
is a diffeomorphism (for $\left\vert h^{\prime}\right\vert <\delta$), we can
invert it, writing%
\[
h^{\prime}=\Theta\left(  h\right)  \equiv E_{\eta}^{X}\left(  x,\cdot\right)
^{-1}E_{\eta}^{S^{x}}\left(  x,h\right)  .
\]
Moreover, by Remark \ref{Remark liberi}, the constants involved in this
diffeomorphism are under control.

We want to prove that%
\begin{equation}
\left\Vert h^{\prime}\right\Vert _{\eta}\leqslant c\left\Vert h\right\Vert
_{\eta} \label{brando1}%
\end{equation}
for some constant $c>0$ only depending on $\Omega^{\prime}$ and the $X_{i}$'s.
We have%
\begin{align}
\left\Vert h^{\prime}\right\Vert _{\eta}  &  \leqslant c\left(  \left\Vert
h\right\Vert _{\eta}+\left\Vert h^{\prime}-h\right\Vert _{\eta}\right)
\nonumber\\
&  \leqslant c\left(  \left\Vert h\right\Vert _{\eta}+\left\vert h^{\prime
}-h\right\vert ^{1/r}\right)  \label{brando2}%
\end{align}
Since $E_{\eta}^{X}\left(  x,\cdot\right)  ^{-1}$ is a diffeomorphism, we have%
\begin{align}
\left\vert h^{\prime}-h\right\vert  &  =\left\vert E_{\eta}^{X}\left(
x,\cdot\right)  ^{-1}E_{\eta}^{S^{x}}\left(  x,h\right)  -E_{\eta}^{X}\left(
x,\cdot\right)  ^{-1}E_{\eta}^{X}\left(  x,h\right)  \right\vert \nonumber\\
&  \leqslant c\left\vert E_{\eta}^{S^{x}}\left(  x,h\right)  -E_{\eta}%
^{X}\left(  x,h\right)  \right\vert \label{brando3}%
\end{align}
We are going to show that
\begin{equation}
\left\vert E_{\eta}^{S^{x}}\left(  x,h\right)  -E_{\eta}^{X}\left(
x,h\right)  \right\vert \leqslant c\left\Vert h\right\Vert _{\eta}^{r}
\label{brando4}%
\end{equation}
which, together with (\ref{brando1}), (\ref{brando2}), (\ref{brando3}), will
imply our assertion.

Namely,
\begin{align*}
E_{\eta}^{S^{x}}\left(  x,h\right)   &  =\prod_{j=1}^{M}\exp\left(  \left\vert
h_{i_{j}}\right\vert ^{1/l_{j}}\sigma_{j}S_{r_{j}}^{x}\right)  \left(
x\right)  ;\\
E_{\eta}^{X}\left(  x,h\right)   &  =\prod_{j=1}^{M}\exp\left(  \left\vert
h_{i_{j}}\right\vert ^{1/l_{j}}\sigma_{j}X_{r_{j}}\right)  \left(  x\right)
\end{align*}
(here $\sigma_{j}=\pm1$). Set:%
\begin{align*}
y_{n}  &  =\prod_{j=1}^{n}\exp\left(  \left\vert h_{i_{j}}\right\vert
^{1/l_{j}}\sigma_{j}S_{r_{j}}^{x}\right)  \left(  x\right) \\
\widetilde{y}_{n}  &  =\prod_{j=1}^{n}\exp\left(  \left\vert h_{i_{j}%
}\right\vert ^{1/l_{j}}\sigma_{j}X_{r_{j}}\right)  \left(  x\right)
\end{align*}
and let us show by induction on $n$ that $\left\vert y_{n}-\widetilde{y}%
_{n}\right\vert \leqslant c\left\Vert h\right\Vert _{\eta}^{r}$. For $n=1$ we
can apply directly the argument in the proof of Theorem \ref{Thm dX equiv dS},
getting%
\[
\left\vert y_{1}-\widetilde{y}_{1}\right\vert \leqslant c\left(  \left\vert
h_{i_{1}}\right\vert ^{1/l_{1}}\right)  ^{r}\leqslant c\left\Vert h\right\Vert
_{\eta}^{r}.
\]
Assuming the assertion for $n,$ let us write now:%
\begin{align*}
y_{n+1}  &  =\exp\left(  \left\vert h_{i_{n+1}}\right\vert ^{1/l_{n+1}}%
\sigma_{n+1}S_{r_{j}}^{x}\right)  \left(  y_{n}\right)  ;\\
\widetilde{y}_{n+1}  &  =\exp\left(  \left\vert h_{i_{n+1}}\right\vert
^{1/l_{n+1}}\sigma_{n+1}X_{r_{j}}\right)  \left(  \widetilde{y}_{n}\right)  .
\end{align*}
We can repeat again the argument in the proof of Theorem \ref{Thm dX equiv dS}%
. Let%
\[
y_{n+1}=\varphi\left(  1\right)  ;\widetilde{y}_{n+1}=\gamma\left(  1\right)
\text{ with}%
\]%
\[
\left\{
\begin{array}
[c]{l}%
\varphi^{\prime}\left(  \tau\right)  =\left\vert h_{i_{n+1}}\right\vert
^{1/l_{n+1}}\sigma_{n+1}\left(  S_{r_{j}}^{x}\right)  _{\varphi\left(
\tau\right)  }\\
\varphi\left(  0\right)  =y_{n}%
\end{array}
\right.  \text{ \ }\left\{
\begin{array}
[c]{l}%
\gamma^{\prime}\left(  \tau\right)  =\left\vert h_{i_{n+1}}\right\vert
^{1/l_{n+1}}\sigma_{n+1}\left(  X_{r_{j}}\right)  _{\gamma\left(  \tau\right)
}\\
\gamma\left(  0\right)  =\widetilde{y}_{n}%
\end{array}
\right.
\]
Then%
\begin{align*}
\varphi\left(  s\right)  -\gamma\left(  s\right)   &  =y_{n}-\widetilde{y}%
_{n}+\int_{0}^{s}\left[  \varphi^{\prime}\left(  \tau\right)  -\gamma^{\prime
}\left(  \tau\right)  \right]  d\tau=\\
&  =y_{n}-\widetilde{y}_{n}+\int_{0}^{s}\left\vert h_{i_{n+1}}\right\vert
^{1/l_{n+1}}\sigma_{n+1}\left[  \left(  S_{r_{j}}^{x}\right)  _{\varphi\left(
\tau\right)  }-\left(  X_{r_{j}}\right)  _{\varphi\left(  \tau\right)
}\right]  d\tau+\\
&  +\int_{0}^{s}\left\vert h_{i_{n+1}}\right\vert ^{1/l_{n+1}}\sigma
_{n+1}\left[  \left(  X_{r_{j}}\right)  _{\varphi\left(  \tau\right)
}-\left(  X_{r_{j}}\right)  _{\gamma\left(  \tau\right)  }\right]  d\tau\\
&  =A+B+C.
\end{align*}
Now, by inductive assumption,%
\[
\left\vert A\right\vert \leq c\left\Vert h\right\Vert _{\eta}^{r}%
\]
while by (\ref{X_I-S_I})
\begin{align*}
\left\vert B\right\vert  &  \leq c\int_{0}^{s}\left\vert h_{i_{n+1}%
}\right\vert ^{1/l_{n+1}}\left\vert \varphi\left(  \tau\right)  -x\right\vert
^{r-1}d\tau\leq\\
&  \leq c\int_{0}^{s}\left\Vert h\right\Vert _{\eta}\left\Vert h\right\Vert
_{\eta}^{r-1}d\tau\leq c\left\Vert h\right\Vert _{\eta}^{r}%
\end{align*}
where we used the fact that%
\[
\left\vert \varphi\left(  \tau\right)  -x\right\vert \leq\left\vert
\varphi\left(  \tau\right)  -y_{n}\right\vert +\sum_{k=2}^{n}\left\vert
y_{k}-y_{k-1}\right\vert +\left\vert y_{1}-x\right\vert \leq c\sum_{k=1}%
^{n+1}\left\vert h_{i_{k}}\right\vert ^{1/l_{k}}\leq c\left\Vert h\right\Vert
_{\eta}.
\]
Finally,%
\[
\left\vert C\right\vert \leq c\int_{0}^{s}\left\vert h_{i_{n+1}}\right\vert
^{1/l_{n+1}}\left\vert \varphi\left(  \tau\right)  -\gamma\left(  \tau\right)
\right\vert d\tau\leq c\left\Vert h\right\Vert _{\eta}\int_{0}^{s}\left\vert
\varphi\left(  \tau\right)  -\gamma\left(  \tau\right)  \right\vert d\tau.
\]
Collecting the previous inequalities, Gronwall's Lemma implies%
\[
\left\vert \varphi\left(  s\right)  -\gamma\left(  s\right)  \right\vert \leq
c\left\Vert h\right\Vert _{\eta}^{r},
\]
which for $s=1$ gives the desired assertion. This ends the proof of
(\ref{brando4}) and hence of (\ref{box inclusion}).
\end{proof}

Next, we need the following:

\begin{lemma}
\label{Lemma Lip}For any $y\in\Omega_{\eta}$ and $h\in Q,$ the map%
\[
x\longmapsto E_{\eta}^{S^{x}}\left(  y,h\right)
\]
is Lipschitz continuous in $\Omega_{\eta},$ and its Jacobian satisfies:%
\[
\left\vert \frac{\partial E_{\eta}^{S^{x}}\left(  y,h\right)  }{\partial
x}\right\vert \leq c\left\vert h\right\vert ^{1/r}\text{ \ for a.e. }%
x\in\Omega_{\eta}.
\]
Also, for any $y\in\Omega_{\eta}$ the map $\left(  x,h\right)  \longmapsto
E_{\eta}^{S^{x}}\left(  y,h\right)  $ is continuous in $\Omega_{\eta}\times Q$.
\end{lemma}

\begin{proof}
Continuity with respect to $\left(  x,h\right)  ,$ as well as Lipschitz
continuity with respect to $x$ are immediate. Let us prove the bound on
derivatives. We have%
\[
\frac{\partial}{\partial x}\left[  E_{\eta}^{S^{x}}\left(  y,h\right)
\right]  =\frac{\partial}{\partial x}\left[  \left(  \prod\limits_{j=i}%
^{M}\exp\left(  \left\vert h_{k_{j}}\right\vert ^{1/l_{k_{j}}}\sigma
_{j}S_{r_{j}}^{x}\right)  \right)  \left(  y\right)  \right]
\]
with $\sigma_{j}=\pm1$. To fix ideas, let us compute the derivative of the
composition of two such terms:%
\begin{align*}
&  \frac{\partial}{\partial x}\left[  \exp\left(  \left\vert h_{k_{1}%
}\right\vert ^{1/l_{k_{1}}}\sigma_{1}S_{r_{1}}^{x}\right)  \exp\left(
\left\vert h_{k_{2}}\right\vert ^{1/l_{k_{2}}}\sigma_{2}S_{r_{2}}^{x}\right)
\left(  y\right)  \right] \\
&  =\frac{\partial}{\partial x}\left[  \exp\left(  \left\vert h_{k_{1}%
}\right\vert ^{1/l_{k_{1}}}\sigma_{1}S_{r_{1}}^{x}\right)  \left(  z\right)
\right]  _{z=\exp\left(  \left\vert h_{k_{2}}\right\vert ^{1/l_{k_{2}}}%
\sigma_{2}S_{r_{2}}^{x}\right)  \left(  y\right)  }+\\
&  +\frac{\partial}{\partial z}\left[  \exp\left(  \left\vert h_{k_{1}%
}\right\vert ^{1/l_{k_{1}}}\sigma_{1}S_{r_{1}}^{x}\right)  \left(  z\right)
\right]  _{z=\exp\left(  \left\vert h_{k_{2}}\right\vert ^{1/l_{k_{2}}}%
\sigma_{2}S_{r_{2}}^{x}\right)  \left(  y\right)  }\cdot\frac{\partial
}{\partial x}\left[  \exp\left(  \left\vert h_{k_{2}}\right\vert ^{1/l_{k_{2}%
}}\sigma_{2}S_{r_{2}}^{x}\right)  \left(  y\right)  \right] \\
&  \equiv A\left(  h,x\right)  +B\left(  h,x\right)  \cdot C\left(
h,x\right)  .
\end{align*}

Let us inspect the term $A\left(  h,x\right)  $. In order to compute%
\[
\frac{\partial}{\partial x}\left[  \exp\left(  \left\vert h_{k_{1}}\right\vert
^{1/l_{k_{1}}}\sigma_{1}S_{r_{1}}^{x}\right)  \left(  z\right)  \right]  ,
\]
let us write%
\[
\exp\left(  \left\vert h_{k_{1}}\right\vert ^{1/l_{k_{1}}}\sigma_{1}S_{r_{1}%
}^{x}\right)  \left(  z\right)  =\gamma\left(  x,1\right)
\]
with%
\[
\left\{
\begin{array}
[c]{l}%
\frac{d\gamma}{\partial\tau}\left(  x,\tau\right)  =\left\vert h_{k_{1}%
}\right\vert ^{1/l_{k_{1}}}\sigma_{1}\left(  S_{r_{1}}^{x}\right)
_{\gamma\left(  x,\tau\right)  }\\
\gamma\left(  x,0\right)  =z.
\end{array}
\right.
\]
Then%
\begin{align*}
\left\vert \gamma\left(  x+w,\tau\right)  -\gamma\left(  x,\tau\right)
\right\vert  &  \leq\left\vert h_{k_{1}}\right\vert ^{1/l_{k_{1}}}\int
_{0}^{\tau}\left\vert \left(  S_{r_{1}}^{x+w}\right)  _{\gamma\left(
x+w,s\right)  }-\left(  S_{r_{1}}^{x}\right)  _{\gamma\left(  x,s\right)
}\right\vert ds\\
&  \leq\left\vert h_{k_{1}}\right\vert ^{1/l_{k_{1}}}\int_{0}^{\tau}\left\vert
\left(  S_{r_{1}}^{x+w}\right)  _{\gamma\left(  x+w,s\right)  }-\left(
S_{r_{1}}^{x}\right)  _{\gamma\left(  x+w,s\right)  }\right\vert ds+\\
&  +\left\vert h_{k_{1}}\right\vert ^{1/l_{k_{1}}}\int_{0}^{\tau}\left\vert
\left(  S_{r_{1}}^{x}\right)  _{\gamma\left(  x+w,s\right)  }-\left(
S_{r_{1}}^{x}\right)  _{\gamma\left(  x,s\right)  }\right\vert ds\\
&  \equiv A_{1}+A_{2}.
\end{align*}
Now, since the coefficients of the vector field $S_{r_{1}}^{x}$ depend in a
Lipschitz continuous way on the point $x,$%
\[
\left\vert A_{1}\right\vert \leq c\left\vert h_{k_{1}}\right\vert
^{1/l_{k_{1}}}\left\vert w\right\vert \tau
\]
while%
\[
\left\vert A_{2}\right\vert \leq\left\vert h_{k_{1}}\right\vert ^{1/l_{k_{1}}%
}\int_{0}^{\tau}\left\vert \gamma\left(  x+w,s\right)  -\gamma\left(
x,s\right)  \right\vert ds
\]
and, by Gronwall's inequality,%
\[
\left\vert \gamma\left(  x+w,\tau\right)  -\gamma\left(  x,\tau\right)
\right\vert \leq c\left\vert h_{k_{1}}\right\vert ^{1/l_{k_{1}}}\left\vert
w\right\vert \tau,
\]
which for $\tau=1$ gives%
\[
\left\vert \exp\left(  \left\vert h_{k_{1}}\right\vert ^{1/l_{k_{1}}}%
\sigma_{1}S_{r_{1}}^{x+w}\right)  \left(  z\right)  -\exp\left(  \left\vert
h_{k_{1}}\right\vert ^{1/l_{k_{1}}}\sigma_{1}S_{r_{1}}^{x}\right)  \left(
z\right)  \right\vert \leq c\left\vert h_{k_{1}}\right\vert ^{1/l_{k_{1}}%
}\left\vert w\right\vert .
\]
This shows that $x\mapsto\exp\left(  \left\vert h_{k_{1}}\right\vert
^{1/l_{k_{1}}}\sigma_{1}S_{r_{1}}^{x}\right)  \left(  z\right)  $ is Lipschitz
continuous with Lipschitz constant $\leq c\left\vert h_{k_{1}}\right\vert
^{1/l_{k_{1}}}.$ Hence, the $L^{\infty}$ function
\[
x\mapsto\frac{\partial}{\partial x}\exp\left(  \left\vert h_{k_{1}}\right\vert
^{1/l_{k_{1}}}\sigma_{1}S_{r_{1}}^{x}\right)  \left(  z\right)
\]
has $L^{\infty}$ norm $\leq c\left\vert h_{k_{1}}\right\vert ^{1/l_{k_{1}}}.$

The term $C\left(  h,x\right)  $ is similar to $A\left(  h,x\right)  $. As to
the term%
\[
B\left(  h,x\right)  =\frac{\partial}{\partial z}\left[  \exp\left(
\left\vert h_{k_{1}}\right\vert ^{1/l_{k_{1}}}\sigma_{1}S_{r_{1}}^{x}\right)
\left(  z\right)  \right]  ,
\]
note that the exponential is taken with respect to a smooth vector field, and
the derivative of an exponential with respect to the initial condition $z$ is
bounded in terms of the derivatives of the coefficients of the vector field
$\left\vert h_{k_{1}}\right\vert ^{1/l_{k_{1}}}\sigma_{1}S_{r_{1}}^{x}$ which
is bounded by $c\left\vert h_{k_{1}}\right\vert ^{1/l_{k_{1}}},$ uniformly
with respect to $x$. By composition, we can conclude that%
\[
\left\vert \frac{\partial}{\partial x}\left[  E_{\eta}^{S^{x}}\left(
y,h\right)  \right]  \right\vert \leq c\left\vert h\right\vert ^{1/r}%
\]
for a.e. $x\in\Omega_{\eta},$ any $y\in\Omega_{\eta},h\in Q.$
\end{proof}

\begin{lemma}
\label{Lemma T(x,h)}Let%
\[
\Theta\left(  x,h\right)  =E_{\eta}^{X}\left(  x,\cdot\right)  ^{-1}E_{\eta
}^{S^{x}}\left(  x,h\right)  \text{ for }\left(  x,h\right)  \in\Omega_{\eta
}\times Q.
\]
Then, the mapping $\left(  x,h\right)  \mapsto\Theta\left(  x,h\right)  $ is
Lipschitz continuous in $\Omega_{\eta}\times Q$. Moreover, the Jacobian
$J_{\Theta\left(  x,h\right)  }$ of the map $x\mapsto\Theta\left(  x,h\right)
$ satisfies%
\[
\left\Vert J_{\Theta\left(  \cdot,h\right)  }\right\Vert _{L^{\infty}\left(
\Omega_{\eta}\right)  }\leq\omega\left(  h\right)  \text{ \ }\forall h\in Q,
\]
where $\omega\left(  h\right)  \rightarrow0$ as $h\rightarrow0$.
\end{lemma}

\begin{proof}
The function%
\[
E_{\eta}^{S^{x}}\left(  y,h\right)
\]
is $C^{1}$ with respect to $\left(  y,h\right)  ,$ locally uniformly with
respect to $x$ and, by Lemma \ref{Lemma Lip}, is Lipschitz continuous with
respect to $x;$ hence%
\[
\left(  x,h\right)  \mapsto E_{\eta}^{S^{x}}\left(  x,h\right)
\]
is Lipschitz continuous. Let us show that the function%
\[
\left(  x,y\right)  \mapsto E_{\eta}^{X}\left(  x,\cdot\right)  ^{-1}\left(
y\right)
\]
is $C^{1}$ in the joint variables; this will allow to conclude that $\left(
x,h\right)  \mapsto\Theta\left(  x,h\right)  $ is Lipschitz continuous in
$\Omega_{\eta}\times Q.$

Let us consider the function%
\[
G\left(  x,y,h\right)  =y-E_{\eta}^{X}\left(  x,h\right)  .
\]
Since the function $\left(  x,h\right)  \mapsto E_{\eta}^{X}\left(
x,h\right)  $ is $C^{1}$ in the joint variables, $G\left(  x,x,0\right)  =0$
and $\frac{\partial G}{\partial h}\left(  x,x,0\right)  $ has maximal rank, by
the implicit function theorem there exists a unique $C^{1}$ function
$h=F\left(  x,y\right)  $ such that
\[
G\left(  x,y,F\left(  x,y\right)  \right)  =0.
\]
This $F$ is exactly $\left(  x,y\right)  \mapsto E_{\eta}^{X}\left(
x,\cdot\right)  ^{-1}\left(  y\right)  ,$ so we are done.

We now want to prove the bound on the Jacobian of $\Theta$.

Since $\Theta\left(  x,0\right)  =0$ $\forall x,$ for $h=0$ the Jacobian of
$\Theta\left(  \cdot,0\right)  $ vanishes. We have therefore to prove that
$J_{\Theta\left(  x,h\right)  }$ is continuous at $h=0,$ uniformly with
respect to $x$. We have, with the obvious meaning of symbols,%
\begin{align*}
J_{\Theta\left(  x,h\right)  }  &  =\left(  J_{x\mapsto E_{\eta}^{X}\left(
x,\cdot\right)  ^{-1}\left(  y\right)  }\right)  _{y=E_{\eta}^{S^{x}}\left(
x,h\right)  }+\left(  J_{y\mapsto E_{\eta}^{X}\left(  x,\cdot\right)
^{-1}\left(  y\right)  }\right)  _{y=E_{\eta}^{S^{x}}\left(  x,h\right)
}\cdot J_{x\mapsto E_{\eta}^{S^{x}}\left(  x,h\right)  }\\
&  =\left(  J_{x\mapsto E_{\eta}^{X}\left(  x,\cdot\right)  ^{-1}\left(
y\right)  }\right)  _{y=E_{\eta}^{S^{x}}\left(  x,h\right)  }+\left(
J_{E_{\eta}^{X}\left(  x,\cdot\right)  }\right)  ^{-1}\left(  \Theta\left(
x,h\right)  \right)  \cdot J_{x\mapsto E_{\eta}^{S^{x}}\left(  x,h\right)  }.
\end{align*}
Differentiating with respect to $x$ the identity%
\[
F\left(  x,E_{\eta}^{X}\left(  x,h^{\prime}\right)  \right)  =h^{\prime}%
\]
we get%
\[
\frac{\partial}{\partial x}F\left(  x,E_{\eta}^{X}\left(  x,h^{\prime}\right)
\right)  +\frac{\partial}{\partial y}F\left(  x,E_{\eta}^{X}\left(
x,h^{\prime}\right)  \right)  \cdot J_{E_{\eta}^{X}\left(  \cdot,h^{\prime
}\right)  }=0
\]
that is%
\begin{align*}
\left(  \frac{\partial}{\partial x}F\left(  x,y\right)  \right)  _{/y=E_{\eta
}^{S^{x}}\left(  x,h\right)  }  &  =-\left(  \frac{\partial}{\partial
y}F\left(  x,y\right)  \right)  _{/y=E_{\eta}^{S^{x}}\left(  x,h\right)
}\cdot J_{x\mapsto E_{\eta}^{X}\left(  x,h^{\prime}\right)  }=\\
&  =-\left(  J_{E_{\eta}^{X}\left(  x,\cdot\right)  }\right)  ^{-1}\left(
h^{\prime}\right)  \cdot J_{x\mapsto E_{\eta}^{X}\left(  x,h^{\prime}\right)
}.
\end{align*}
Hence%
\begin{align*}
J_{\Theta\left(  \cdot,h\right)  }  &  =-\left(  J_{E_{\eta}^{X}\left(
x,\cdot\right)  }\right)  ^{-1}\left(  h^{\prime}\right)  \cdot J_{x\mapsto
E_{\eta}^{X}\left(  x,h^{\prime}\right)  }+\left(  J_{E_{\eta}^{X}\left(
x,\cdot\right)  }\right)  ^{-1}\left(  h^{\prime}\right)  \cdot J_{x\mapsto
E_{\eta}^{S^{x}}\left(  x,h\right)  }\\
&  \equiv A\left(  x,h\right)  +B\left(  x,h\right)
\end{align*}
Now, since $E_{\eta}^{X}\left(  x,h^{\prime}\right)  \ $is $C^{1}$ in the
joint variables, $E_{\eta}^{X}\left(  x,\cdot\right)  $ is a diffeomorphism,
and $h^{\prime}=\Theta\left(  x,h\right)  $ is Lipschitz continuous, we
conclude that $A\left(  x,h\right)  $ is $h$-continuous, uniformly with
respect to $x$.

The same is true for the term $\left(  J_{E_{\eta}^{X}\left(  x,\cdot\right)
}\right)  ^{-1}\left(  h^{\prime}\right)  $ appearing in $B\left(  x,h\right)
$. It remains to check that $h\mapsto J_{x\mapsto E_{\eta}^{S^{x}}\left(
x,h\right)  }$ is continuous \textit{at least at }$h=0$. This can be seen as
follows:%
\[
\frac{\partial}{\partial x}E_{\eta}^{S^{x}}\left(  x,h\right)  =\frac
{\partial}{\partial x}\left[  E_{\eta}^{S^{x}}\left(  y,h\right)  \right]
_{/y=x}+\frac{\partial}{\partial x}\left[  E_{\eta}^{S^{z}}\left(  x,h\right)
\right]  _{/z=x}.
\]
The first term is continuous at $h=0,$ by Lemma \ref{Lemma Lip}, while the
second term is continuous in $h,$ because for fixed $z$ the vector fields
$S^{z}$ are smooth, and $E_{\eta}^{S^{z}}\left(  x,h\right)  $ is a $C^{1}$
function in the joint variables $\left(  x,h\right)  .$ This ends the proof.
\end{proof}

\bigskip

We can come, at last, to the

\begin{proof}
[Proof of Proposition \ref{Prop X contr}]We have to prove that the map
$E_{\eta}^{S^{x}}\left(  x,h\right)  $ is $X$-controllable with hitting time
$T\leq\alpha\rho.$ Actually, it is enough to prove $X$-controllability with a
hitting time $T\leq c\rho$ (with $c$ possibly larger than $\alpha$), because
if condition (1)\ in the statement of Theorem \ref{Thm 2.1. LM} holds for some
$\alpha$ then it still holds for a larger one, and if we replace $\alpha$ with
a larger one we can still fulfil condition (iii) replacing $\beta$ with a
larger one. We have therefore to prove that there exists a function
$\gamma:\Omega_{\eta}\times Q\times\left[  0,T\right]  \rightarrow
\mathbb{R}^{p}$ satisfying conditions (C1)-(C2'), which we will recall here.

(C1) For any $\left(  x,h\right)  \in\Omega_{\eta}\times Q_{\eta}\left(
\rho\right)  ,$ $t\mapsto\gamma\left(  x,h,t\right)  $ is a $X$-subunit path
connecting $x$ and $E_{\eta}^{S^{x}}\left(  x,h\right)  ,$ that is%
\[
\left\{
\begin{array}
[c]{l}%
\frac{d}{dt}\gamma\left(  x,h,t\right)  =\sum a_{j}\left(  t\right)  \left(
X_{j}\right)  _{\gamma\left(  x,h,t\right)  }\text{ for suitable }a_{j}\text{
with }\sum\left\vert a_{j}\left(  t\right)  \right\vert ^{2}\leq1\\
\gamma\left(  x,h,0\right)  =x;\gamma\left(  x,h,T\left(  x,h\right)  \right)
=E_{\eta}^{S^{x}}\left(  x,h\right)
\end{array}
\right.
\]
for a suitable $T\left(  x,h\right)  \leq T.$

To prove this, let $y=E_{\eta}^{S^{x}}\left(  x,h\right)  $ for some $h\in
Q_{\eta}\left(  \rho\right)  .$ By Lemma \ref{Brandolemma} (and its proof),
\begin{align*}
y  &  =E_{\eta}^{X}\left(  x,h^{\prime}\right)  \text{ for some }h^{\prime}\in
Q_{\eta}\left(  c\rho\right)  ,\text{ namely}\\
h^{\prime}  &  =\Theta\left(  x,h\right)  \equiv E_{\eta}^{X}\left(
x,\cdot\right)  ^{-1}E_{\eta}^{S^{x}}\left(  x,h\right)  .
\end{align*}
(Note that this $\Theta\left(  x,h\right)  $ is the one studied in Lemma
\ref{Lemma T(x,h)}). We can build an admissible curve $t\mapsto\widetilde
{\gamma}\left(  x,h^{\prime},t\right)  $ connecting $x$ to $y=E_{\eta}%
^{X}\left(  x,h^{\prime}\right)  $ in time $T\leq c\left\Vert h^{\prime
}\right\Vert _{\eta}\leq c\rho$ moving along the curve which define the
quasiexponential map $E_{\eta}^{X}\left(  x,\cdot\right)  ,$ as suggested in
\cite[pp.336-7]{LM}: namely, if%
\[
E_{\eta}^{X}\left(  x,h^{\prime}\right)  =%
{\displaystyle\prod\limits_{j=1}^{M}}
\exp\left(  \left\vert h_{k_{j}}^{\prime}\right\vert ^{1/l_{k_{j}}}\sigma
_{j}X_{r_{j}}\right)  \left(  x\right)
\]
then one easily checks that any map%
\[
\left(  x,h^{\prime}\right)  \mapsto\exp\left(  \left\vert h_{k_{j}}^{\prime
}\right\vert ^{1/l_{k_{j}}}\sigma_{j}X_{r_{j}}\right)  \left(  x\right)
\]
is $X$-controllable with hitting time $\left\vert h_{k_{j}}^{\prime
}\right\vert ^{1/l_{k_{j}}};$ by composition, Lemma 4.2 in \cite{LM} implies
that $E_{\eta}^{X}\left(  x,h^{\prime}\right)  $ is $X$-controllable with
hitting time
\[
T\leq c\sup_{h\in Q}\left\Vert h\right\Vert _{\eta}\leq c\rho,
\]
and in particular there exists a curve $\widetilde{\gamma}$ with the
properties required by (C1). Next, we have to check:

(C2') For any $\left(  h,t\right)  \in Q\times\left[  0,T\right]  ,$
$x\mapsto\gamma\left(  x,h,t\right)  $ is a one-to-one bilipschitz map having
jacobian determinant bounded away from zero, i.e.%
\[
b\equiv\inf_{O\times Q\times\left[  0,T\right]  }\left\vert \frac
{\partial\gamma}{\partial x}\right\vert >0.
\]

Namely, we have to compute the $x$-derivative of the composed function%
\[
\gamma\left(  x,h,t\right)  =\widetilde{\gamma}\left(  x,\Theta\left(
x,h\right)  ,t\right)  ,
\]
that is%
\begin{equation}
\frac{\partial}{\partial x}\gamma\left(  x,h,t\right)  =\frac{\partial
\widetilde{\gamma}}{\partial x}\left(  x,\Theta\left(  x,h\right)  ,t\right)
+\frac{\partial\widetilde{\gamma}}{\partial h^{\prime}}\left(  x,\Theta\left(
x,h\right)  ,t\right)  \cdot\frac{\partial}{\partial x}\left[  E_{\eta}%
^{X}\left(  x,\cdot\right)  ^{-1}E_{\eta}^{S^{x}}\left(  x,h\right)  \right]
. \label{Jac b}%
\end{equation}

First, let us recall that $\Theta\left(  x,0\right)  =0$ and $x\mapsto
\widetilde{\gamma}\left(  x,h^{\prime},t\right)  $ has the smoothness of
$x\mapsto E_{\eta}^{X}\left(  x,h^{\prime}\right)  $ hence is $C^{r-1,1}$.
Moreover, $\widetilde{\gamma}\left(  x,0,t\right)  =x,$ hence%
\[
\frac{\partial\widetilde{\gamma}}{\partial x}\left(  x,0,t\right)  =I
\]
(identity matrix) and, by continuity, $\frac{\partial\widetilde{\gamma}%
}{\partial x}\left(  x,\Theta\left(  x,h\right)  ,t\right)  $ is close to the
identity matrix for $h$ small enough.

Second, we know that $\left(  x,h^{\prime}\right)  \mapsto\widetilde{\gamma
}\left(  x,h^{\prime},t\right)  $ has the smoothness of $\left(  x,h^{\prime
}\right)  \mapsto E_{\eta}^{X}\left(  x,h^{\prime}\right)  ,$ hence is
$C^{1}.$ Therefore%
\[
\frac{\partial\widetilde{\gamma}}{\partial h^{\prime}}\left(  x,\Theta\left(
x,h\right)  ,t\right)  \text{ is bounded.}%
\]

Finally, by Lemma \ref{Lemma T(x,h)}%
\[
\left\vert \frac{\partial}{\partial x}\left[  E_{\eta}^{X}\left(
x,\cdot\right)  ^{-1}E_{\eta}^{S^{x}}\left(  x,h\right)  \right]  \right\vert
\leq c\omega\left(  h\right)  \text{ }\rightarrow0\text{ as }h\rightarrow0.
\]
By (\ref{Jac b}), these facts imply that $\frac{\partial\gamma}{\partial x}$
is a small perturbation of the identity, for small $h.$

Summarizing, the situation is the following:%
\begin{align*}
&  \gamma\left(  x_{1},h,t\right)  -\gamma\left(  x_{2},h,t\right)  =\\
&  =\left[  \widetilde{\gamma}\left(  x_{1},\Theta\left(  x_{1},h\right)
,t\right)  -\widetilde{\gamma}\left(  x_{2},\Theta\left(  x_{1},h\right)
,t\right)  \right]  +\left[  \widetilde{\gamma}\left(  x_{2},\Theta\left(
x_{1},h\right)  ,t\right)  -\widetilde{\gamma}\left(  x_{2},\Theta\left(
x_{2},h\right)  ,t\right)  \right] \\
&  \equiv A+B.
\end{align*}
Since, for small $h^{\prime},$ the map $x\longmapsto\widetilde{\gamma}\left(
x,h^{\prime},t\right)  $ is a diffeomorphism,%
\[
\left\vert A\right\vert \geq c\left\vert x_{1}-x_{2}\right\vert .
\]
On the other hand, by Lemma \ref{Lemma T(x,h)},%
\[
\left\vert B\right\vert \leq\left\vert \Theta\left(  x_{1},h\right)
-\Theta\left(  x_{2},h\right)  \right\vert \leq c\omega\left(  h\right)
\left\vert x_{1}-x_{2}\right\vert
\]
Hence for $h$ small enough $x\longmapsto\gamma\left(  x,h,t\right)  $ is a
bilipschitz map, with Jacobian determinant bounded away from zero. Note that,
asking $h$ small enough amounts to diminishing the constant $r_{0}$ in Theorem
\ref{Thm 4.1 LM}, which is allowed, as we have already noted. This completes
the proof of Proposition \ref{Prop X contr}, and therefore of Theorem
\ref{Thm Poincare}.
\end{proof}

\begin{remark}
\label{Remark rough Poincare}If we assume that our vector fields $X_{i}$'s
only belong to $C^{r-1}\left(  \overline{\Omega}\right)  ,$ instead of
$C^{r-1,1}\left(  \Omega\right)  ,$ the theory developed in sections
\ref{section subelliptic metric}-\ref{section connectivity} allows to derive a
rougher version of Poincar\'{e}'s inequality. Let us sketch it here. For a
fixed point $x_{0}\in\Omega^{\prime},$ let us consider the smooth
approximating vector fields $S_{i}^{x_{0}}$. Since these are smooth
H\"{o}rmander's vector fields, they satisfy a Poincar\'{e}'s inequality%
\begin{equation}
\int_{B\times B}\left\vert u\left(  y\right)  -u\left(  x\right)  \right\vert
dydx\leq c\rho\left\vert B\right\vert \int_{\lambda B}\left\vert S^{x_{0}%
}u\left(  y\right)  \right\vert dy \label{rough Poincare}%
\end{equation}
where $B=B_{1}^{S^{x_{0}}}\left(  x_{0},\rho\right)  $; by Theorem
\ref{Thm equivalent balls 1}, a similar inequality also holds with
$B=B_{1}^{X}\left(  x_{0},\rho\right)  ,$ and possibly a larger number
$\lambda.$ Now, let us recall that by Proposition \ref{Proposition S_i}%
\[
S_{i}^{x_{0}}=X_{i}+\sum_{j=1}^{n}c_{ij}\left(  x\right)  \partial_{x_{j}%
}\text{ with }c_{ij}\left(  x\right)  =o\left(  \left\vert x-x_{0}\right\vert
^{r-1}\right)  \text{ as }x\rightarrow x_{0}.
\]
Hence (\ref{rough Poincare}) rewrites as%
\[
\int_{B\times B}\left\vert u\left(  y\right)  -u\left(  x\right)  \right\vert
dydx\leq c\rho\left\vert B\right\vert \int_{\lambda B}\left\vert Xu\left(
y\right)  \right\vert dy+\left\vert B\right\vert o\left(  \rho^{r}\right)
\int_{\lambda B}\left\vert \nabla u\left(  y\right)  \right\vert dy
\]
(where $\nabla$ is the Euclidean gradient), or%
\[
\int_{B}\left\vert u\left(  y\right)  -u_{B}\left(  x\right)  \right\vert
dydx\leq c\rho\int_{\lambda B}\left\vert Xu\left(  y\right)  \right\vert
dy+o\left(  \rho^{r}\right)  \int_{\lambda B}\left\vert \nabla u\left(
y\right)  \right\vert dy.
\]

\end{remark}

\section{Applications\label{section consequences}}

There is a large literature dealing with relations between Poincar\'{e}'s
inequality and other results about both Sobolev spaces and solutions to second
order PDEs, both in the Euclidean (elliptic) context and in the subelliptic
one. We refer to Hajlasz-Koskela's monograph \cite{HK} for a good exposition
and a rich source of further references on this area of research. Some of
these results have been established in great generality, as axiomatic
theories. For instance, it is well-known that, roughly speaking, the validity
of the doubling condition and a Poincar\'{e}'s inequality imply a Sobolev
embedding. This fact has been proved, at different levels of generality, by
Saloff-Coste \cite{Sa}, Garofalo-Nhieu \cite{GN1}, Franchi-Lu-Wheeden
\cite{FLW}, Hajlasz-Koskela \cite{HK}. In turn, the doubling condition,
Poincar\'{e} and Sobolev inequalities allow to reply Moser's iteration
technique, and prove a Harnack inequality and a H\"{o}lder continuity result
for local solutions to (elliptic or subelliptic) variational second order
equations. In this section we want to point out, for convenience of the
reader, some precise statements of this kind, which describe a few
consequences of the results we have proved so far, which can be easily derived
from the aforementioned general theories, and constitute new results, in our
general setting.

\subsection{Sobolev embedding and $p$-Poincar\'{e}'s inequality}

Here we keep Assumptions (D) stated at the beginning of
\S \ \ref{section poincare}. We start noting that (\ref{Poincare 2}) implies,
by H\"{o}lder's inequality,%
\begin{equation}
\frac{1}{\left\vert B\right\vert }\int_{B}\left\vert u\left(  y\right)
-u_{B}\right\vert dy\leq c\rho\left(  \frac{1}{\left\vert \lambda B\right\vert
}\int_{\lambda B}\left\vert Xu\left(  y\right)  \right\vert ^{p}dy\right)
^{1/p}\text{ for any }p>1. \label{Poincare 1-p}%
\end{equation}

Then, applying Theorem 13.1 in \cite{HK} we have the following strong result:

\begin{theorem}
\label{Thm Sobolev met Poincare}For any $\Omega^{\prime}\Subset\Omega,$
$p\geq1,$ there exist $c,r_{0}>0,$ such that:

(i) (Sobolev inequality) There exists a constant $k>1$ such that
\begin{equation}
\left(  \frac{1}{\left\vert B\right\vert }\int_{B}\left\vert \varphi\left(
x\right)  \right\vert ^{kp}dx\right)  ^{1/kp}\leq c\rho\left(  \frac
{1}{\left\vert B\right\vert }\int_{B}\left\vert X\varphi\left(  x\right)
\right\vert ^{p}dx\right)  ^{1/p} \label{Sobolev}%
\end{equation}
for any $\varphi\in C_{0}^{\infty}\left(  B\right)  ,$ with $B=B\left(
x,\rho\right)  ,\rho\leq r_{0}$, $x\in\Omega^{\prime}$, and the balls are
taken with respect to the distance $d_{X,1}.$

(ii) (Poincar\'{e}'s $p$-$p$ inequality)%
\begin{equation}
\left(  \frac{1}{\left\vert B\right\vert }\int_{B}\left\vert \varphi\left(
x\right)  -\varphi_{B}\right\vert ^{p}dx\right)  ^{1/p}\leq c\rho\left(
\frac{1}{\left\vert B\right\vert }\int_{B}\left\vert X\varphi\left(  x\right)
\right\vert ^{p}dx\right)  ^{1/p} \label{Poincare p-p}%
\end{equation}
for any $\varphi\in C^{\infty}\left(  B\right)  ,B$ as above.
\end{theorem}

Note that, quite surprisingly, in (\ref{Poincare p-p}) a ball of the
\textit{same radius} appears at both sides of the inequality; this fact,
instead, is natural in (\ref{Sobolev}), where the function $\varphi$ is
assumed compactly supported in $B$. This theorem is proved in \cite{HK}
exploiting a set of assumptions which, in our context of nonsmooth
H\"{o}rmander's vector fields, we have proved in the previous sections, namely:

(a) Poincar\'{e}'s inequality (Theorem \ref{Thm Poincare} and in particular
(\ref{Poincare 1-p}));

(b) the doubling condition for metric balls with respect to the distance
$d_{1}$(\ref{Thm equivalent distances d d1});

(c) the equivalence of the Euclidean topology with $d_{1}$-topology, which
follows from Proposition \ref{Prop Fefferman Phong d1}.

\subsection{Moser's iteration for variational second order operators}

Let us consider a linear second order variational operator of the kind%
\begin{equation}
Lu\equiv\sum_{i,j=1}^{n}X_{i}^{\ast}\left(  a_{ij}\left(  x\right)
X_{j}u\right)  \label{L}%
\end{equation}
where $X_{1},...,X_{n}$ is our set of nonsmooth H\"{o}rmander's vector fields,
$X_{i}^{\ast}$ denotes the transposed operator of $X_{i}$, and $\left\{
a_{ij}\right\}  _{i,j=1}^{n}$ is a symmetric uniformly positive definite
matrix of $L^{\infty}\left(  \Omega\right)  $ functions:%
\[
\lambda\left\vert \xi\right\vert ^{2}\leq\sum_{i,j=1}^{n}a_{ij}\left(
x\right)  \xi_{i}\xi_{j}\leq\lambda^{-1}\left\vert \xi\right\vert ^{2}%
\]
for some $\lambda>0,$ any $\xi\in\mathbb{R}^{n},$ a.e. $x\in\Omega$. We say
that $u$ is a local solution to the equation $Lu=0$ in $\Omega\ $if%
\[
u\in W_{X,loc}^{1,2}\left(  \Omega\right)  =\left\{  u\in L_{loc}^{2}\left(
\Omega\right)  :X_{i}u\in L_{loc}^{2}\left(  \Omega\right)  \text{ for
}i=1,2,...,n\right\}
\]
and%
\[
\int_{\Omega}\sum_{i,j=1}^{n}a_{ij}X_{i}uX_{j}\varphi dx=0\text{ for any
}\varphi\in C_{0}^{\infty}\left(  \Omega\right)  .
\]

In this context, Theorem \ref{Thm Sobolev met Poincare} (with $p=2$) gives the
tools to settle the classical Moser's iterative method, and prove the facts
collected in the following:

\begin{theorem}
\label{Thm Moser}Let $u$ be a local solution to $Lu=0$ in $\Omega.$ Then:

(i) $u$ is locally bounded, with%
\[
\left\Vert u\right\Vert _{L^{\infty}\left(  B\right)  }\leq c\left(  \frac
{1}{\left\vert 2B\right\vert }\int_{2B}\left\vert u\left(  x\right)
\right\vert ^{2}dx\right)  ^{1/2}%
\]
for any $2B\subset\Omega,$ with $c$ depending on the coefficients $a_{ij}$
only through the number $\lambda$.

(ii) If $u$ is positive in $\Omega,$ then it satisfies a Harnack's inequality:%
\[
\sup_{B}u\leq c\inf_{B}u
\]
for any $2B\subset\Omega,$ with $c$ depending on the coefficients $a_{ij}$
only through the number $\lambda$.

(iii) $u$ is H\"{o}lder continuos (in the usual, Euclidean sense) of some
exponent $\alpha\in\left(  0,1\right)  ,$ on any subset $\Omega^{\prime
}\Subset\Omega$:%
\[
\left\vert u\left(  x\right)  -u\left(  y\right)  \right\vert \leq c\left\vert
x-y\right\vert ^{\alpha}%
\]
for any $x,y\in\Omega^{\prime},$ with $c,\alpha$ depending on $\Omega^{\prime
}$ and depending on the coefficients $a_{ij}$ only through the number
$\lambda,$ and $c$ also depending on $\left\Vert u\right\Vert _{L^{2}\left(
\Omega\right)  }$.
\end{theorem}

The above theorem follows, for instance, applying the general theory developed
by Sawyer-Wheeden in \cite{SW} (see in particular Theorem 8 in \cite{SW}). To
check the assumptions of this theory one needs to exploit Theorem
\ref{Thm Sobolev met Poincare} and the facts (b), (c) recalled in the previous
subsection. Actually, the H\"{o}lder continuity result which follows from
Theorem 8 in \cite{SW} is much more general than the one we have stated in
(iii): it holds for local solutions to a nonhomogeneous equation, also
involving lower order terms. With the terminology introduced in \cite{SW}, one
can say that the operator $L$ in (\ref{L}) is $L^{q}$-\textit{subelliptic}. We
do not state this result in its full generality for the sake of simplicity.

Clearly, the local H\"{o}lder continuity result can be applied also to local
solutions for nonlinear operators of the kind:%
\[
Pu=\sum_{i=1}^{n}X_{i}^{\ast}\left(  a_{ij}\left(  x,u\left(  x\right)
\right)  X_{j}u\right)  .
\]

For smooth H\"{o}rmander's vector fields, the results contained in Theorem
\ref{Thm Moser} follow from Nagel-Stein-Wainger's doubling condition and
Jerison's Poincar\'{e} inequality (see \cite{NSW}, \cite{J}). Analogous
results, in a weighted context, have been proved by Lu in \cite{Lu}.

We also point out that operators (\ref{L}) structured on nonsmooth
H\"{o}rmander's vector fields can be seen also as particular instances of
$X$-elliptic operators, in the sense of Lanconelli-Kogoj \cite{LK}, with the
same consequences already described.

\section{Appendix: some known results about O.D.E.'s}

\subsubsection*{1. Gronwall's Lemma}

We state the version of Gronwall's Lemma that we use throughout this paper.
For a proof, see for instance \cite[p.625]{E}.

\begin{lemma}
\label{Lemma Gronwall}Let $\phi:\left[  0,1\right]  \rightarrow\mathbb{R}$ be
a nonnegative continuous function such that%
\begin{equation}
\phi\left(  t\right)  \leq c\int_{0}^{t}\phi\left(  s\right)  ds+K
\label{Hp Gronwall}%
\end{equation}
for any $t\in\left[  0,1\right]  $ and two positive constants $c,K.$ Then
there exists a constant $c_{1}>0,$ only depending on $c,$ such that%
\[
\phi\left(  t\right)  \leq c_{1}K.
\]

\end{lemma}

\bigskip

\subsubsection*{2. Existence results and uniformity matters}

\begin{theorem}
[Carath\'{e}odory's existence theorem]Let $F\left(  t,x\right)  $ be a
function defined for $t\in\left(  -T,T\right)  ,x\in\mathbb{R}^{p},$ $F$
continuous in $x$ for fixed $t$ and measurable in $t$ for fixed $x.$ Assume
that%
\[
\left\vert F\left(  t,x\right)  \right\vert \leq M\left(  t\right)
\]
with $M\in L^{1}\left(  a,b\right)  $ for any $\left[  a,b\right]
\subset\left(  -T,T\right)  .$ Then, for every $x_{0}\in$ $\mathbb{R}^{p}$
there exists an absolutely continuous function $\phi:\left(  -T,T\right)
\rightarrow\mathbb{R}^{p}$ solution to the problem
\[
\left\{
\begin{array}
[c]{l}%
\phi^{\prime}\left(  t\right)  =F\left(  t,\phi\left(  t\right)  \right)
\text{ for a.e. }t\in\left(  -T,T\right) \\
\phi\left(  0\right)  =x_{0}%
\end{array}
\right.
\]
exists.
\end{theorem}

For the proof, see e.g. \cite[p.140]{San}. In the proof of Theorem
\ref{Thm dX equiv dS} we apply this theorem to%
\[
F\left(  t,x\right)  =\sum_{\left\vert I\right\vert \leq r}a_{I}\left(
t\right)  \left(  X_{\left[  I\right]  }\right)  _{x}%
\]
where $a_{I}\left(  \cdot\right)  $ are bounded measurable functions on
$\left[  0,1\right]  $, and the vector fields $X_{i}$ are $C^{r-1}\left(
\overline{\Omega}\right)  $. Now, for any fixed $\Omega^{\prime}\Subset
\Omega^{\prime\prime}\Subset\Omega$, we can find a function $\widetilde
{F}\left(  t,x\right)  $ satisfying the assumptions of the Carath\'{e}odory's
theorem and agreeeing with $F\left(  t,x\right)  $ for $x\in\Omega
^{\prime\prime}.$ When $x_{0}\in\Omega^{\prime}$ and $\left\vert a_{I}\left(
t\right)  \right\vert \leq\delta^{\left\vert I\right\vert }$ with $\delta$
small enough, there exists a solution $\phi$ to%
\[
\left\{
\begin{array}
[c]{l}%
\phi^{\prime}\left(  t\right)  =\widetilde{F}\left(  t,\phi\left(  t\right)
\right)  \text{ for }t\in\left[  0,1\right] \\
\phi\left(  0\right)  =x_{0}%
\end{array}
\right.
\]
such that $\phi\left(  t\right)  \in\Omega^{\prime\prime}$ for $t\in\left[
0,1\right]  $ and therefore $\phi$ solves%
\[
\left\{
\begin{array}
[c]{l}%
\phi^{\prime}\left(  t\right)  =\sum_{\left\vert I\right\vert \leq r}%
a_{I}\left(  t\right)  \left(  X_{\left[  I\right]  }\right)  _{\phi\left(
t\right)  }\text{ for }t\in\left[  0,1\right] \\
\phi\left(  0\right)  =x_{0}.
\end{array}
\right.
\]
Note that this $\delta$ depends on $\Omega$ and $\Omega^{\prime},$ but not on
$x_{0}$.

\begin{theorem}
[Cauchy's existence and uniqueness theorem]Let $X$ be a Lipschitz continuous
vector field defined in some domain $\Omega\subset\mathbb{R}^{p},$ and
$\Omega^{\prime}\Subset\Omega.$ There exists a number $\delta>0,$ depending on
$X,\Omega,\Omega^{\prime}$, such that for every $x_{0}\in\Omega^{\prime},$ a
unique $C^{1}$ solution $\phi:\left[  -\delta,\delta\right]  \rightarrow
\Omega$ to the problem
\begin{equation}
\left\{
\begin{array}
[c]{l}%
\phi^{\prime}\left(  t\right)  =X_{\phi\left(  t\right)  }\text{ for }%
t\in\left[  -\delta,\delta\right] \\
\phi\left(  0\right)  =x_{0}%
\end{array}
\right.  \label{Cauchy problem}%
\end{equation}
exists.
\end{theorem}

For the proof, see \cite{P}. We stress the fact that the number $\delta$ can
be chosen independently of $x_{0},$ at least when $x_{0}$ ranges in a compact
subset of $\Omega$. This uniformity property has been implicitly used in this paper.

\bigskip

\subsubsection*{3. Discussion about the dependence of the constants on the
smooth vector fields in the results proved by Nagel-Stein-Wainger \cite{NSW}}

Here we want to justify Claim \ref{Claim Jerison} stated in
\S \ref{section approximating balls}. We have checked in detail this Claim,
revising the whole argument of \cite{NSW}. Here we cannot repeat the whole
reasoning, but limit ourself to some remarks which stress the points to be
kept in mind, in order to understand the quantitative dependence of the
constants. What follows is intended to be read keeping at hand the paper
\cite{NSW}: we will use their notations without any explanation.

1. We apply the construction of \cite{NSW}, Chapter II, \S 1, assuming that
the vector fields $Y_{i}$ are \textit{all }the commutators $X_{\left[
I\right]  }$ of our smooth vector fields $X_{0},X_{1},...,X_{n},$ with
$\left\vert I\right\vert \leq m.$ If $Y_{i}=X_{\left[  I\right]  },$ we will
set $d_{i}=\left\vert I\right\vert .$ It is not difficult to see that, by the
Jacobi identity, for any multiindices $I,J$ we can write%
\[
\left[  X_{\left[  I\right]  },X_{\left[  J\right]  }\right]  =\sum
_{\left\vert K\right\vert =\left\vert I\right\vert +\left\vert J\right\vert
}b_{IJ}^{K}X_{\left[  K\right]  }%
\]
where $b_{IJ}^{K}$ are universal constants only depending on $I,J,K$ (this
fact is stated for instance in \cite{HM})$.$ Therefore we can write equation
(1) of \cite{NSW} as%
\[
\left[  Y_{j},Y_{k}\right]  =\sum_{d_{l}\leq d_{j}+d_{k}}c_{jk}^{l}\left(
x\right)  Y_{l}%
\]
where the \textquotedblleft functions\textquotedblright\ $c_{jk}^{l}\left(
x\right)  $ are actually \textit{universal constants}.

2. Let us call \textquotedblleft admissible function\textquotedblright\ any
function which can be obtained, starting from the coefficients of the smooth
vector fields $X_{i},$ by linear combination and a finite number of operations
of sums, products, and derivatives; moreover, it is allowed to divide by the
quantity%
\[
\det\left(  Y_{i_{1}},Y_{i_{2}},...,Y_{i_{N}}\right)
\]
where $Y_{i_{1}},Y_{i_{2}},...,Y_{i_{N}}$ is a fixed basis in an open subset
$\Omega_{I}\subset\Omega$. Clearly, admissible functions belong to $C^{\infty
}\left(  \Omega_{I}\right)  .$

Let $a_{j}^{l}\left(  x\right)  $ have the meaning explained in \cite[p.116]%
{NSW}. A key role is played in \cite{NSW} by the modules of functions
$A_{s}^{p},$ defined as the $C^{\infty}\left(  \Omega\right)  $ submodule of
$C^{\infty}\left(  \Omega_{I}\right)  $ generated by all the functions of the
form
\[
a_{j_{1}}^{l_{1}}\cdot a_{j_{2}}^{l_{2}}\cdot...a_{j_{k}}^{l_{k}}%
\]
where the indices satisfy suitable conditions. Now, we claim that, keeping in
mind our remark 1 and revising the whole reasoning of Chapter II, \S 1 in
\cite{NSW}, one can check that all the arguments and statements of that
section remain true if we redefine the classes of functions $A_{s}^{p}$ as the
modules generated by the functions $a_{j_{1}}^{l_{1}}\cdot a_{j_{2}}^{l_{2}%
}\cdot...a_{j_{k}}^{l_{k}}$ taking as \textquotedblleft
scalars\textquotedblright\ not all the functions in $C^{\infty}\left(
\Omega\right)  ,$ but only \textit{admissible functions}.

This fact is crucial because, whenever we prove that a function belongs to a
class $A_{s}^{p},$ this implies a quantitative estimate in terms of the
quantities allowed by our Claim \ref{Claim Jerison}.

3. Revising the whole reasoning of the following sections of Chapter II in
\cite{NSW}, then, one can check that most of the arguments do not involve new
forms of dependence of the relevant constants on the vector fields $X_{i}$.
The points that require a more careful inspection are those involving the
Baker-Campbell-Hausdorff formula (henceforth, BCH formula), since this
identity, in principle, involves infinitely many derivatives. So, our next
remark is devoted to BCH formula.

4. We need the following finite BCH formula with a remainder:

for any $\Omega^{\prime}\Subset\Omega,$ given two positive integers
$k_{0},j_{0}$ there exist $r_{0}>0$ and $C>0$ such that, if $\left\vert
s\right\vert ,\left\vert t\right\vert <r_{0}$ then%
\begin{equation}
\exp\left(  sX\right)  \exp\left(  tY\right)  \left(  x\right)  =\exp\left(
\sum_{k+j\geq1,k\leq k_{0},j\leq j_{0}}s^{k}t^{j}C_{k,j}\right)  \left(
x\right)  +O\left(  s^{k_{0}+1}\right)  +O\left(  t^{j_{0}+1}\right)
\label{FBCH}%
\end{equation}
for any $x\in\Omega^{\prime}$, where:

(i) $C_{k,j}$ denotes a finite linear combination of commutators of $X,Y$,
with universal coefficients, where every commutator contains $k$ times $X$ and
$j$ times $Y$;

(ii) the remainders satisfy the estimates%

\[
\left\vert O\left(  s^{k_{0}+1}\right)  \right\vert \leq Cs^{k_{0}%
+1},\left\vert O\left(  t^{j_{0}+1}\right)  \right\vert \leq Ct^{j_{0}+1}%
\]
\textit{where the constants }$r_{0},C$\textit{ only depend on a finite number
of }$C^{k}\left(  \Omega^{\prime}\right)  $\textit{ norms of the coefficients
of }$X,Y$.

Although the above fact is probably well known, we have not been able to find
a precise reference for the last statement about the dependence of the
constants $r_{0}$ and $C.$ However, revising the proof of this formula given
for instance in \cite{CNSW}, one can check that this is actually the case.

The identity (\ref{FBCH}) is applied several times in \S \S 3-5 of Chapter II
of \cite{NSW}, taking as $X,Y$ suitable commutators of our vector fields;
thanks to the above remark, the dependence of the constants satisfies also in
this case the desired control.

\bigskip

\noindent\textsc{Dipartimento di Matematica}

\noindent\textsc{Politecnico di Milano}

\noindent\textsc{Via Bonardi 9, 20133 Milano, ITALY}

\noindent\texttt{marco.bramanti@polimi.it}

\bigskip

\noindent\textsc{Dipartimento di Ingegneria dell'Informazione e Metodi
Matematici}

\noindent\textsc{Universit\`{a} di Bergamo}

\noindent\textsc{Viale Marconi 5, 24044 Dalmine BG, ITALY}

\noindent\texttt{luca.brandolini@unibg.it}

\bigskip

\noindent\textsc{Dipartimento di Ingegneria dell'Informazione e Metodi
Matematici}

\noindent\textsc{Universit\`{a} di Bergamo}

\noindent\textsc{Viale Marconi 5, 24044 Dalmine BG, ITALY}

\noindent\texttt{marco.pedroni@unibg.it}

\end{document}